\documentclass[reqno,tbtags,a4paper,12pt,oneside]{amsart} 
\usepackage[T2A]{fontenc}
\usepackage[cp1251]{inputenc}
\usepackage[english]{babel}
\usepackage{graphics}

\usepackage[tbtags]{amsmath}
\usepackage{amsfonts,amssymb,mathrsfs,amscd,comment}

\usepackage{graphicx,euscript,amsthm}
\usepackage[all]{xy}
\usepackage[in]{fullpage}
\usepackage{pstricks,pst-grad,pst-node}

\theoremstyle{plain}
\newtheorem{theorem}{Theorem}[section]
\newtheorem{lemma}[theorem]{Lemma}
\newtheorem{utv}[theorem]{Proposition}
\newtheorem{cor}[theorem]{Corollary}
\theoremstyle{definition}
\newtheorem{defin}[theorem]{Definition}
\newtheorem{constr}[theorem]{Construction}
\newtheorem{exam}[theorem]{Example}
\theoremstyle{remark}
\newtheorem{rem}[theorem]{Remark}

\newcommand{\aff}{\mathrm{aff}\,}

\renewcommand{\a}{\alpha}
\renewcommand{\b}{\beta}

\renewcommand{\int}{\mathrm{int}}

\renewcommand{\epsilon}{\varepsilon}
\renewcommand{\vert}{\mathrm{vert }}

\newcommand{\ib}[1]{\boldsymbol{#1}}
\newcommand{\m}[1]{\mathcal{#1}}

\newcommand{\F}{\mathcal{F}}

\renewcommand{\P}{\mathcal{P}}

\renewcommand{\L}{\mathcal{L}}
\newcommand{\B}{\mathcal{B}}
\newcommand{\W}{\mathcal{W}}
\newcommand{\D}{\mathcal{D}}
\renewcommand{\C}{\mathcal{C}}

\renewcommand{\phi}{\varphi}

\begin{document}
\title[Construction of fullerenes]{Construction of fullerenes}
\author{V.M.Buchstaber,~N.Yu.Erokhovets}
\thanks{The work was supported by the RSCF grant No 14-11-00414}

\begin{abstract} We present an infinite series of operations on fullerenes generalizing the Endo-Kroto operation,
such that each combinatorial fullerene is obtained from the dodecahedron by a
sequence of such operations. We prove that these operations are invertible in
the proper sense, and are compositions of $(1;4,5)-$, $(1;5,5)$-,
$(2,6;4,5)$-, $(2,6;5,5)$-, $(2,6;5,6)$-, $(2,7;5,5)$-, and
$(2,7;5,6)$-truncations, where each truncation increases the number of
hexagons by one.
\end{abstract}
\maketitle
\section{Introduction}
The study of fullerenes is stimulated by problems of the quantum physics, the
quantum chemistry and the nanotechnology \cite{LMR06}. For an introduction to
the mathematical study of fullerenes we refer to \cite{DSS13}. The effective
algorithm of enumeration of combinatorial types of fullerenes was described
in \cite{BD97}.

The Endo-Kroto operation (Fig. \ref{EK-O}, see \cite{EK92}) is an operation
on fullerenes that increases the number $p_6$ of hexagons by one. It allows
to obtain a fullerene with arbitrary number $p_6\geqslant 2$ from the Barrel
fullerene ($p_6=2$). For details see, for example, \cite{DSS13}.
\begin{figure}[h]
\begin{center}
\includegraphics[scale=0.3]{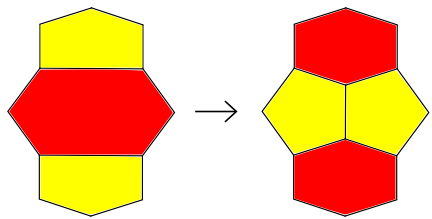}
\end{center}
\caption{The Endo-Kroto operation}\label{EK-O}
\end{figure}

The generalizations of the Endo-Kroto operation were studied in
\cite{BFFG06,GG10,GGG14}. The idea is to transform one fullerene to other by
replacement of one fragment by another fragment with the same boundary. If
the operation increases the number of facets, then this is named a
\emph{growth-operation}. If the operation does not change the number of
facets, then it is named an \emph{isomerization-operation}.  In \cite{BFFG06}
it is mentioned that
\begin{enumerate}
\item an infinite number of growth operations would be needed to obtain
    all fullerenes;
\item it is an open question whether a combination of a finite number of
    isomerization and growth steps can access all fullerenes. This
    question could be answered in the negative if it could be shown that
    isomerization fragments with less than two pentagons cannot occur as
    parts of fullerenes.
\end{enumerate}
In \cite{GG10} it was proved that isomerization patches with less than two
pentagons, if they exist, must be relatively large.

In this paper we present an infinite number of grows-operations that produce
all combinatorial fullerenes from the dodecahedron. Most of our operations
one can find in~\cite{GG10}.

The paper has the following structure.

In Section \ref{Def} we give necessary information about polytopes.

In Section \ref{Belts}  we develop the technique of $k$-belts.

In Section \ref{Trunc} we present the technique of $(s,k)$-truncations with
inverse operations of straightenings along edges \cite{BE15}.

In Section \ref{Main-Th} we prove the main result by the following scheme.
\begin{itemize}
\item First we prove that any combinatorial fullerene except for the
    dodecahedron contains one of the fragments from the given infinite
    series.
\item Then we prove that any fragment from the given series can be moved
    to the other fragment by a sequence of finite number of
    straightenings along edges, which are well defined.
\end{itemize}
\section{Definitions}\label{Def}
For an introduction to the polytope theory we recommend the books
\cite{Gb03,Z07}.

\begin{defin}
A \emph{convex polytope} $P$ is a bounded intersections of finite number of
halfspaces in $\mathbb R^n$. Equivalently, it is a convex hull of a finite
set of points in $\mathbb R^n$.

Denote by $\m{F}=\{F_1,\dots,F_m\}$ the set of all facets of $P$. Denote by
$s_i$ the number of edges of $F_i$. By definition $\dim P=\dim \aff(P)$.

For any collection $S$ of facets set $|S|=\bigcup_{F_i\in S}F_{i}$. Then
$|\F|=\partial P$ is the boundary of~$P$.

In what follows by a \emph{polytope} we mean a convex polytope. Polytopes of
dimension $n$ we call \emph{$n$-polytopes}.

An $n$-polytope is called \emph{simple} if any it's vertex belongs to exactly
$n$ facets.

For a polytope $P$ denote by $f_i(P)$ be the number of $i$-dimensional faces
and by $p_k(P)$ the number of $k$-gonal $2$-faces.
\end{defin}

\begin{utv}[The Euler formula]
For a $3$-polytope we have
$$
f_0-f_1+f_2=2
$$
\end{utv}

The following fact is a corollary of the Euler formula
\begin{utv}For a simple $3$-polytope $P$ we have
\begin{equation}\label{pk}
3p_3+2p_4+p_5=12+\sum\limits_{k\geqslant 7}(k-6)p_k.
\end{equation}
\end{utv}

\begin{defin}
A \emph{combinatorial polytope} is an equivalence class of
combinatorially equivalent convex polytopes, where two polytopes $P$ and $Q$
are \emph{combinatorially equivalent} (we denote it $P\simeq Q$), if there is
an inclusion-preserving bijection of the sets of their faces.
\end{defin}

\begin{defin}
A (mathematical) \emph{fullerene} is a simple $3$-polytope with all facets
pentagons and hexagons.

A fullerene is called an \emph{$IPR$-fullerene}, if it has no incident
pentagons.
\end{defin}
The formula (\ref{pk}) implies the following.
\begin{cor}
For a fullerene $P$ we have $p_5=12$.
\end{cor}

\begin{theorem}[\cite{Eb1891,Br1900}, see also \cite{Gb03}]
Any simple $3$-polytope is combinatorially equivalent to a polytope that is
obtained from the tetrahedron by a sequence of \emph{vertex, edge} and
\emph{$(2,k)$-truncations}.
\end{theorem}
\begin{figure}[h]
\begin{center}
\includegraphics[scale=0.13]{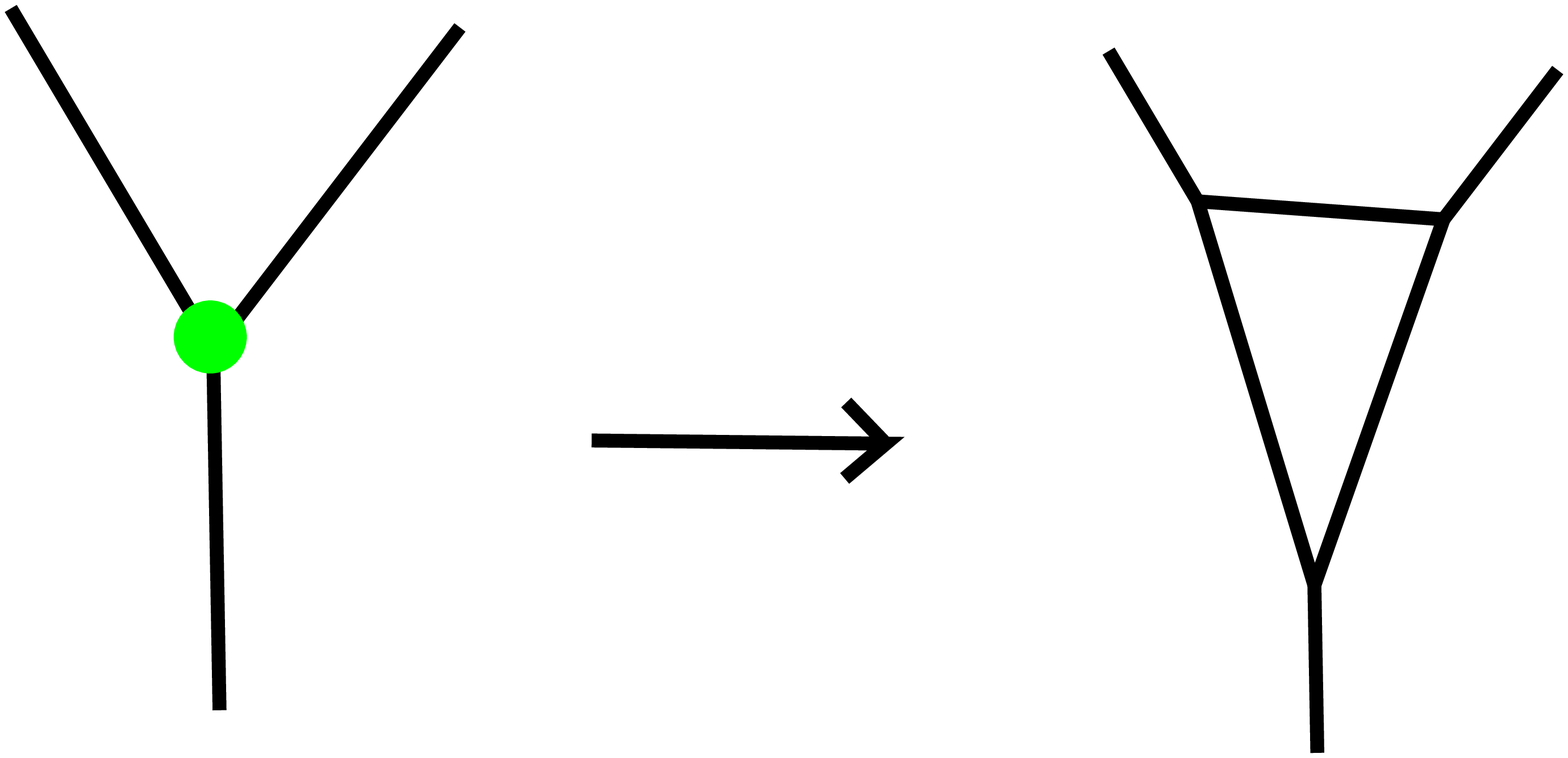}\qquad\includegraphics[scale=0.13]{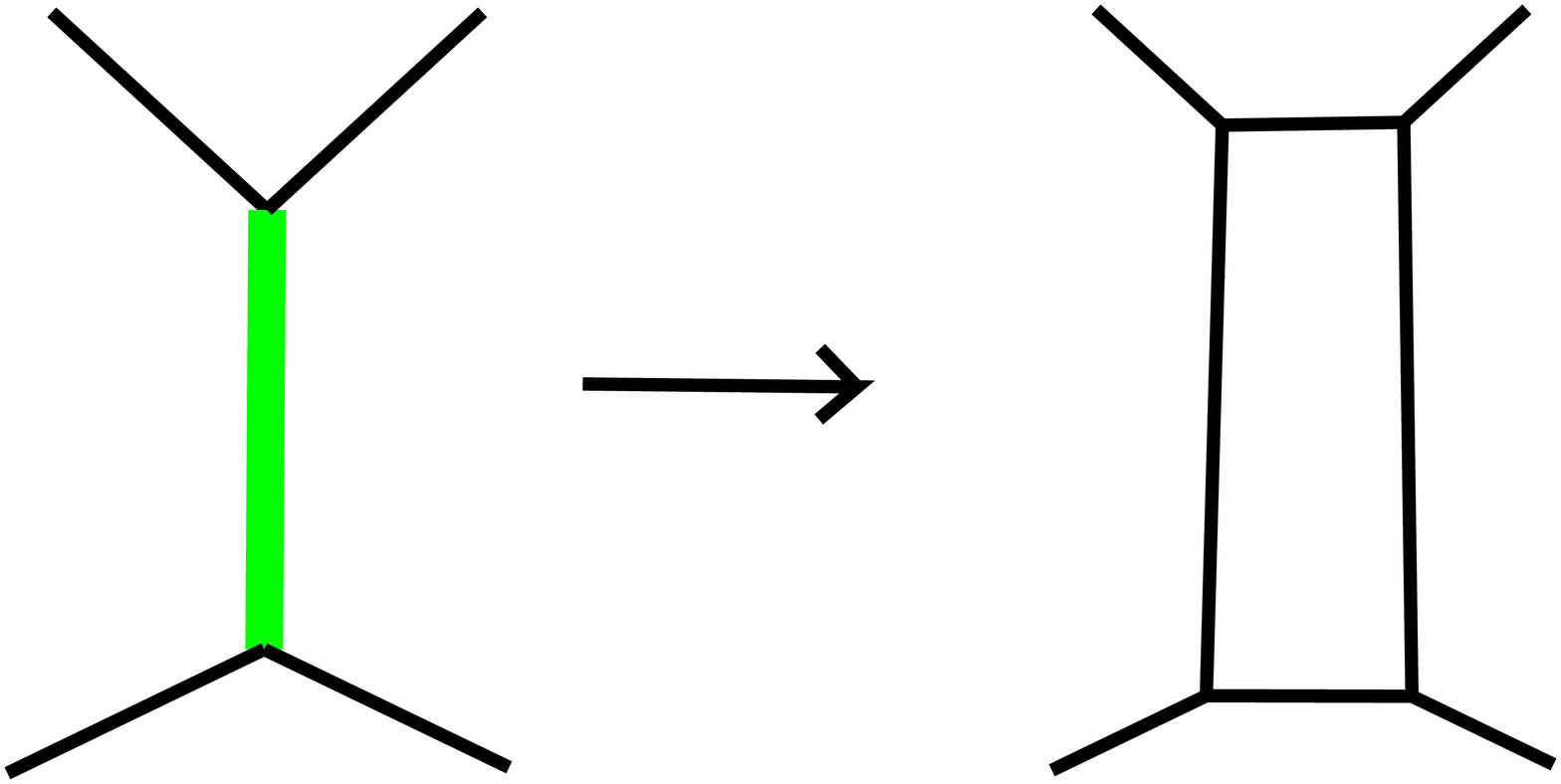}\qquad\includegraphics[scale=0.13]{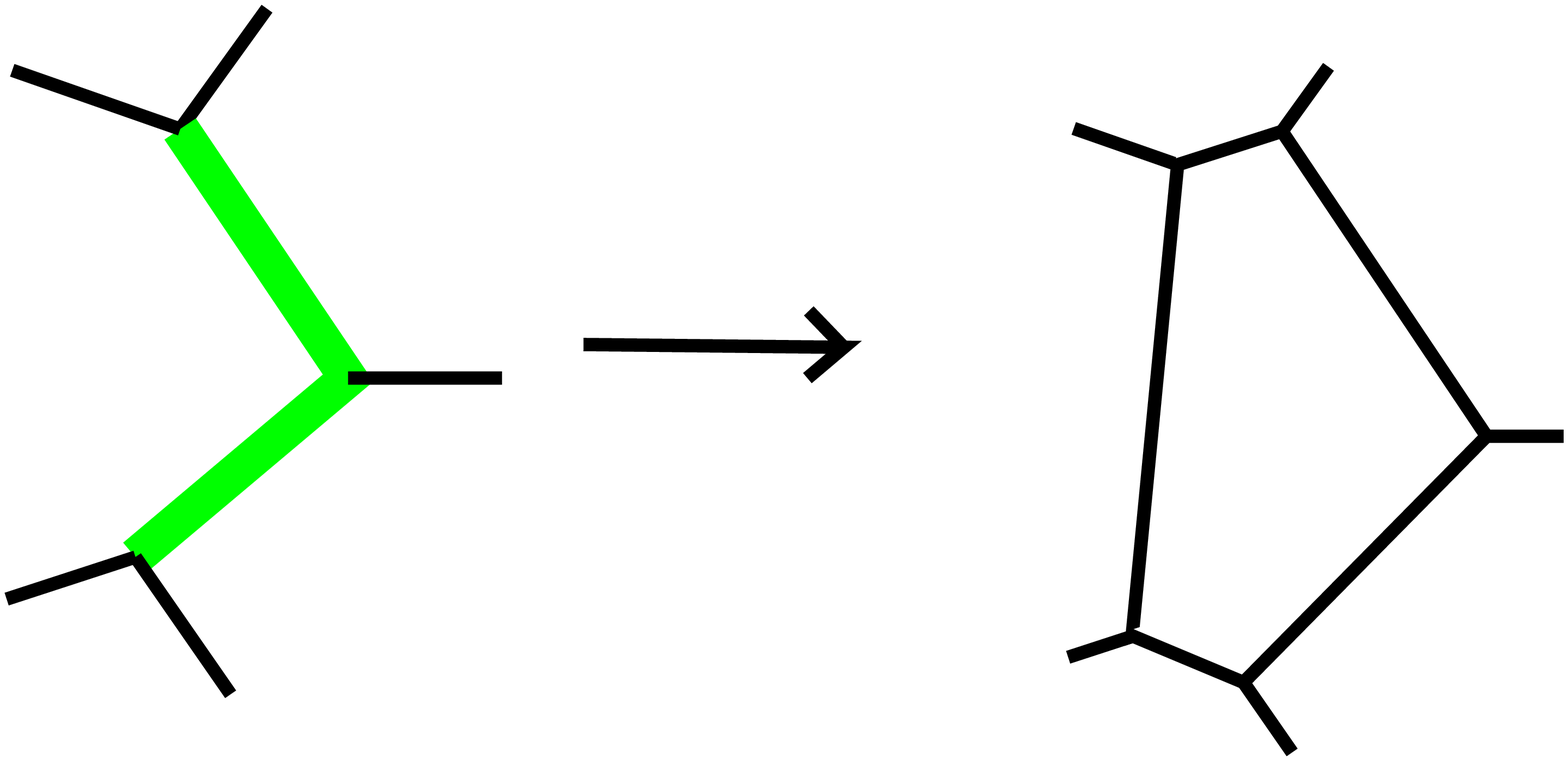}
\end{center}
\caption{Truncations of polytopes}\label{3cut}
\end{figure}

Recall some terminology from the graph theory.
\begin{defin}
A graph is called \emph{simple} if it has no loops and multiple edges.

A \emph{path} in a graph $\Gamma$ is a sequence of vertices $(\ib{v}_1,\dots,
\ib{v}_l)$ with $\ib{v}_i$ and $\ib{v}_{i+1}$ being connected by an edge for
$i=1,\dots,l-1$.

A \emph{cycle} is a closed path $(\ib{v}_1,\dots, \ib{v}_{l+1})$, where
$\ib{v}_{l+1}=\ib{v}_1$. We will denote a cycle by $(\ib{v}_1,\dots,
\ib{v}_l)$. A cycle is called \emph{simple} if it passes any vertex at most
once. A cycle is called \emph{non-chordal}, if it has no \emph{chords} --
edges of the graph, connecting non-successive vertices of a cycle.

A connected graph $\Gamma$ with at least $4$ edges is called
\emph{$3$-connected} if deletion of any one or any two vertices with all
incident edges gives a connected graph.

\emph{Graph $G(P)$} of a polytope $P$ is a graph of vertices and edges of
$P$. Paths and cycles in $G(P)$ we call \emph{edge-paths} and
\emph{edge-cycles}.
\end{defin}

Our main constructions are based on the following result.
\begin{theorem}[Steinitz, see \cite{Z07}]
A graph $\Gamma$ is a graph $G(P)$ of some $3$-polytope $P$ if and only if it
is simple, planar and $3$-connected.
\end{theorem}

We will implicitly use the following version of the Jordan curve theorem,
which can be proved rather directly.

\begin{theorem}\label{Jordan-Th}
Let $P$ be a simple $3$-polytope, and $\gamma$  be a simple edge-cycle. Then
\begin{enumerate}
\item $\partial P\setminus\gamma$ consists of  two connected components
    $\m{C}_1$ and~ $\m{C}_2$.
    \item Let $\mathcal{D}_\a=\{F_j\in\F \colon \int F_j\subset
        \m{C}_{\a}\}\subset \F$, $\a=1,2$. Then $\m{D}_1\sqcup \m{D}_2=\F$.
\item Closure $\overline{\m{C}_\a}$ is homeomorphic to a disk. We have
    $\overline{\m{C}_\a}=|\m{D}_\a|$.
\end{enumerate}
\end{theorem}
\section{Belts on fullerenes}\label{Belts}
\begin{defin}
Let $P$ be a simple polytope.

A \emph{thick path} is a sequence of facets $(F_{i_1},\dots,F_{i_k})$, such
that $F_{i_j}\cap F_{i_{j+1}}$ is an edge for $j=1,\dots,k-1$.

A \emph{$k$-loop}, $k\geqslant 2$, is a closed thick path
$\mathcal{L}=(F_{i_1},\dots,F_{i_k}, F_{i_{k+1}})$, $F_{i_{k+1}}=F_{i_1}$. We
denote it simply by $(F_{i_1},\dots, F_{i_k})$. By definition a $1$-loop $\L$
consists of one facet $F_{i_1}$. A $k$-loop is \emph{simple}, if it consist
of pairwise different facets.
\end{defin}
\begin{defin}
Let $\gamma$ be a simple edge-cycle. Walking round $\gamma$ in $\C_{\a}$ we
obtain an $l_\a$-loop $\L_\a\subset \D_\a$, $\a\in\{1,2\}$. We say that
$\L_\a$ \emph{borders} the cycle $\gamma$.

An $l_1$-loop $\L_1=(F_{i_1},\dots,F_{i_{l_1}})$ \emph{borders} an $l_2$-loop
$\L_2=(F_{j_1},\dots,F_{j_{l_2}})$, if they border the same edge-cycle
$\gamma$.  If $l_2=1$, then $\L_1$ \emph{surrounds} a facet~$F_{j_1}$.

Let $F_{i_p}\in\L_1$ has $a^1_p$ edges in common with $\gamma$, and
$F_{j_q}\in\L_2$ has $a^2_q$ edges in common with $\gamma$.
\end{defin}
\begin{lemma}\label{b-lemma}
If $\L_1$ borders $\L_2$, then one of the following holds:
\begin{enumerate}
\item $\L_\a$ is a $1$-loop, and $\L_\b$ is an $a^\b_1$-loop,
    $a^\b_1\geqslant 3$, for some $\a\in\{1,2\}$, and
    $\b=\{1,2\}\setminus\{\a\}$;
\item $l_1,l_2\geqslant 2$, and
    $l_\a=\sum_{r=1}^{l_\b}(a^\b_r-1)=\sum_{r=1}^{l_\b}a^\b_r-l_\b$
for any $\a\in\{1,2\}$,~$\b=\{1,2\}\setminus\{\a\}$.
\end{enumerate}
\end{lemma}
\begin{proof}
If $l_2=1$, then $\gamma$ is the boundary edge-cycle of $F_{j_1}$, successive
edges of $\gamma$ belong to different facets of $\L_1$, and $l_1=a^2_1$.
Similarly for $l_1=1$.

Let $l_1,l_2\geqslant 2$. Any edge of $\gamma$ is an intersection of a facet
of $\L_1$ with a facet of $\L_2$. The successive edges in $\gamma$ belong to
the same facet in $\L_\a$ if and only if they belong to successive facets in
$\L_\b$, hence $l_\a=\sum_{r=1}^{l_\b}(a^\b_r-1)=\sum_{r=1}^{l_\b}-l_\b$.
\end{proof}

\begin{defin}
A \emph{$k$-belt}, $k\geqslant 3$, is a $k$-loop $\B=(F_{i_1},\dots,F_{i_k})$
such that
\begin{itemize}
\item $F_{i_1}\cap F_{i_2}\cap F_{i_3}=\varnothing$, if $k=3$;
\item $F_{i_p}\cap F_{i_q}\ne\varnothing$ if and only if
    $(p,q)\in\{(1,2),(2,3),\dots,(k-1,k),(k,1)\}$, if $k\geqslant 4$. .
\end{itemize}
A simple $3$-polytope $P\ne\Delta^3$ without $3$-belts is called a
\emph{flag} $3$-polytope.
\end{defin}
\begin{rem}
For the simplicial polytope $P^*$ dual to the simple polytope $P$ the notion
of a $k$-belt corresponds to a non-chordal $k$-cycle that does not surround a
triangular facet.
\end{rem}

\begin{figure}[h]
\includegraphics[height=5cm]{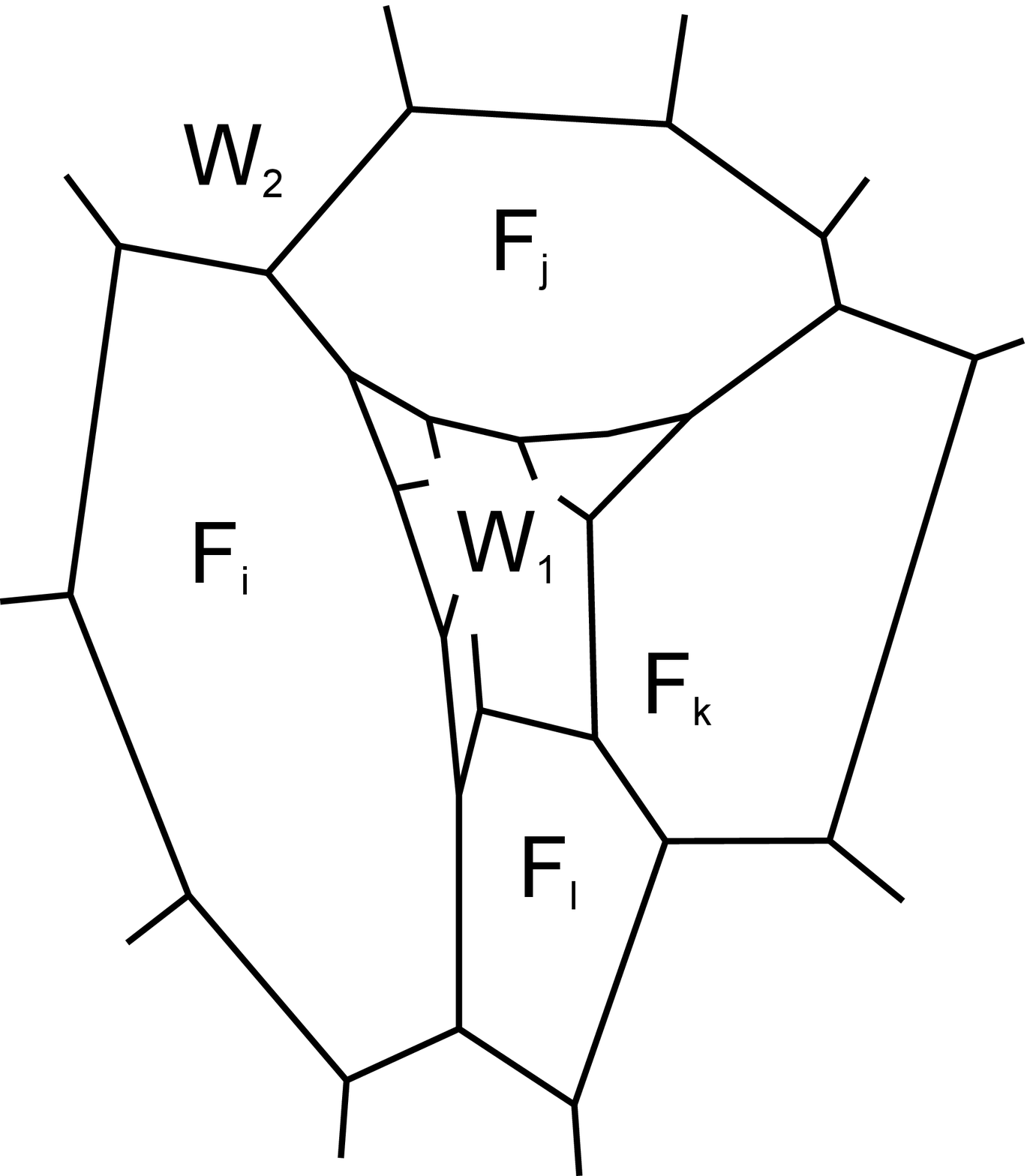}
\caption{$4$-belt}\label{4Belt-pic}
\end{figure}
\begin{lemma}\label{belt-lemma}
Let $\mathcal{B}=(F_{i_1},\dots,F_{i_k})$ be a $k$-belt. Then
\begin{enumerate}
\item  $|\B|$ is homeomorphic to a cylinder;
\item $\partial|\B|$ consists of two simple edge-cycles $\gamma_1$ and
    $\gamma_2$.
\item $\partial P\setminus |\B|$ consists of two connected components
    $\m{P}_1$ and $\m{P}_2$.
\item  Let $\W_\a=\{F_j\in\F \colon \int F_j\subset \m{P}_{\a}\}\subset
    \F$, $\a=1,2$. Then $\W_1\sqcup \W_2\sqcup \B=\F$.
\item $\overline{\P_\a}=|\W_\a|$ is homeomorphic to a disk, $\a=1,2$.
\item $\partial \P_\a=\partial \overline{\P_\a}=\gamma_\a$, $\a=1,2$.
\end{enumerate}
\end{lemma}
The proof is straightforward using Theorem \ref{Jordan-Th}.

Let a facet $F_{i_j}\in\B$ has $\alpha_j$ edges in $\gamma_{\alpha}$ and
    $\beta_j$ edges in $\gamma_{\beta}$, $\b=\{1,2\}\setminus\{\a\}$.\\
If $F_{i_j}$ is an $s_{i_j}$-gon, then $\a_j+\b_j=s_{i_j}-2$.

\begin{lemma}\label{loop-lemma}
Let  $P$ be a simple polytope with $p_3=0$, $p_k=0$, $k\geqslant 8$,
$p_7\leqslant 1$, and let $\B_k$ be a $k$-belt consisting of $b_i$ $i$-gons,
$4\leqslant i\leqslant 7$. Then one of the following holds:
\begin{enumerate}
\item $\B_k$ surrounds two $k$-gonal facets $F_s: \{F_s\}=\W_1$, and
    $F_t: \{F_t\}=\W_2$,\\ and all facets of $\B_k$ are quadrangles;
\item $\B_k$ surrounds a $k$-gonal facet $F_s:\{F_s\}=\W_\a$, and borders
    an $l_\b$-loop $\L_\b\subset \W_\b$, $\beta=\{1,2\}\setminus\{\a\}$,
    $l_\b=p_5+2p_6+3p_7\geqslant 2$;
\item $\B_k$ borders an $l_1$-loop $\L_1\subset \W_1$ and an $l_2$-loop
    $\L_2\subset W_2$, where
    \begin{enumerate}
    \item  $l_\alpha=\sum_{j=1}^k(\a_j-1)\geqslant 2$, $\alpha=1,2$;
    \item $l_1+l_2=2k-2b_4-b_5+b_7\leqslant 2k+1$.
    \item $\min\{l_1,l_2\}\leqslant
        k-b_4-\lceil\frac{b_5-b_7}{2}\rceil\leqslant k$.
    \item If $b_7=0$, $l_1, l_2\geqslant k$, then $l_1=l_2=k$,
        $b_4=b_5=0$, $b_6=k$.
    \end{enumerate}
\end{enumerate}
If $l_\alpha=2$, then $\L_\a=\{F_i,F_j\}$, and $k=s_i+s_j-4\geqslant 4$.
\end{lemma}
\begin{figure}[h]
\begin{tabular}{c}
\includegraphics[scale=0.3]{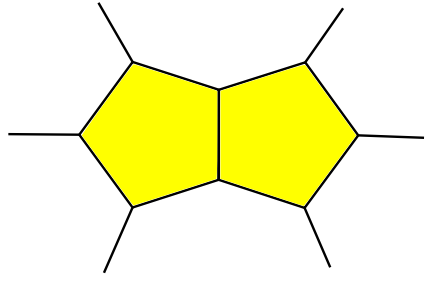}
\end{tabular}
\caption{$2$-loop with $p=q=5$}\label{2-loop}
\end{figure}
\begin{proof}
Walking round $\gamma_\a$ in $\P_\a$ we obtain an $l_\a$-loop $\L_\a\subset
\W_{\a}$.

If $\B_k$ surrounds two $k$-gons $F_s:\{F_s\}=\W_1$, and $F_t:\{F_t\}=\W_2$,
then all facets in $\B_k$ are quadrangles.

If $\B_k$ surrounds a $k$-gon $F_s:\{F_s\}=\W_\a$ and borders an
$l_\beta$-loop $\L_\b\subset W_\b$, $\l_\b\geqslant 2$, then from Lemma
\ref{b-lemma} we have
$\l_\beta=\sum\limits_{i=1}^k(s_{j_i}-3-1)=\sum\limits_{j=4}^7jb_j-4\sum\limits_{j=4}^7b_j=b_5+2b_6+3b_7$.

If $\B_k$ borders an $\l_1$-loop $\L_1$ and an $l_2$-loop $\L_2$,
$l_1,l_2\geqslant 2$, then (a) follows from Lemma~\ref{b-lemma}.
$$
l_1+l_2=\sum\limits_{i=1}^k(\a_i+\b_i-2)=\sum\limits_{i=1}^k(s_{j_i}-4)=\sum\limits_{j=4}^7jb_j-4\sum\limits_{j=4}^7b_j=b_5+2b_6+3b_7=2k-2b_4-b_5+b_7.
$$
We have $\min\{\l_1,\l_2\}\leqslant
k-b_4-\lceil\frac{b_5-b_7}{2}\rceil\leqslant k$

If $b_7=0$ and $l_1,l_2\geqslant k$, then from (3b) we have $l_1=l_2=k$,
$b_4=b_5=0$, $b_6=k$.

The last statement follows from Lemma \ref{b-lemma}.
\end{proof}

\begin{lemma}\label{3-loop-lemma}
Let $P$ be a simple polytope with $p_3=0$. If a $3$-loop $\L_3=(F_i,F_j,F_k)$
is not a $3$-belt, then $F_i\cap F_j\cap F_k$ is a vertex
(Fig.~\ref{3-loop}а), and $\L_3$ borders a $k$-loop,
$k=s_i+s_j+s_k-9\geqslant 3$.
\end{lemma}
\begin{figure}[h]
\begin{tabular}{cc}
\includegraphics[scale=0.2]{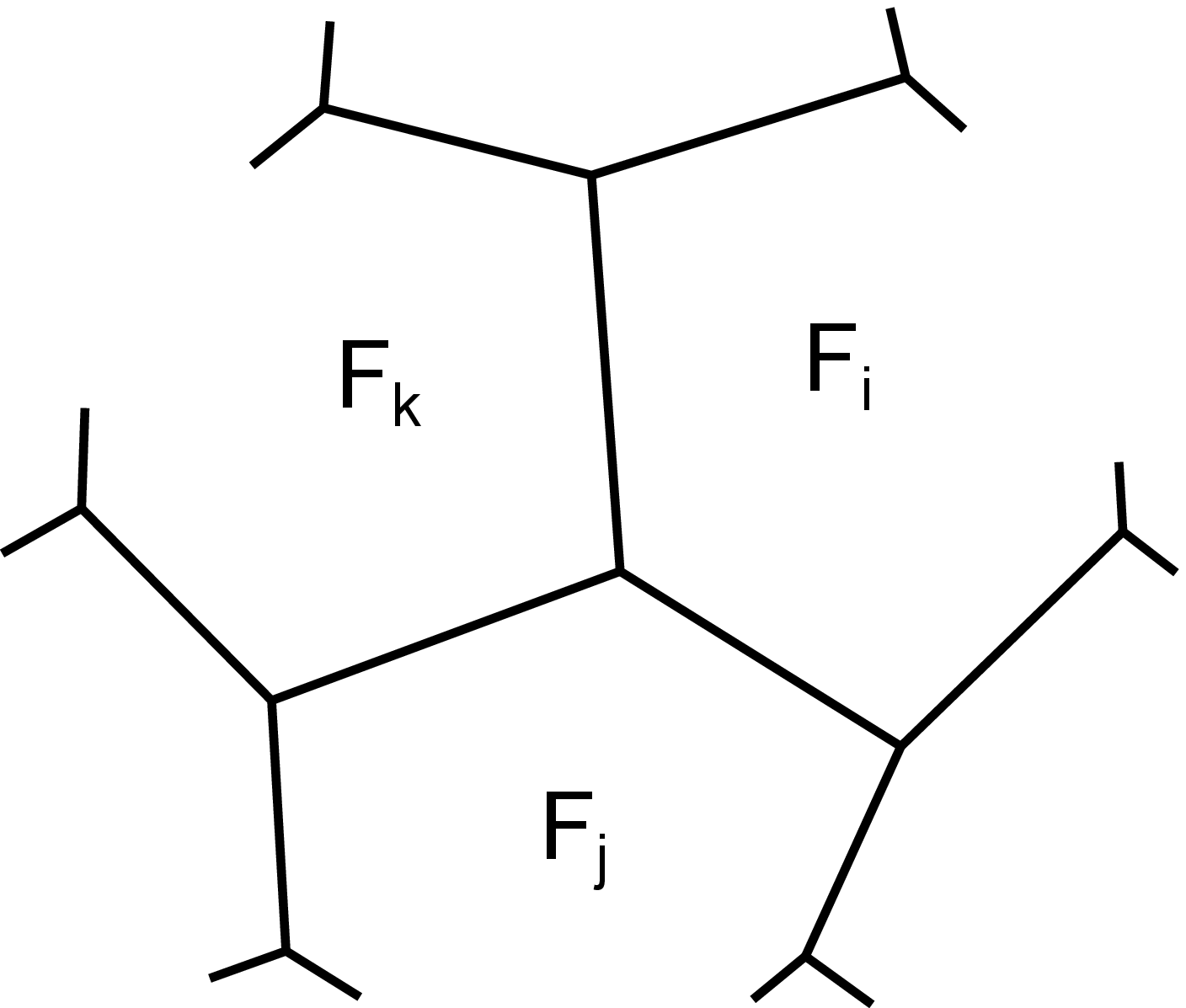}&\includegraphics[scale=0.2]{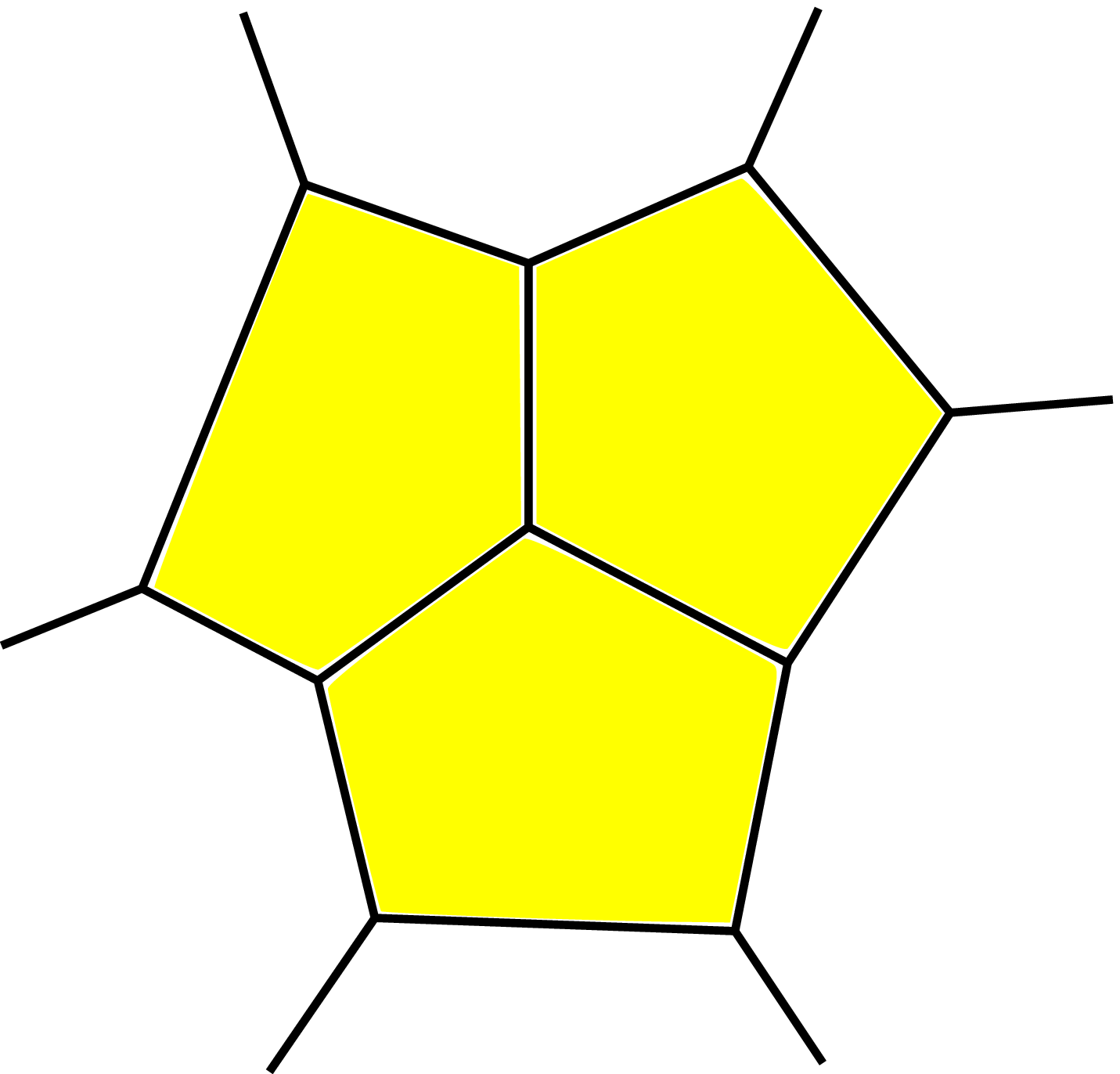}\\
\\
а)&б)
\end{tabular}
\caption{}\label{3-loop}
\end{figure}
\begin{proof}
Since $\L_3$ is not a $3$-belt, $F_i\cap F_j\cap F_k$ is a vertex. Any two
vertices in $\vert(F_r)$ are different. Let two vertices
$\ib{v}\in\vert(F_i)$ and $\ib{w}\in\vert(F_j)$ coincide. Then
$\ib{v},\ib{w}\in F_i\cap F_j$. From this we obtain that the vertices
$\vert(F_i)\cup\vert(F_j)\cup\vert(F_k)\setminus\{F_i\cap F_j\cap F_k\}$ form
a simple edge-cycle $\gamma$ with $\{F_i,F_j,F_k\}=\D_\a$. Walking round
$\gamma$ in $\C_\b$ we obtain a $k$-loop. Lemma \ref{b-lemma} implies
$k=(s_i-2)+(s_j-2)+(s_k-2)-3=s_i+s_j+s_k-9$. Since $s_i,s_j,s_k\geqslant 4$,
we have~$k\geqslant 3$.
\end{proof}
\begin{theorem}\label{3-belt-theorem}
Let $P$ be simple $3$-polytope with $p_3=0$, $p_4\leqslant 2$, $p_7\leqslant
1$, and $p_k=0$, $k\geqslant 8$. Then it has no $3$-belts. In particular, it
is a flag polytope.
\end{theorem}
\begin{proof}
Let $P$ has a $3$-belt $\B_3$. By Lemma \ref{loop-lemma} it borders an
$l_1$-loop $\L_1$ and $\l_2$-loop $\L_2$, where $\l_1,\l_2\geqslant 2$,
$\l_1+\l_2\leqslant 7$. Also from Lemma \ref{loop-lemma} we have
$l_1,l_2\geqslant 3$, and $\min\{l_1,l_2\}=3$. If $\B_3$ contains a heptagon,
then $\W_1,\W_2$ contain no heptagons. If $\B_3$ contains no heptagons, then
$l_1=l_2=3$, and one of the sets $\W_1$ and $\W_2$, say $\W_\a$, contains no
heptagons. In both cases we obtain a set $\W_\a$ without heptagons and a
$3$-loop $\L_\a\subset \W_\a$. Then $\L_\a$ is a $3$-belt, else by Lemma
\ref{3-loop-lemma}  $\B_3$ should have at least $4+4+5-9=4$ facets.
Considering the other boundary component of $\L_\a$ we obtain again a
$3$-belt there. Thus we obtain an infinite series of different $3$-belts
inside $|\W_\a|$. A contradiction.
\end{proof}
\begin{cor}[\cite{BE15}]
Any fullerene is a flag polytope.
\end{cor}

In what follows we will implicitly use the fact that if a $3$-polytope has no
$3$-belts, then any set of $3$ pairwise intersecting facets has a common
vertex.

\begin{lemma}\label{facet-belt}
Any $k$-gonal facet, $k\geqslant 3$, of a flag $3$-polytope $P$  is
surrounded by a $k$-belt.
\end{lemma}
\begin{proof}
A $k$-gonal facet $F$ is surrounded by a $k$-loop $\L_k=(F_{i_1},\dots,
F_{i_k})$. If $F_{i_p}\cap F_{i_q}\ne\varnothing$, then $F_{i_p}\cap
F_{i_q}\cap F$ is a vertex, therefore the facets are successive while walking
round the boundary of $F$. If $k>3$, this implies that $\L_k$ is a $k$-belt.
For $k=3$ let $F_{i_1}\cap F_{i_2}\cap F_{i_3}\ne\varnothing$. Then it is a
vertex. But $F_{i_1}\cap F_{i_2}\cap F$, $F_{i_2}\cap F_{i_3}\cap F$, and
$F_{i_3}\cap F_{i_1}\cap F$ are also vertices, hence $P\simeq \Delta^3$. Thus
$\L_k$ is indeed a $k$-belt.
\end{proof}
\begin{cor}\label{flag-3}
For a flag polytope we have $p_3=0$.
\end{cor}
\begin{lemma}\label{4-loop-lemma}
Let $P$ ba a flag polytope. If a simple $4$-loop $\L_4=(F_i,F_j,F_k,F_l)$ is
not a $4$-belt, then it consist of $4$ facets surrounding an edge
(Fig.~\ref{4-loop}а), and $\L_4$ borders a $k$-loop, $
k=s_1+s_2+s_3+s_4-14\geqslant 2$.
\end{lemma}
\begin{figure}[h]
\begin{tabular}{cc}
\includegraphics[scale=0.2]{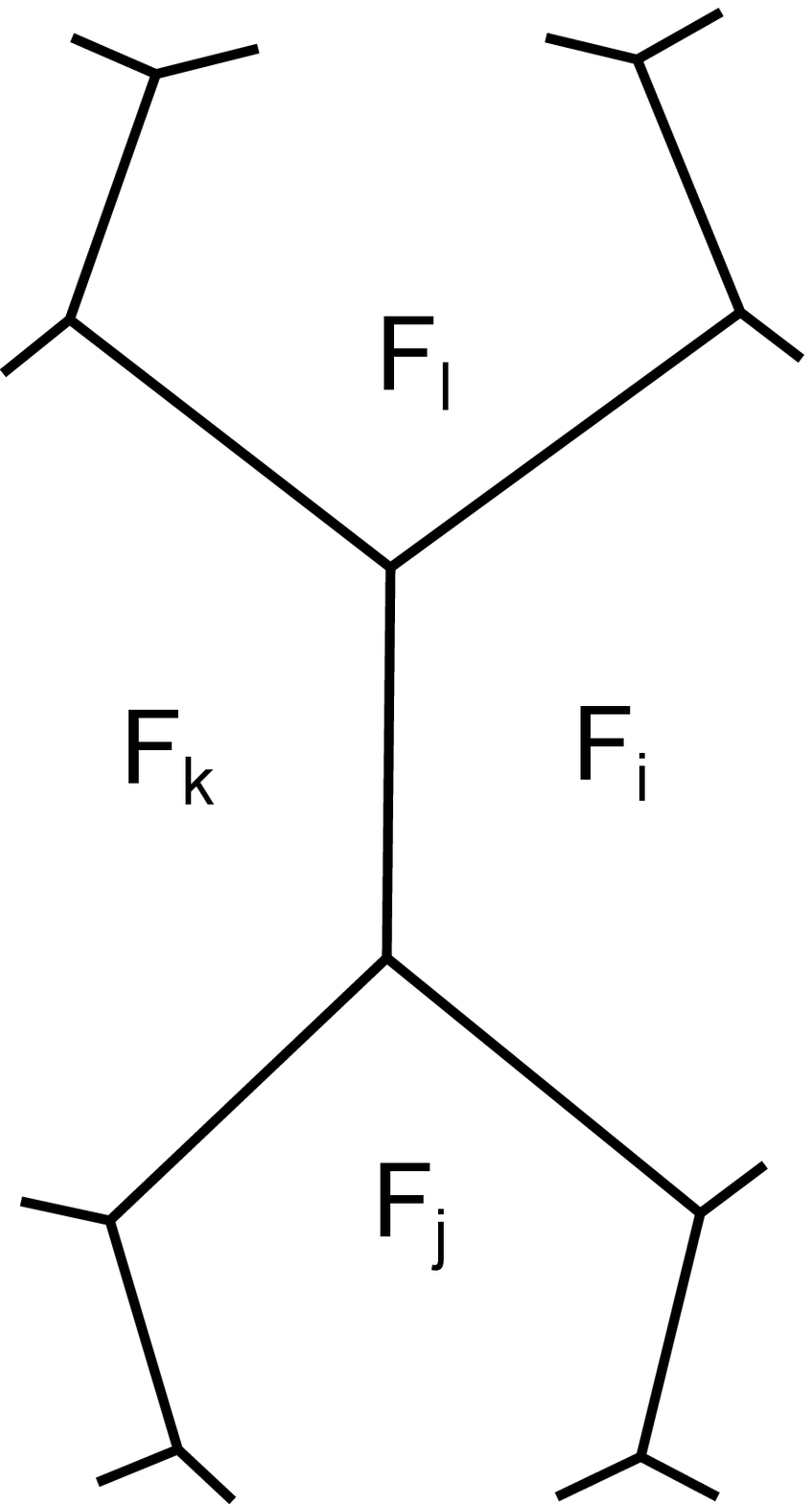}&\includegraphics[scale=0.2]{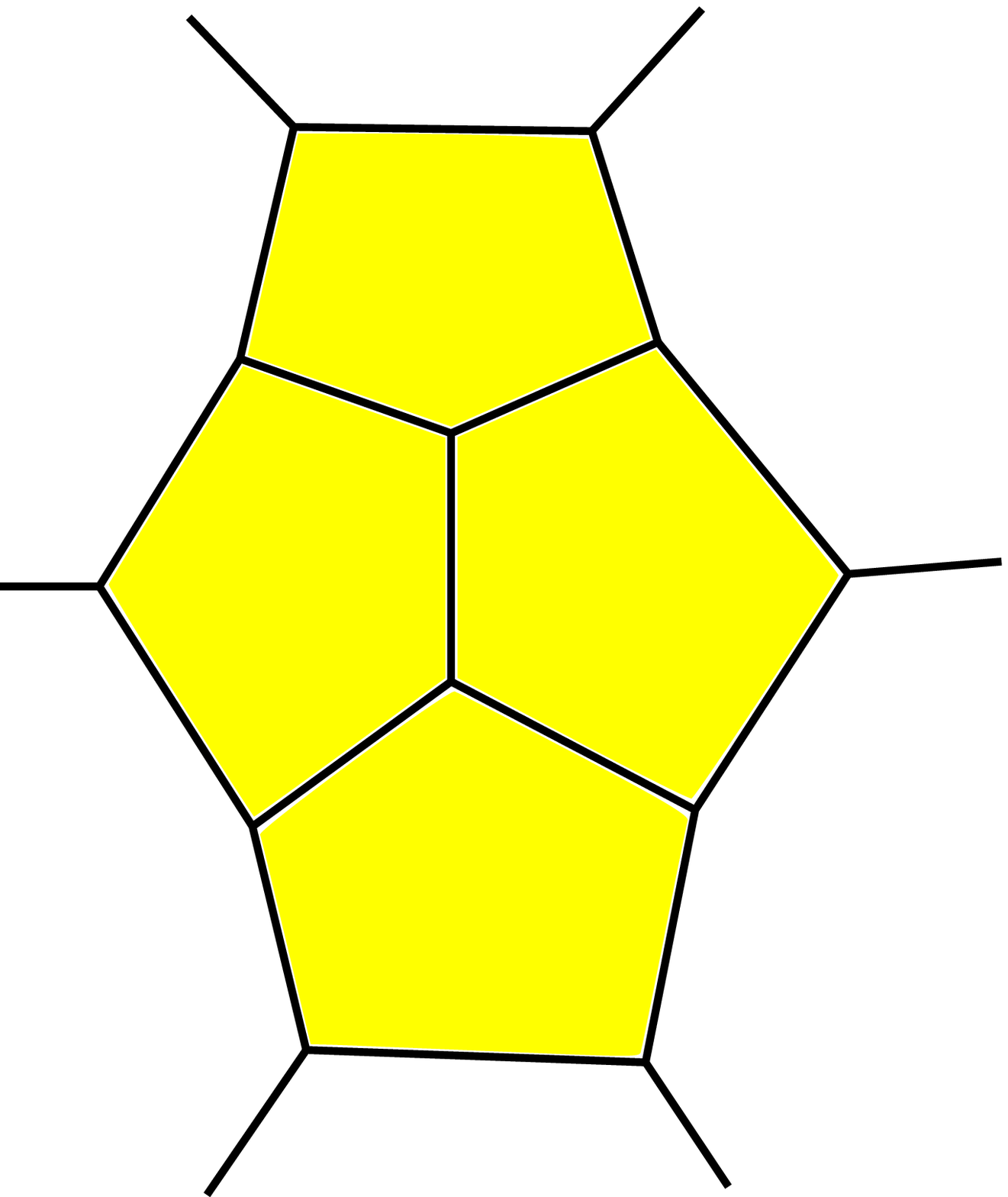}\\
а)&б)
\end{tabular}
\caption{}\label{4-loop}
\end{figure}
\begin{proof}
If $\L_4$ is not a $4$-belt, then either $F_i\cap F_k\ne\varnothing$, or
$F_j\cap F_l\ne\varnothing$. Without loss of generality let $F_i\cap
F_k\ne\varnothing$. Then $F_i\cap F_j\cap F_k$ and $F_i\cap F_k\cap F_l$ are
vertices, therefore $\L_4$ surrounds an edge $F_i\cap F_k$. We have $F_j\cap
F_l=\varnothing$, else $(F_i,F_j,F_l)$ is a $3$-belt. Then if two vertices in
$\vert(F_i)\cup\vert(F_j)\cup\vert(F_k)\cup\vert(F_l)\setminus\{F_i\cap
F_j\cap F_k,F_i\cap F_k\cap F_l\}$ coincide, they lie in $F_i\cap F_j$,
$F_j\cap F_k$,  $F_k\cap F_l$, or $F_l\cap F_i$. Therefore these vertices
form a simple edge-cycle $\partial|\L_4|$ and $\L_4$ borders a $k$-loop. By
Lemma \ref{b-lemma} we have
$$
k=(s_i-3)+(s_j-2)+(s_k-3)+(s_l-2)-4=s_i+s_j+s_k+s_l-14\geqslant 4+4+4+4-14=2.
$$
\end{proof}
\begin{lemma}\label{4-dl-lemma}
Let $P$ be a flag polytope with $p_4\leqslant 2$. If a $4$-belt $\B_k$
borders a $4$-loop $\L_4$, then all facets in $\L_4$ are different.
\end{lemma}
\begin{proof}
Let $\L_4=(F_i,F_j,F_k,F_l)$ and $\B_k$ border the same simple edge-cycle
$\gamma$. The successive facets in $\L_4$ are different by definition. Let
$F_i=F_k$. Then $F_j\ne F_l$, else $\L_4=(F_i,F_j,F_i,F_j)$ and some vertex
in $F_i\cap F_j$ is passed twice while walking round~$\gamma$. Also $F_j\cap
F_l\ne\varnothing$, else $F_i\cap F_j\cap F_l\in\int |\L_4|$ is a vertex, and
the other vertex of $F_i\cap F_j$ is passed twice while walking round $
\gamma$. Thus, $\L_4$ consists of a facet $F_i$ and facets $F_j$ and $F_l$
that intersect $F_i$ by non-incident edges. If an edge $E\in F_r$ lies in
$\gamma$, $r=i,j,l$, then walking from $E$ round $\partial F_r$ in both
directions we obtain edges in $\gamma$ until we meet a vertex in
$\{F_i,F_j,F_l\}\setminus\{F_r\}$. Hence $\gamma=\partial \left(F_i\cup
F_j\cup F_l\right)$, and by Lemma \ref{b-lemma} $\B_k$ has
$(s_i-2-2)+(s_j-2)+(s_l-2)=s_i+s_j+s_l-8\geqslant 4+4+5-8=5$ facets. A
contradiction.
\end{proof}

\begin{theorem}\label{4-belt-theorem}
Let $P$ be a simple polytope with all facets pentagons and hexagons with at
most one exceptional facet $F$ being a quadrangle or a heptagon.
\begin{enumerate}
\item If $P$ has no quadrangles, then $P$ has no $4$-belts.
\item If $P$ has a quadrangle $F$, then there is exactly one $4$-belt. It
    surrounds $F$.
\end{enumerate}
\end{theorem}
\begin{proof}
By Theorem \ref{3-belt-theorem} the polytope $P$ is flag.

By Lemma \ref{facet-belt} a quadrangular facet is surrounded by a $4$-belt.

Let $\B_4$ be a $4$-belt that does not surround a quadrangular facet. By
Lemma \ref{loop-lemma} it borders an $l_1$-loop $\L_1$ and $\l_2$-loop
$\L_2$, where $\l_1,\l_2\geqslant 2$, and $\l_1+\l_2\leqslant 9$. We have
$l_1,l_2\geqslant 3$, since a $2$-loop borders a $k$-loop with $k\geqslant
4+5-4=5$. We have $l_1,l_2\geqslant 4$ by Theorem \ref{3-belt-theorem} and
Lemma \ref{3-loop-lemma}, since a $3$-loop that is not a $3$-belt borders a
$k$-loop with $k\geqslant 4+5+5-9=5$. Also $\min\{l_1,l_2\}=4$.  If $\B_4$
contains a heptagon, then $\W_1,\W_2$ contain no heptagons. If $\B_3$
contains no heptagons, then $l_1=l_2=4$, and one of the sets $\W_1$ and
$\W_2$, say $\W_\a$, contains no heptagons. In both cases we obtain a set
$\W_\a$ without heptagons and a $4$-loop $\L_\a\subset \W_\a$. Then $\L_\a$
is a $4$-belt, else by Lemmas \ref{4-loop-lemma} and \ref{4-dl-lemma} $\B_4$
should have at least $4+5+5+5-14=5$ facets. Applying the same argument to
$\L_\a$ instead of $\B_k$, we have that either $\L_\a$ surrounds on the
opposite side a quadrangle, or it borders a $4$-belt and consists of
hexagons. In the first case by Lemma \ref{loop-lemma} $\L_\a$ consists of
pentagons. Thus we can move inside $\W_\a$ until we finish with a quadrangle.
If $P$ has no quadrangles, then we obtain a contradiction. If $P$ has a
quadrangle $F$, then it has no heptagons, therefore moving inside $\W_\b$ we
should meet some other quadrangle. A contradiction.
\end{proof}
\begin{cor}\label{4-belt-ful}
Fullerenes have no $4$-belts.
\end{cor}

\begin{lemma}\label{5-dl-lemma}
Let $P$ be a fullerene, and let a $k$-loop $\L_k$ border a $5$-belt $\B_5$.
Then the facets of $\L_k$ are pairwise different.
\end{lemma}
\begin{proof}
Let $\B_5=(F_i,F_j,F_k,F_l,F_r)$, and let the facet $F_s\in \W_{\alpha}$ be
incident to non-successive facets of $\B_5$, say $F_i$ and $F_k$. We obtain a
$4$-loop $(F_i,F_j,F_k,F_s)$. Since $F_i\cap F_k=\varnothing$, Theorem
\ref{4-belt-theorem} implies $F_j\cap F_s\ne\varnothing$. Then $F_i\cap
F_j\cap F_s$, $F_j\cap F_k\cap F_s$ are vertices, and $F_j\cap \gamma_\a$
consists of one edge $F_j\cap F_s$. Thus if the facet $F_s\in \W_{\alpha}$
intersects two facets from $\B_5$ by $E_1$ and $E_2$, then all the edges in
one of the arcs of $\partial F_s$ between $E_1$ and $E_2$ belong to
$\gamma_\a$, and $F_s$ can not appear twice in $\L_k$.
\end{proof}

\begin{lemma}\label{5-loop-lemma}
Let $P$ be a fullerene. If a simple $5$-loop $\L_5=(F_i,F_j,F_k,F_l,F_r)$ is
not a  $5$-belt, then up to a relabeling it has a form drawn on the
Fig.~~\ref{5-loop-1}. Moreover, it borders a $k$-loop,
$k=s_i+s_j+s_k+s_l+s_r-19\geqslant 6$.
\begin{figure}[h]
\includegraphics[scale=0.3]{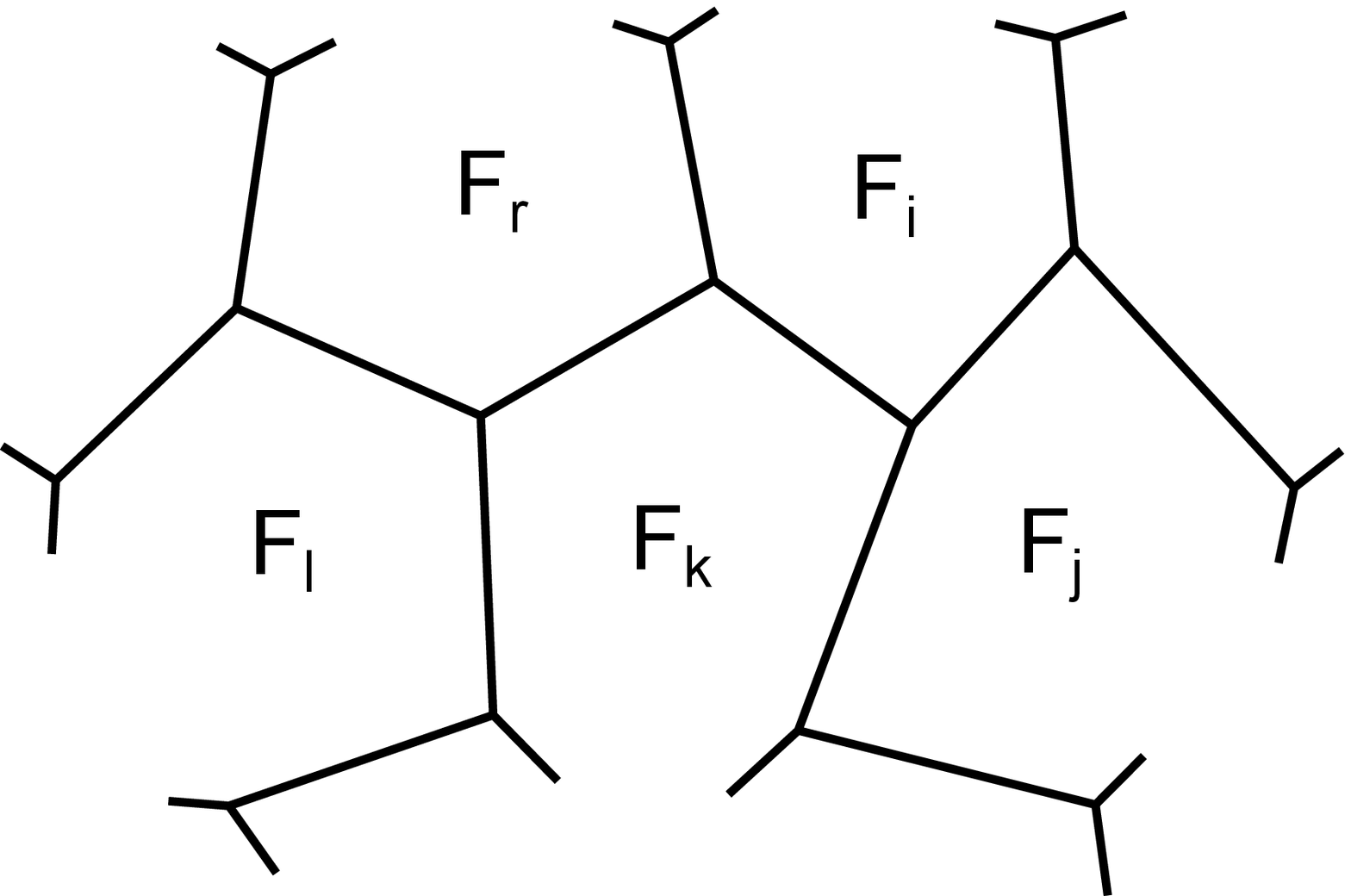}
\caption{$5$-loop}\label{5-loop-1}
\end{figure}
\end{lemma}
\begin{proof}
If $\L_5$ is not a $5$-belt, then some two non-successive facets intersect.
Without loss of generality assume $F_i\cap F_k\ne\varnothing$. Then $F_i\cap
F_j\cap F_k$ is a vertex. For $4$-loop $(F_i,F_k,F_l,F_r)$ from Theorem
\ref{4-belt-theorem} we obtain $F_i\cap F_l\ne\varnothing$ or $F_k\cap
F_r\ne\varnothing$. Both cases give the declared fragment. Let the labeling
be as on Fig \ref{5-loop-1}. We have $F_i\cap F_l=\varnothing$, else
$(F_i,F_k,F_l)$ is a $3$-belt, since $F_l\ne F_j$, $F_l\ne F_r$. By the
symmetry $F_r\cap F_j=\varnothing$. We have $F_j\cap F_l\ne\varnothing$, else
$F_j\cap F_l\cap F_k$ is a vertex and $F_k$ is a quadrangle. Then if two
vertices in $\vert(F_i)\cup\vert(F_j)\cup\vert(F_k)\cup\vert(F_l)\cup
\vert(F_r)\setminus\{F_i\cap F_j\cap F_k,F_i\cap F_k\cap F_r, F_k\cap F_l\cap
F_r\}$ coincide, they lie in $F_i\cap F_j$, $F_j\cap F_k$,  $F_k\cap F_l$,
$F_l\cap F_r$ or $F_r\cap F_i$. Therefore these vertices form a simple
edge-cycle $\partial|\L_5|$ and $\L_5$ borders a $k$-loop. By Lemma
\ref{b-lemma} we have
$k=(s_i-4)+(s_j-3)+(s_k-5)+(s_l-3)+(s_r-4)=s_i+s_j+s_k+s_l+s_r-19 \geqslant
25-19=6$.
\end{proof}
\begin{theorem}\label{5-belt-theorem} Let $P$ be a fullerene.
\begin{enumerate}
\item Facets that surround a pentagon, form a $5$-belt. There are $12$ of
    them.

\item If $P$ has a $5$-belt $\B$ not of this type, then
\begin{enumerate}
\item $\B$ consists of hexagons incident with neighbors by opposite
    edges;
\item $P$ consists of $k$ belts of this type bordering each other and
    two caps \ref{Cap-1}a).
\begin{figure}[h]
\begin{tabular}{ccc}
\includegraphics[scale=0.2]{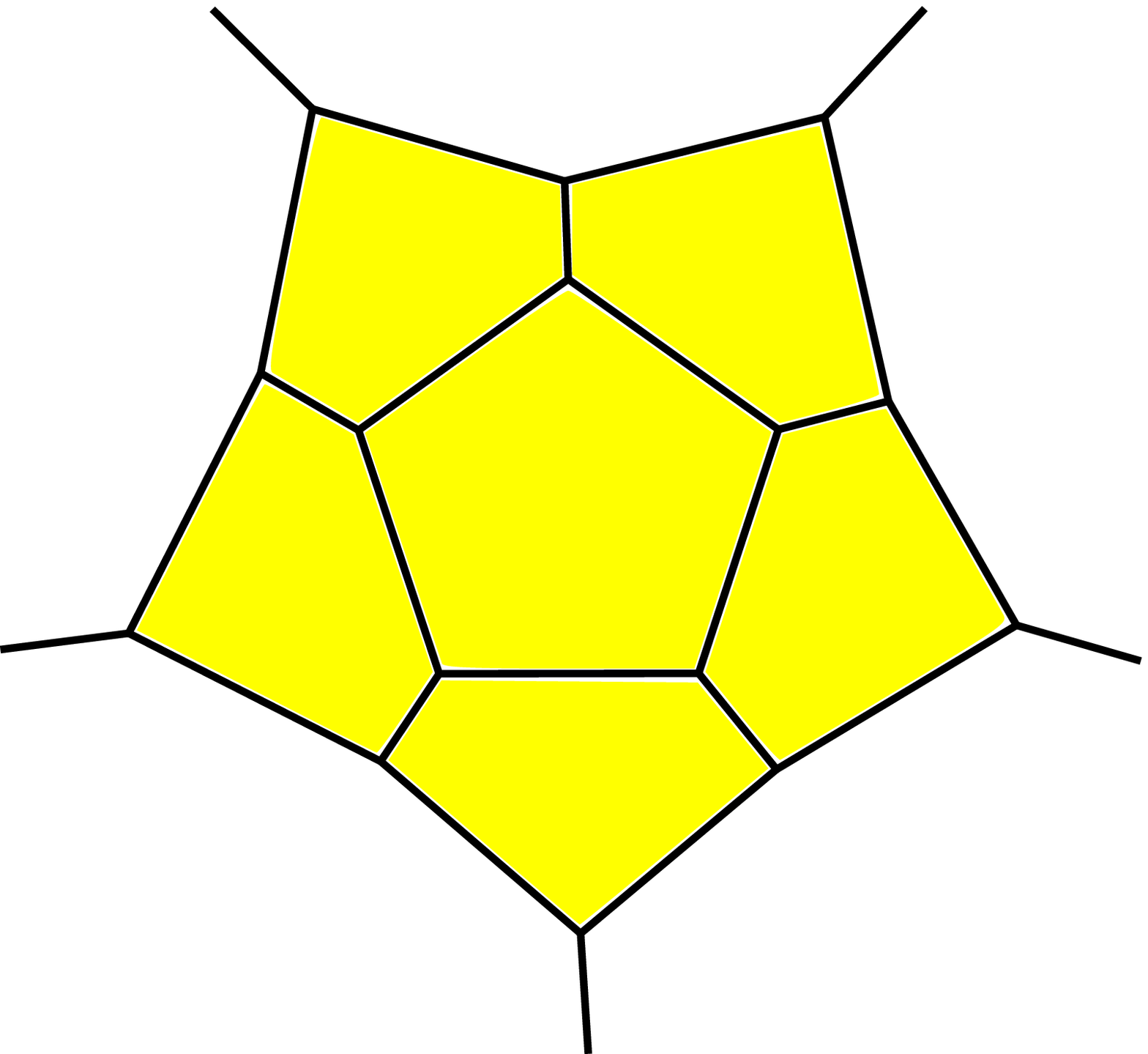}&\includegraphics[scale=0.17]{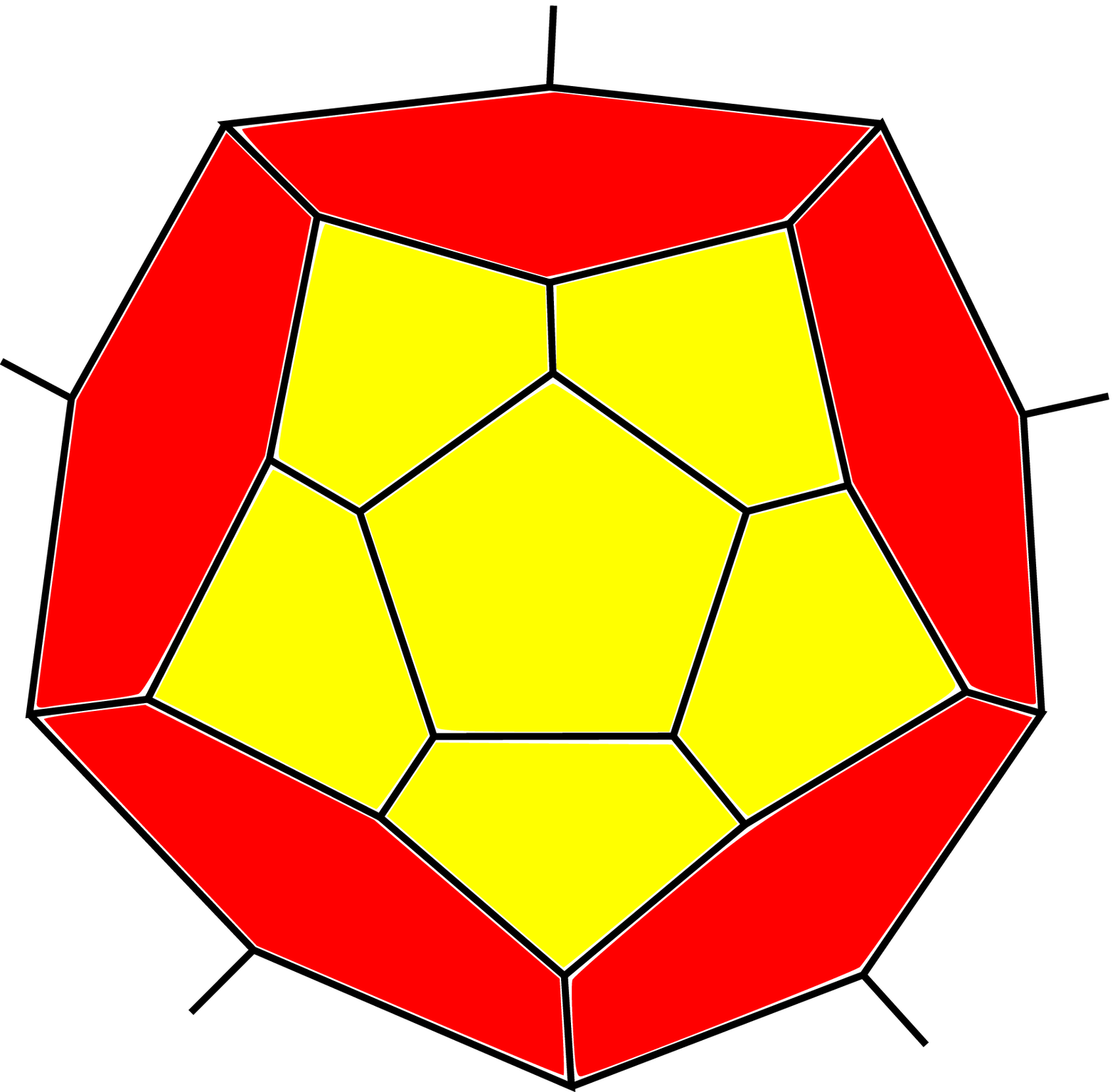}&\includegraphics[scale=0.15]{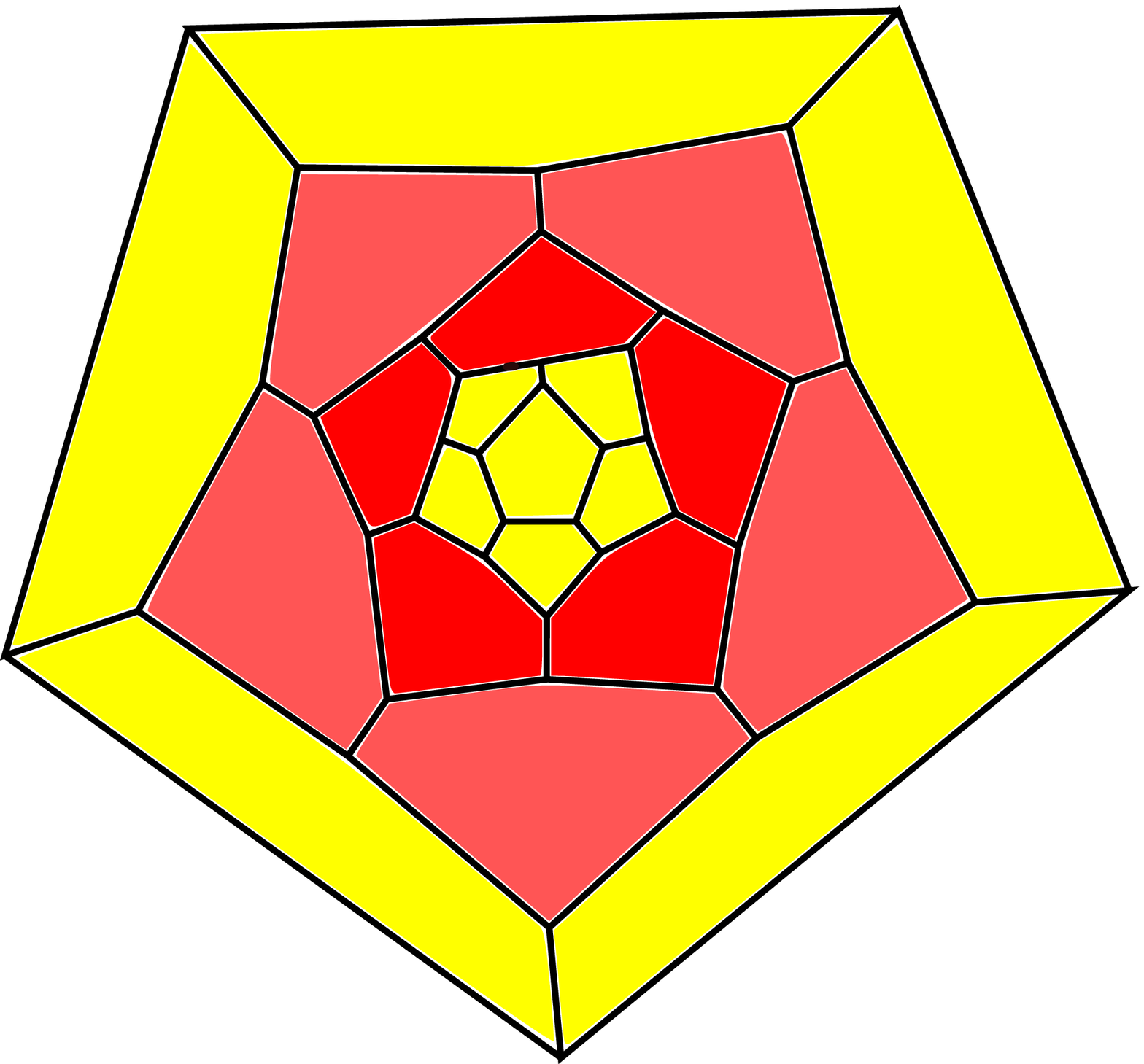}\\
a)&b)&c)
\end{tabular}
\caption{}\label{Cap-1}
\end{figure}
\item There are exactly $12+k$ five-belts.
\end{enumerate}
\item If a fullerene $P$ has fragment \ref{Cap-1}a), then $P$ has the
    form as in (2) for $k\geqslant 0$.
\end{enumerate}
\end{theorem}

\begin{defin}
We will call the family of combinatorial polytopes arising in (2) by
\emph{Family I}. They are enumerated by the parameter $k\geqslant 0$. For
$k=0$ we have the dodecahedron. The $k$-th fullerene has $p_6=5k$.
\end{defin}
\begin{proof}
(1) Follows from Lemma \ref{facet-belt}.

(2) Let the $5$-belt $\B_5$ do not surround a pentagon. By Lemma
\ref{loop-lemma} it borders an $l_1$-loop $\L_1\subset \W_1$ and an
$l_2$-loop $\L_2\subset \W_2$, $l_1,l_2\geqslant 3$, $l_1+l_2\leqslant 10$.
From Theorem \ref{3-belt-theorem} and Lemma \ref{3-loop-lemma} we obtain
$l_1,l_2\geqslant 4$. From Lemmas  \ref{4-loop-lemma} and \ref{4-dl-lemma},
and Theorem \ref{4-belt-theorem} we obtain that $l_1,l_2\geqslant 5$ and all
facets in $\B_5$ are hexagons. From Lemmas \ref{5-dl-lemma} and
\ref{5-loop-lemma} we obtain that $\L_1$ and $\L_2$ are $5$-belts. Moving
inside $\W_1$ we obtain a series of hexagonal $5$-belts, and this series can
stop only if the last $5$-belt $\B_l$ surrounds a pentagon. Since $\B_l$
borders a $5$-belt,  Lemma \ref{loop-lemma} implies that $\B_l$ consists of
pentagons, which have $(2,2,2,2,2)$ edges on the common boundary with a
$5$-belt. We obtain the fragment \ref{Cap-1}a. Moving from this fragment
backward we obtain a series of hexagonal $5$-belts including $\B_5$ with
facets having $(2,2,2,2,2)$ edges on both boundaries. This series can finish
only with fragment \ref{Cap-1}a again. Thus any belt not surrounding a
pentagon belongs to this series and the number of $5$-belts is equal to
$12+k$.

(3) By Lemma \ref{facet-belt} the outer $5$-loop of the fragment \ref{Cap-1}a
is a $5$-belt. By Lemma \ref{5-dl-lemma} it borders a simple $5$-loop. By
Lemma \ref{5-loop-lemma} it is a $5$-belt. If this belt surrounds a pentagon,
then we obtain a combinatorial dodecahedron (case $k=0$). If not, then $P$
has the form we need by (2).
\end{proof}

\begin{theorem}\label{131-theorem}
If a fullerene $P$ has the fragment from Fig.~\ref{131313}a), then $P$ has
the following structure.
\begin{figure}[h]
\begin{center}
\begin{tabular}{cccc}
\includegraphics[height=2.75cm]{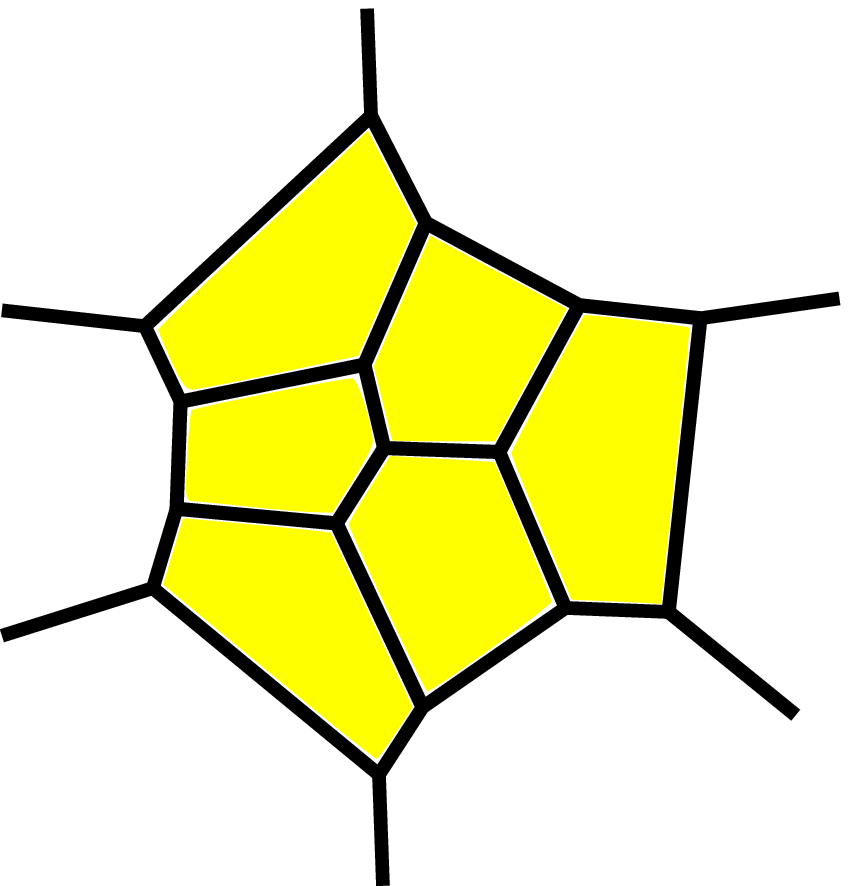}&\includegraphics[height=2.75cm]{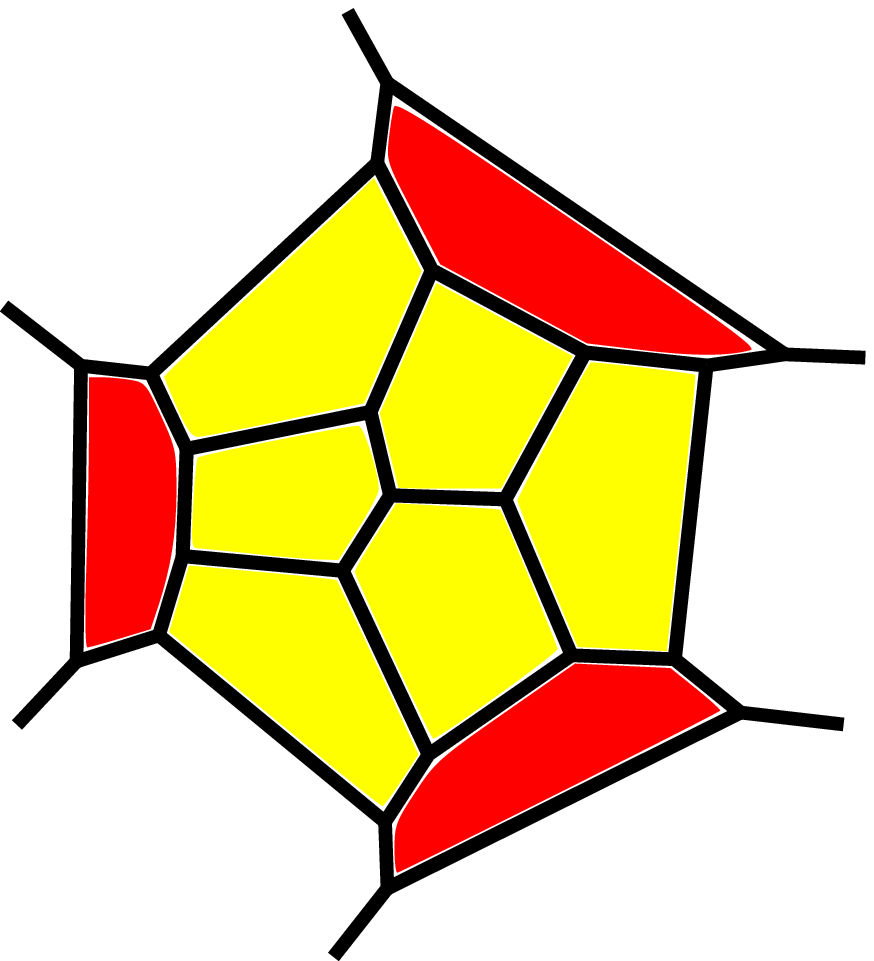}&
\includegraphics[height=2.75cm]{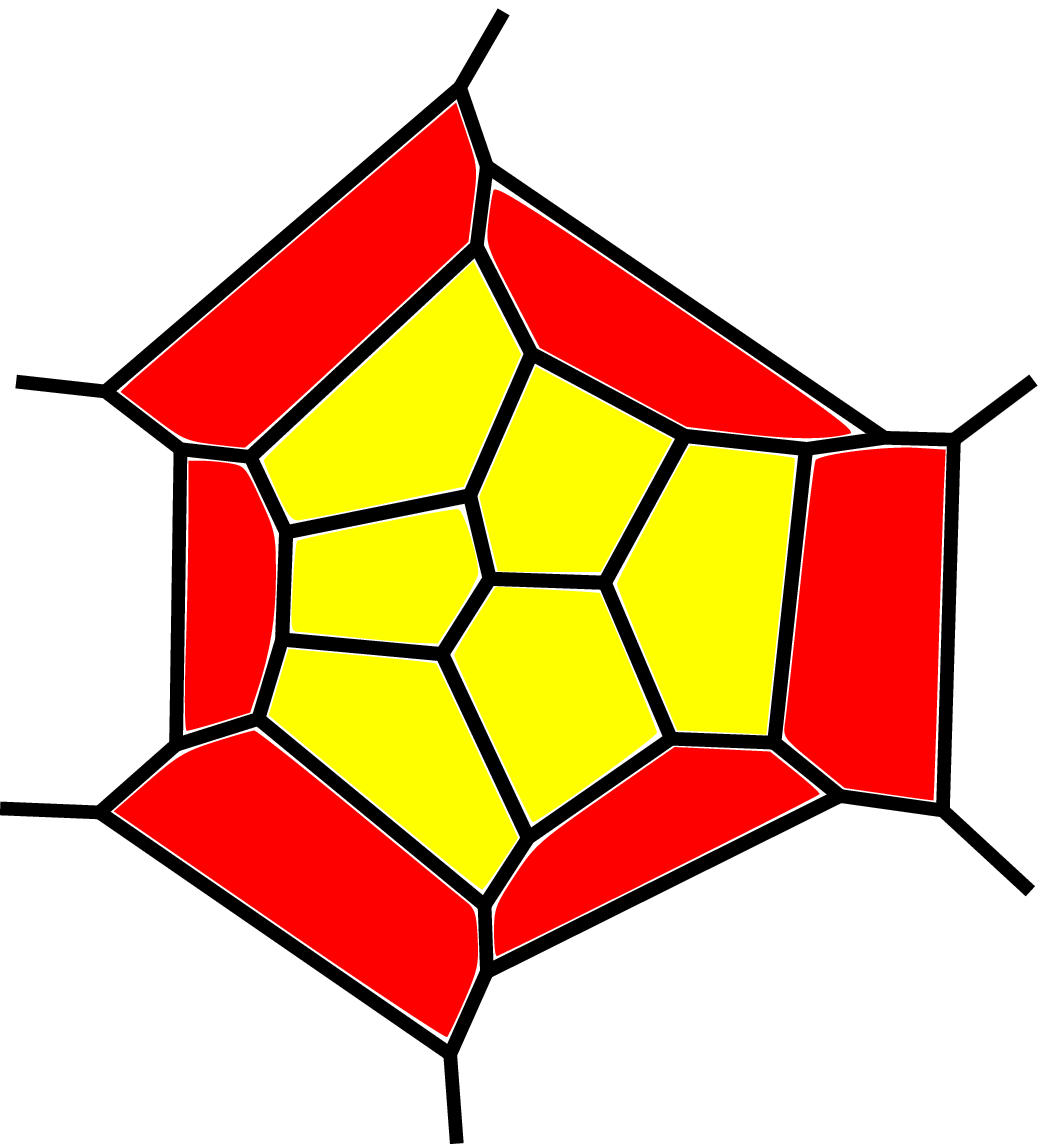}&\includegraphics[height=2.75cm]{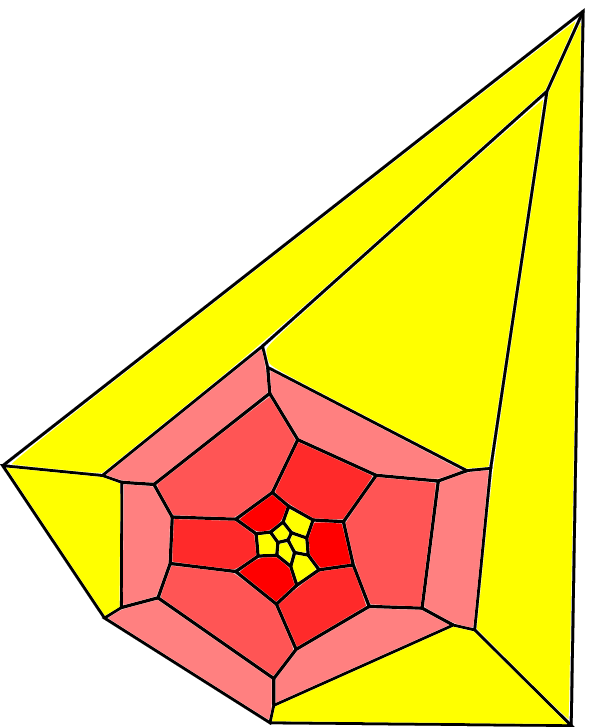}\\
a)&b)&c)&d)
\end{tabular}
\end{center}
\caption{}\label{131313}
\end{figure}
\begin{enumerate}
\item Start with a).
\item Make $k$ steps of the following type: add three hexagons incident
    to the faces with single edges on the boundary.
\item After each step the boundary facets have still $(1,3,1,3,1,3)$
    edges on the boundary.
\item In the end glue up the fragment a).
\end{enumerate}
\end{theorem}
\begin{defin}
We will call the family of combinatorial fullerenes arising in\linebreak
Theorem~\ref{131-theorem} by \emph{Family II}. They are enumerated by the
parameter $k\geqslant 0$. The fullerene with $k=0$ is the dodecahedron.  The
$k$-th fullerene has~$p_6=3k$.
\end{defin}
\begin{rem}
In the literature fullerenes like those in Family I and Family II are called
\emph{capped nanotubes}.
\end{rem}
\begin{proof} We need the following lemma.
\begin{lemma}\label{131313-lemma}
Let $P$ be a fullerene. Let a $6$-loop $\L_1=(F_p,F_t,F_q,F_u,F_v,F_w)$
border a simple $6$-loop $\L_2=(F_i,F_j,F_k,F_l,F_r,F_s)$ with
$(1,3,1,3,1,3)$ edges on the boundary (Fig.~\ref{131313-lem-fig}a). Then
$\L_6$ is simple and either forms the fragment on
Fig.~\ref{131313-lem-fig}b), or $F_p,F_q,F_v$ are hexagons and the $6$-loop
$\L_3=(F_j,F_q,F_l,F_v,F_s,F_p)$ is simple and has $(1,3,1,3,1,3)$ edges on
the boundary component intersecting $F_t$.
\end{lemma}
\begin{figure}[h]
\begin{center}
\begin{tabular}{ccc}
\includegraphics[height=2.75cm]{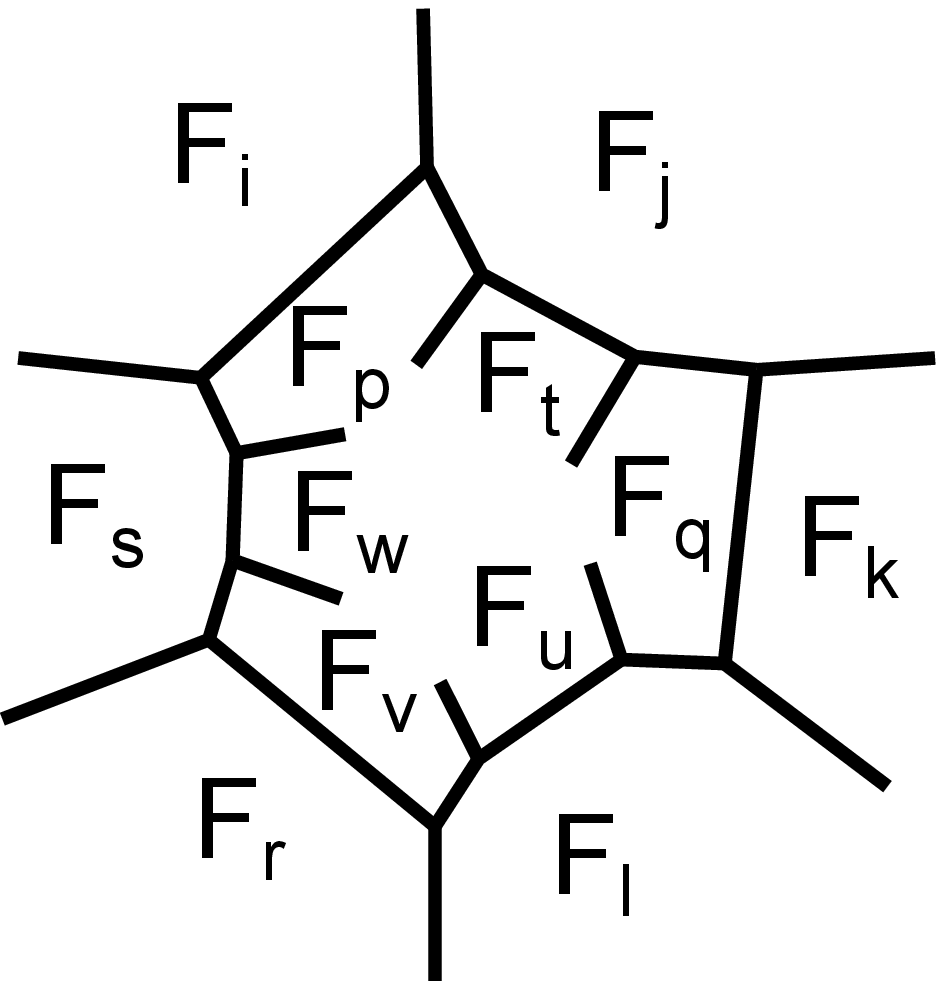}&\includegraphics[height=2.75cm]{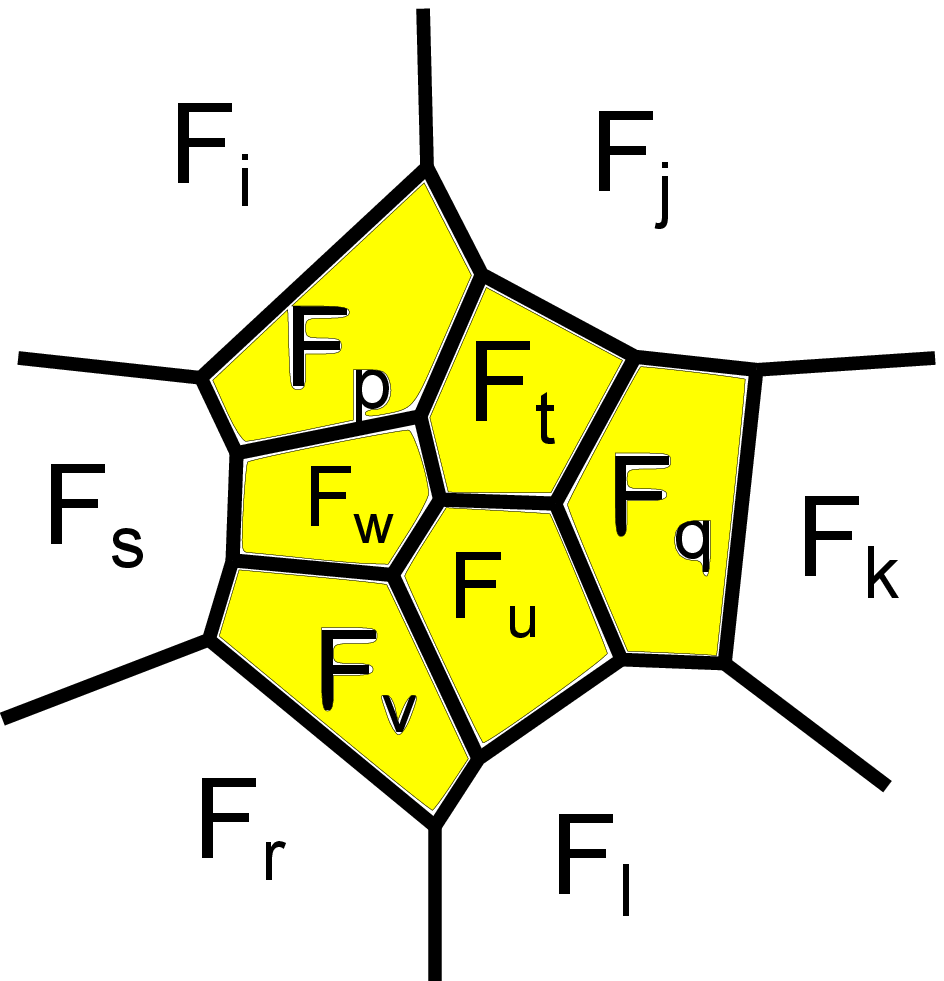}&
\includegraphics[height=2.75cm]{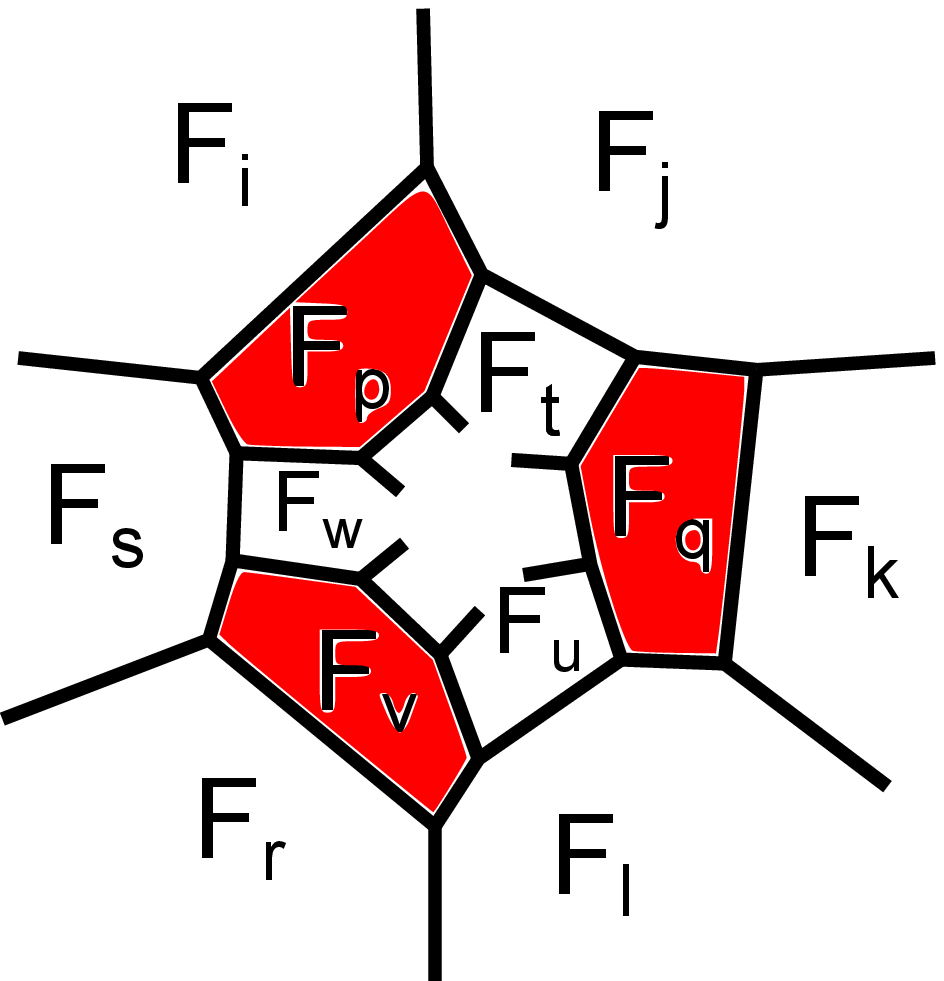}\\
a)&b)&c)
\end{tabular}
\end{center}
\caption{}\label{131313-lem-fig}
\end{figure}
\begin{proof}
Let $\L_1$ and $\L_2$ border the same simple edge-cycle $\gamma$.

We have $F_p\ne F_q$, since they intersect $F_j$ by different edges. Also
$F_p\cap F_q=\varnothing$, else $F_p\cap F_q\cap F_j$ is a vertex and $F_j$
has only $2$ edges in $\gamma$, but not $3$. Similarly $F_p\cap
F_v=\varnothing$. Then $F_p\ne F_u$, else $F_p\cap F_q\ne\varnothing$. We
have $F_t\ne F_u$, else $F_q$ is a quadrangle. By the~symmetry we obtain that
the $6$-loop $\L_6$ is simple and $F_p\cap F_q=F_q\cap F_v=F_v\cap
F_q=\varnothing$.

Let $F_t\cap F_w\ne\varnothing$. Then $F_p\cap F_t\cap F_w$ is a vertex and
$F_p$ is a pentagon. We have a simple $5$-loop $\L_5=(F_t,F_q,F_u,F_v,F_w)$.
It is not a $5$-belt, else by theorem \ref{5-belt-theorem} it should either
surround a pentagon and have $(1,1,1,1,1)$ and $(\b_t,\b_q,\b_u,\b_v,\b_w)$,
$\b_{j}\geqslant 3$, edges on the boundary components, or consist of hexagons
and have $(2,2,2,2,2)$ edges on each boundary component. But $(F_q,F_u)$ have
$(3,1)$ edges on the boundary component.  We have $F_t\cap
F_v\ne\varnothing$, else $F_t\cap F_w\cap F_v$ is a vertex and $F_w$ is a
quadrangle. Similarly $F_q\cap F_w=\varnothing$. Also recall that $F_q\cap
F_w=\varnothing$. Then either $F_t\cap F_u\ne\varnothing$, or $F_u\cap
F_w\ne\varnothing$. In the first case $(F_t, F_u,F_v,F_w)$ is a simple
$4$-loop with $F_t\cap F_v=\varnothing$. Then Corollary~\ref{4-belt-ful}
implies $F_u\cap F_w\ne\varnothing$. Then $F_u\cap F_w\cap F_v$ and $F_u\cap
F_w\cap F_t$ are vertices and we obtain Fig.~\ref{131313-lem-fig}b). In the
second case we obtain the same picture by the symmetry.

If $F_t\cap F_w=\varnothing$, then $F_p$ is a hexagon.

By the symmetry considering the facets $F_q$ and $F_v$ we obtain that either
all of them are pentagons, and we obtain Fig.~~\ref{131313-lem-fig}b), or all
of them are hexagons, and we obtain Fig.~~\ref{131313-lem-fig}c). Since
$F_p\cap F_q=F_q\cap F_v=F_v\cap F_w$, the boundary component of $|\L_3|$
intersecting $F_t$ is a simple edge-cycle. Since both sets $\{F_j,F_l,F_s\}$
and $\{F_q,F_v,F_p\}$ consist of pairwise different facets and these sets
belong to different connected components with respect to $\gamma$, the
$6$-loop $\L_3$ is simple. It has $(1,3,1,3,1,3)$ edges on the boundary
component intersecting $F_t$.
\end{proof}
Let $P$ contain the fragment on Fig.~~\ref{131313-lem-fig}b). Consider the
$6$-loop $\L=(F_p,F_t,F_q,F_u,F_v,F_w)$. The facets $\L\setminus\{F_v\}$ are
pairwise different, since four of them surround the fifth. The same if for
$\L\setminus\{F_p\}$ and $\L\setminus\{F_q\}$. Any two facets of $\L$ belong
to one of these sets, therefore $\L$ is a simple loop. $F_p\cap
F_q=\varnothing$, else $F_p\cap F_q\cap F_t$ is a vertex and $F_t$ is a
quadrangle. Similarly $F_q\cap F_v=F_v\cap F_p=\varnothing$. Then
$\partial|L|$ is a simple edge-cycle, and $\L$ has $(1,3,1,3,1,3)$ edges on
it. Applying Lemma \ref{131313-lemma} to $\L$ we obtain the proof of the
theorem.
\end{proof}
\section{$(s,k)$-truncations and straightenings along edges}\label{Trunc}
A natural generalization of operations drawn on Fig.~\ref{3cut} is an
operation of cutting off a sequence of successive edges of some facet (see
Fig.~\ref{SKcut}). This construction appears in other notations in
\cite{V12}. It is studied in details in \cite{BE15}.
\begin{figure}[h]
\includegraphics[height=3cm]{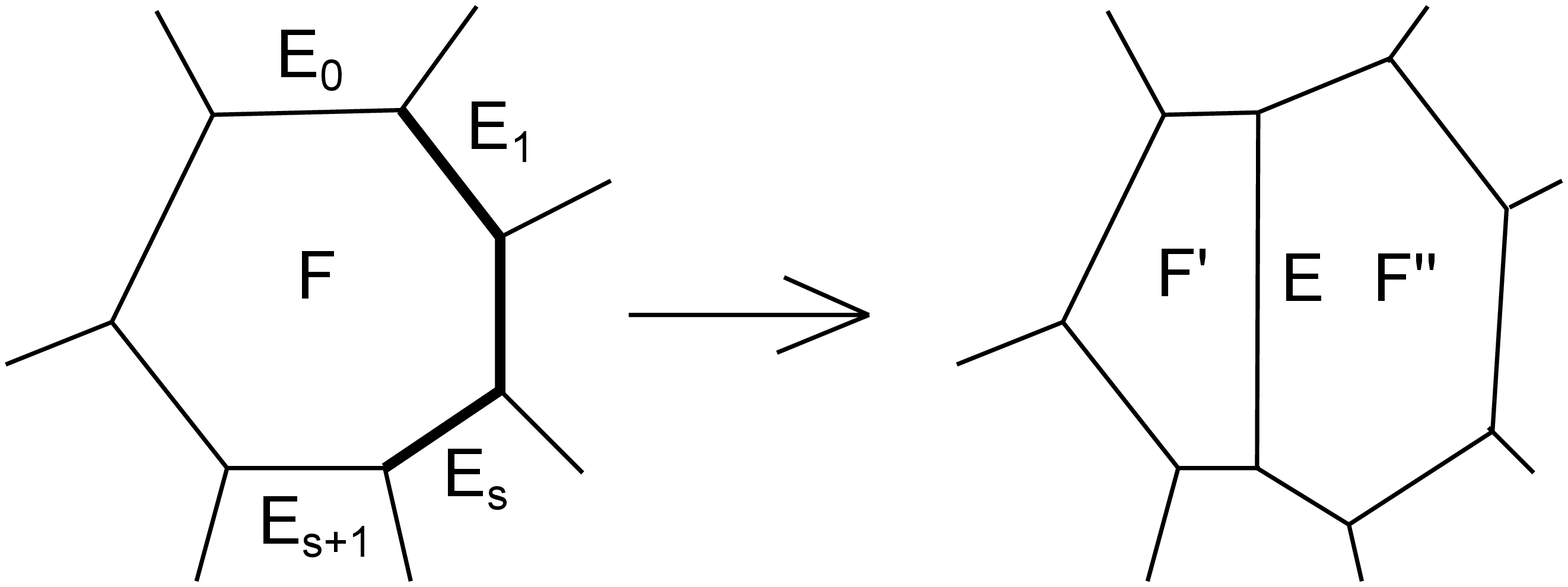}
\caption{$(3,7)$-truncation}\label{SKcut}
\end{figure}

\begin{constr}[$(s,k)$-truncation]
Given $s+2$ successive edges $(E_0,E_1,\dots,E_s,E_{s+1})$ of a $k$-gonal
facet $F$ of a combinatorial simple $3$-polytope $P$, $0\leqslant s\leqslant
k-2$ consider the edges $E_0,E_{s+1}$ of $F$ in some geometric realization of
$P$. Rotate the plane ${\rm aff} F$ around the axis passing through the
midpoints of $E_0$ and $E_{s+1}$ to obtain chosen $(E_1,\dots,E_s)$ lying in
one halfspace, and other vertices of $P$  -- in the other. Then the
intersection of the second halfspace with $P$ gives a resulting polytope $Q$.
This operation is correctly defined on combinatorial types of polytopes. The
graph $G(Q)$ is obtained from $G(P)$ by dividing the edges $E_0$ and
$E_{s+1}$ by midpoints and adding a new edge $E$ connecting the midpoints. We
call this operation an \emph{$(s,k)$-truncation} (see Fig.~\ref{SKcut})

If the facets intersecting $F$ by $E_0$ and $E_{s+1}$ have $t_0$ and
$t_{s+1}$ edges respectively, then we also call the corresponding operation
an \emph{$(s,k;t_0,t_{s+1})$-truncation}.

For $s=1$, let $E_1=F\cap F_{i_1}$ be the edge we cut off and $F_{i_1}$ be a
$t_1$-gon. Then combinatorially $(1,k;t_0,t_2)$-truncation is the same
operation as $(1,t_1;t_0,t_2)$-truncation of the same edge of the facet
$F_{i_1}$. Will call this operation simply a \emph{$(1;t_0,t_2)$-truncation}.

\begin{figure}[h]
\includegraphics[scale=0.3]{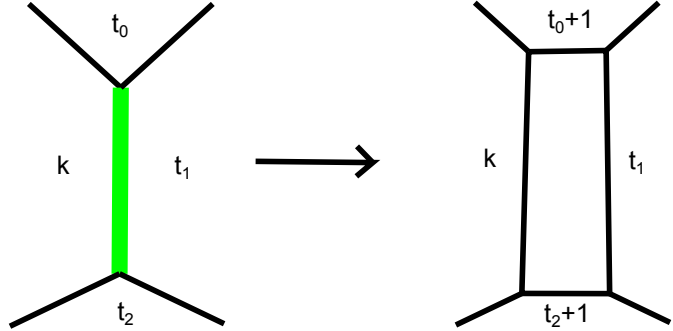}
\caption{$(1;s_1,s_2)$-truncation}\label{1-cut}
\end{figure}

\end{constr}
\begin{exam}
$(0,k)$-truncation is a vertex cut, $(1,k)$-truncations is an edge cut, and
$(2,k)$-truncation is a cutting off of two incident edges (see.
Fig.~\ref{3cut}).
\end{exam}
Let us describe the change of the face lattice under $(s,k)$-truncation.
\begin{utv}
Under $(s,k)$-truncation of a polytope $P$
\begin{enumerate}
\item the facets that do not contain $E_0$ or $E_{s+1}$, do not change
    the combinatorial type;
\item the facet $F$ is divided into two facets: $(s+3)$-gon $F'$ and
    $(k-s+1)$-gon $F''$, $F'\cap F''=E$;
\item   the number of edges of the facets intersecting $F$ by $E_0$ and
    $E_{s+1}$ increases by one.
\end{enumerate}
\end{utv}
It is easy to see that combinatorially the $(k-s-2,k)$-truncation along edges
in $F$ complementary to $(E_0,E_1,\dots,E_s,E_{s+1})$ is the same operation
as $(s,k)$-truncation.

The next result is contained in \cite{V12}. See the proof in \cite{BE15}.
\begin{lemma}
Let $P$ be a flag $3$-polytope. Then the polytope obtained by an
$(s,k)$-truncation is flag if and only if  $0<s<k-2$.
\end{lemma}

Now we define an operation that is combinatorially inverse to the
$(s,k)$-truncation. This operation corresponds to an edge contraction of
simplicial polytopes from \cite{V12,V15}.
\begin{figure}[h]
\includegraphics[height=4cm]{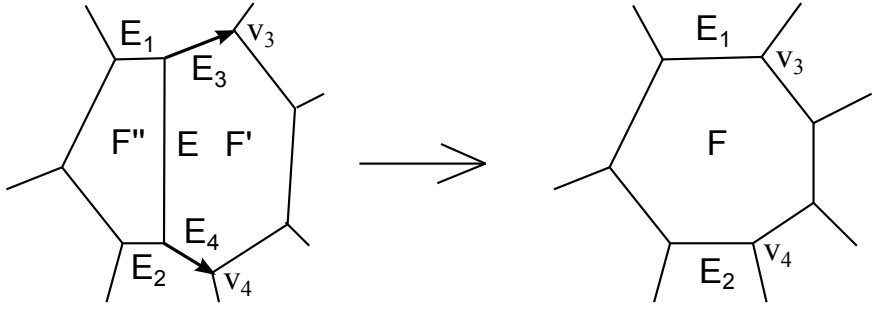}
\caption{Straightening along the edge}\label{Unfold}
\end{figure}
\begin{constr}[Straightening along the edge]
Let $P$ be a combinatorial simple polytope, and let $E=F'\cap F''$ be an
edge, where $F'$ is an $(s+3)$-gon and $F''$ is an $(k-s+1)$-gon, $k\geqslant
3$, $0\leqslant s\leqslant k-2$. If a geometrical realization of $P$ is
obtained from a geometrical realization of a combinatorial polytope $Q$ by an
$(s,k)$-truncation such that the facet $F\subset Q$ is divided into $F'$ and
$F''$, then we say that $Q$ is obtained from $P$ by \emph{straightening along
the edge $E$}. If the operation is defined, then the combinatorial type of
$Q$ is defined uniquely.  The graph is changed in the following way. Let in
$G(P)$ the edge $E$ is incident to $E_1,E_2\subset F''$ and $E_3,E_4\subset
F'$. Then $G(Q)$ is obtained from $G(P)$ by deletion of $E$ and contraction
of $E_3$ and $E_4$ (see Fig.~\ref{Unfold}).
\end{constr}
Let us describe the change of facet lattice under the straightening along the
edge.
\begin{utv} Under the straightening along the edge of $P$
\begin{enumerate}
\item the combinatorics of facets that do not intersect $E$ does not
    change;
\item  The facets  $F'$  $F''$, containing $E$, are united into the facet
    $F$.
\item the number of edges of facets intersecting $E$ by a vertex
    decreases by one.
\end{enumerate}
\end{utv}

Let $P$ be a simple polytope. Let the edge $E=F_i\cap F_j$ connects the
vertices $F_i\cap F_j\cap F_p$ and $F_i\cap F_j\cap F_q$, and the facet $F_i$
is incident to $F_{i_1},\dots, F_{i_s}$, and $F_j$ -- to $F_{j_1},\dots,
F_{j_t}$ (Fig.~\ref{UF1}). The following result is proved in \cite{BE15} on
the base of the Steinitz Theorem for $3$-polytopes.
\begin{lemma}[\cite{BE15}]\label{ras}
Let $P\not\simeq\Delta^3$ be a simple polytope. The operation of
straightening along $E=F_i\cap F_j$ is defined if and only if
$\{F_{i_1},\dots, F_{i_s}\}\cap \{F_{j_1},\dots,F_{j_t}\}=\varnothing$.
\end{lemma}
\begin{cor}\label{cor-3-s}
For $P\simeq\Delta^3$ no straightening operations are defined. Let
$P\not\simeq \Delta^3$ be a simple polytope. The operation of straightening
along $E=F_i\cap F_j$ is defined if and only there is a $3$-belt
$(F_i,F_j,F_k)$ for some $F_k$.
\end{cor}
\begin{proof}
The straightening operation decreases the number of facets, therefore it is
not defined for a simplex. By Lemma \ref{ras} the straightening along $E$ is
not defined if and only $F_{i_a}=F_{j_b}$ for some $a,b$. This is equivalent
to the fact that $(F_i,F_j,F_{i_a})$ is a $3$-belt.
\end{proof}
\begin{figure}[h]
\includegraphics[height=4cm]{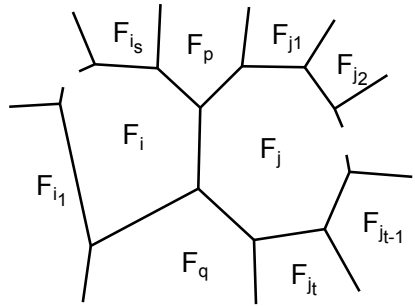}
\caption{A polytope in the neighborhood of $F_i\cap F_j$}\label{UF1}
\end{figure}
\begin{cor}\label{cor-flag}
A simple $3$-polytope is flag if and only of it can be straightened along any
edge.
\end{cor}
\begin{cor}\label{Flag-ij}
For a flag polytope $P$ two adjacent facets and all the facets incident to
them (Fig.~\ref{UF1}) are pairwise different.
\end{cor}
\begin{proof}
All the facets except for $\{F_{i_1},\dots, F_{i_s}\}$ are pairwise
different, since any two of them either intersect, or intersect $F_j$ by
different edges. The same is for $\{F_{j_1},\dots, F_{j_t}\}$. Therefore if
two facets on the fragment coincide, then one of them if $F_{i_a}$, an the
other is $F_{j_b}$. But they are different by Corollary \ref{cor-flag} and
Lemma \ref{ras}.
\end{proof}
This result and the following Lemma and Theorem for simplicial $3$-polytopes
and an~edge contraction operation were proved in \cite{V12,V15}.
\begin{lemma}[\cite{BE15}]\label{Uflag}
The polytope $Q$ obtained by straightening the flag polytope $P$ along the
edge $F_i\cap F_j$ is flag if and only if there is a $4$-belt
$(F_i,F_j,F_k,F_l)$.
\end{lemma}
\begin{theorem}[\cite{V12,V15}]
Any flag combinatorial simple $3$-polytope can be moved to a cube $I^3$ by a
sequence of straightenings along edges in such a way that all intermediate
polytopes are flag.
\end{theorem}
This theorem was straightened in \cite{BE15}.
\begin{theorem}[\cite{BE15}]\label{flag-theorem}
A simple $3$-polytope $P$ is flag if and only if it is combinatorially
equivalent to a polytope obtained from the cube by a sequence of
$(1,k)$-truncations, $k\geqslant 4$, and $(2,k)$-truncation, $k\geqslant 6$.
\end{theorem}

\begin{utv}\label{utv-unfold}
Let $P$ be a simple polytope with all facets pentagons and hexagons with at
most one exceptional facet $F$ being a quadrangle or a heptagon. Then it can
be straightened along any edge that does not intersect the quadrangular facet
by a vertex to obtain a flag polytope.
\end{utv}
\begin{proof}
This follows from Theorems \ref{3-belt-theorem} and \ref{4-belt-theorem},
Corollary  \ref{cor-3-s} and Lemma \ref{Uflag}.
\end{proof}

\section{Realization of fullerenes}\label{Main-Th}
It turned out that for fullerenes Theorem \ref{flag-theorem} can be
strengthened.

\begin{theorem}\label{C-F-theorem} Let $P$ be combinatorial fullerene.
\begin{enumerate}
\item $P$ can be obtained from the dodecahedron by a sequence of the
    following operations:
\begin{enumerate}
\item  An operation drawn on Fig.~\ref{OP-12}a). This is a movement
    from $k$-th to $(k+1)$-th fullerene in Family I.
\item  An operation drawn on Fig.~\ref{OP-12}b). This is a movement
    from $k$-th to $(k+1)$-th fullerene in Family II.
\begin{figure}[h]
\begin{tabular}{ccc}
\includegraphics[height=2cm]{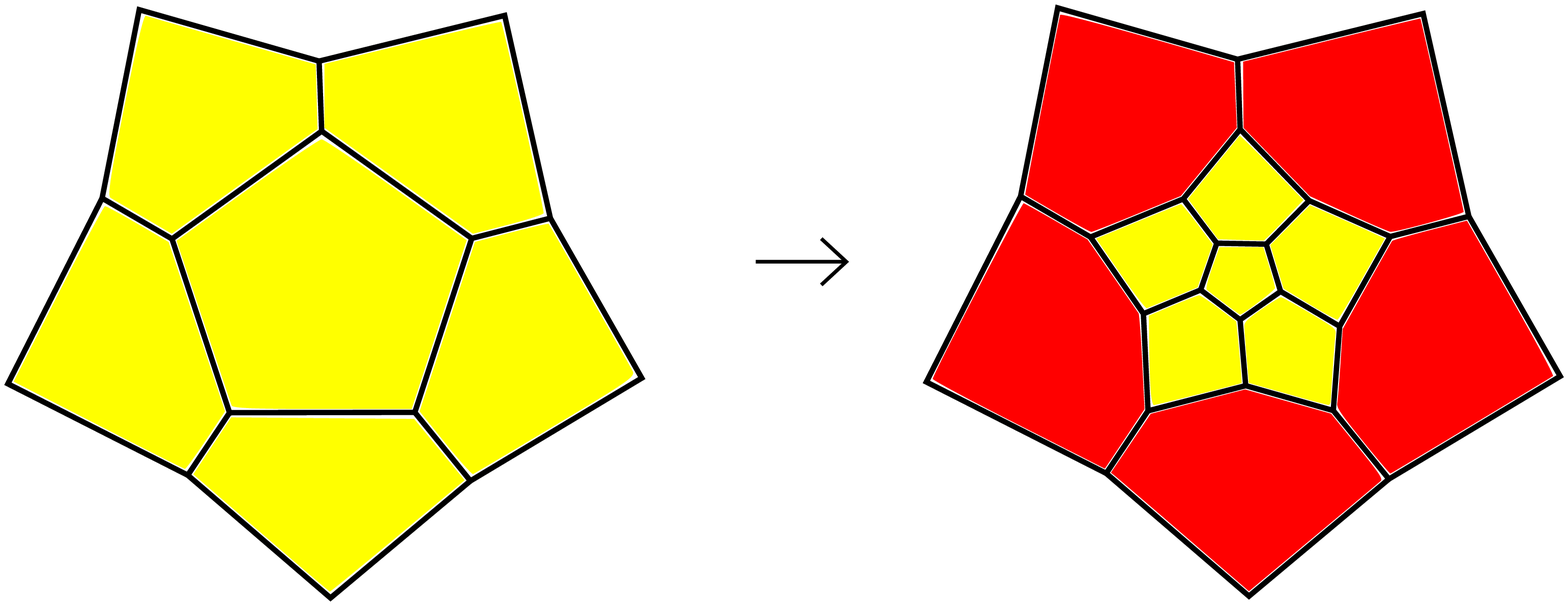}&&\includegraphics[height=2cm]{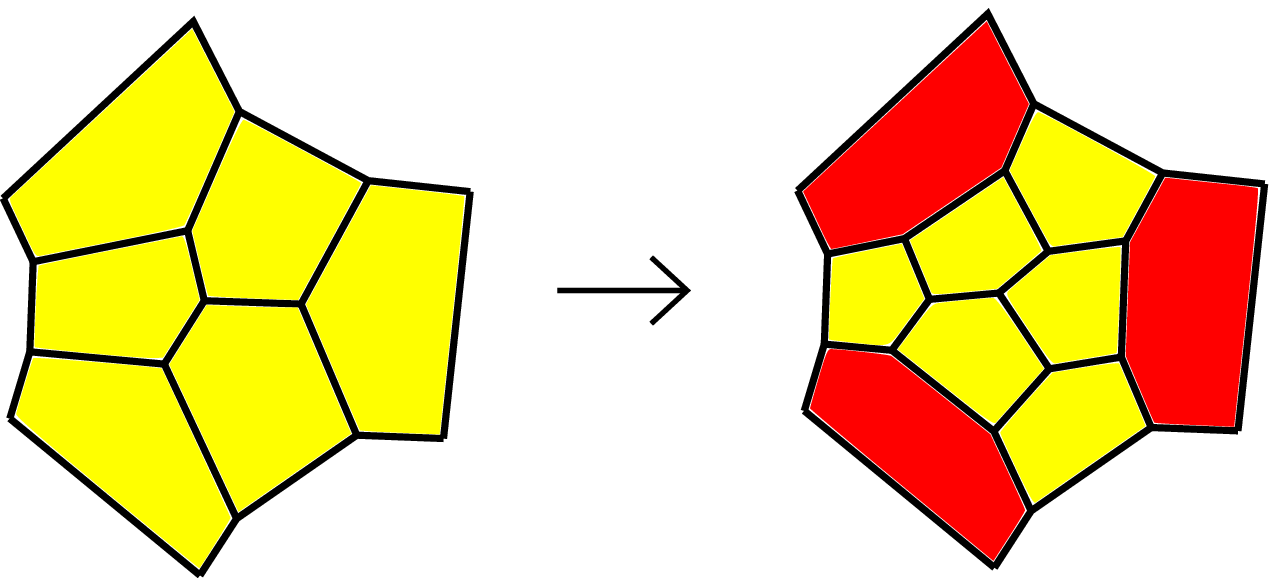}\\
a)&&b)
\end{tabular}
\caption{}\label{OP-12}
\end{figure}

\item The Endo-Kroto operation drawn on Fig.~\ref{OP-34}a).
\item An operation drawn on Fig.~\ref{OP-34}b).
\begin{figure}[h]
\begin{tabular}{ccc}
\includegraphics[height=2cm]{EK.eps}&&\includegraphics[height=2cm]{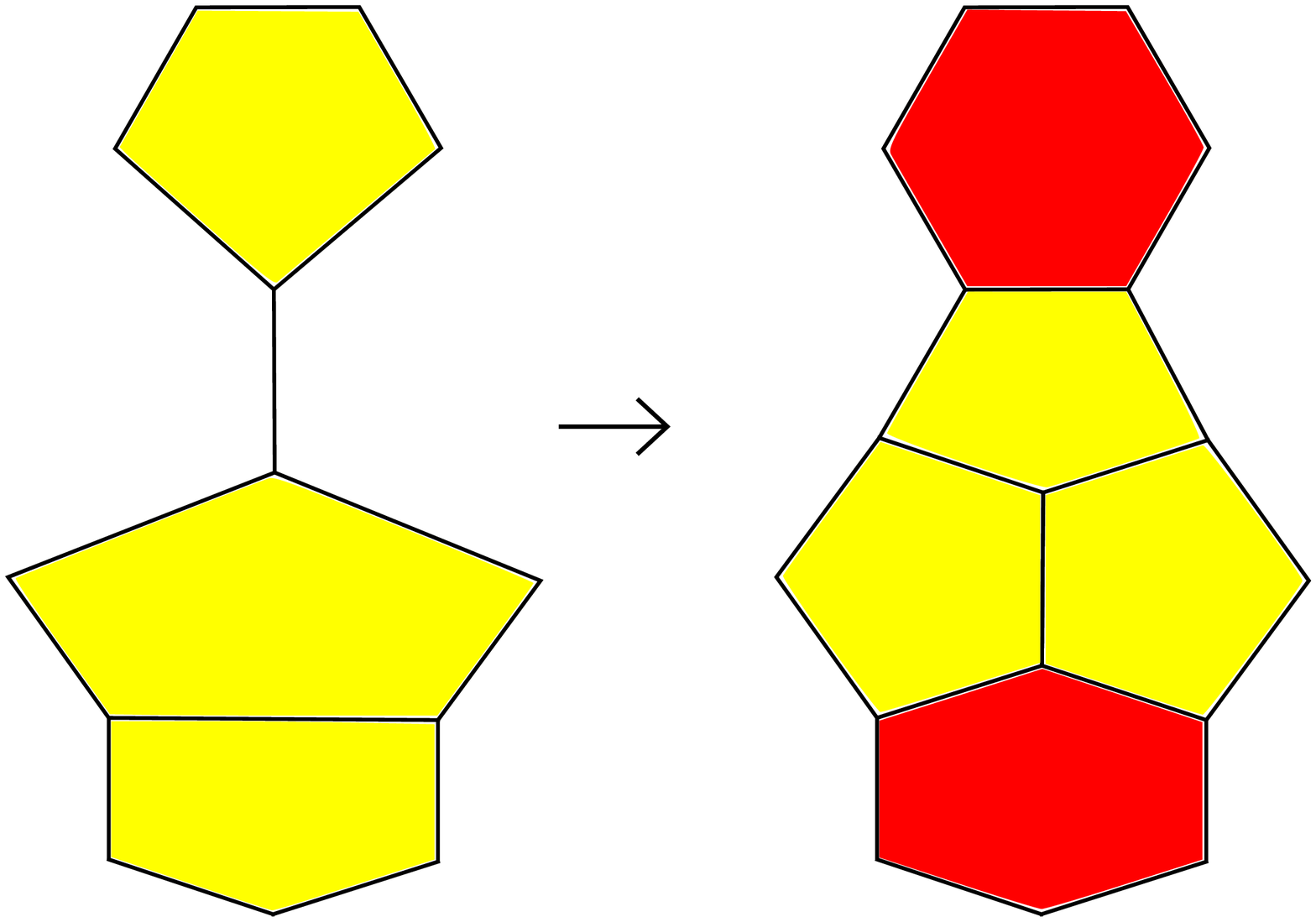}\\
a)&&b)
\end{tabular}
\caption{}\label{OP-34}
\end{figure}
\item  An operation drawn on Fig.~\ref{OP-5}.
\begin{figure}[h]
\begin{center}
\includegraphics[height=2cm]{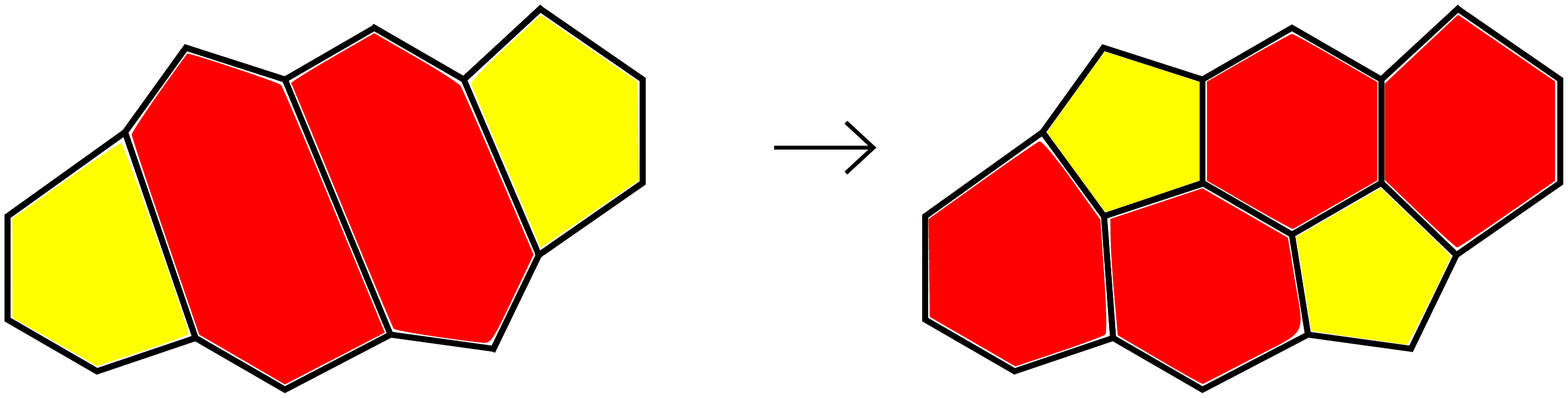}\\
\end{center}
\caption{}\label{OP-5}
\end{figure}
\item An operation drawn on Fig.~\ref{OP-67}a). On the left there can
    be an arbitrary number $k\geqslant 3$ of hexagons;
\item An operation drawn on Fig.~\ref{OP-67}b). On the left there can
    be arbitrary number $k_1,k_2\geqslant 1$,
    $\max\{k_1,k_2\}\geqslant 2$, of hexagons before and after the
    corner hexagon.
\begin{figure}[h]
\begin{tabular}{cc}
\includegraphics[height=1.2cm]{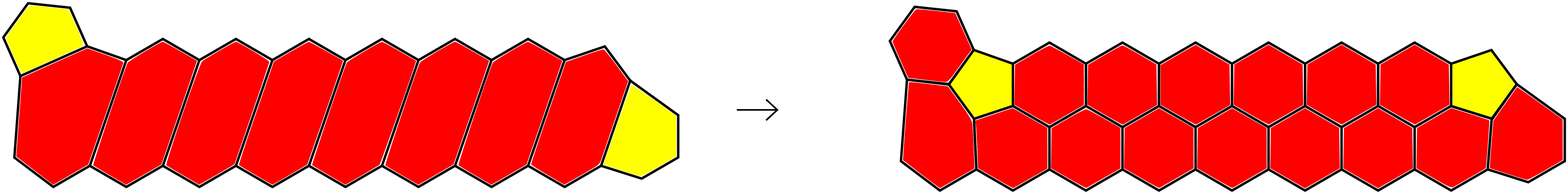}&\includegraphics[height=1.8cm]{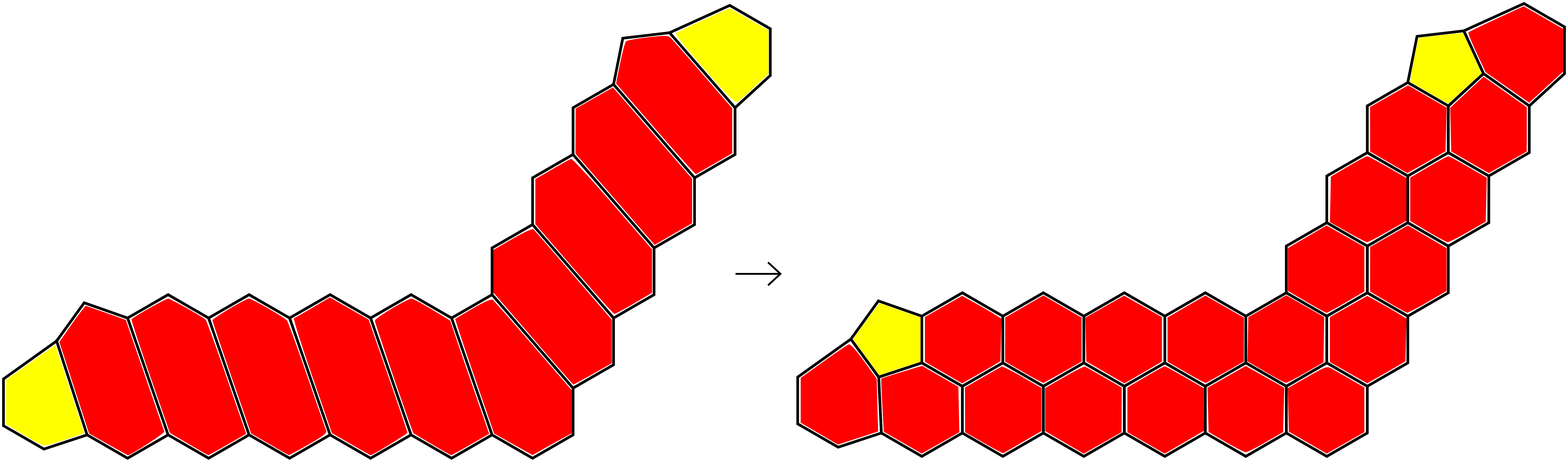}\\
a)&b)
\end{tabular}
\caption{}\label{OP-67}
\end{figure}
\end{enumerate}
\item Any of the operations in (1) is decomposed as a sequence of
    $(1;4,5)$-, $(1;5,5)$-, $(2,6;4,5)$-, $(2,6;5,5)$-, $(2,6;5,6)$-,
    $(2,7;5,5)$-, and $(2,7;5,6)$-truncations. All the intermediate
    polytopes have facets pentagons, hexagons, and no more than one
    exceptional facet, which is either a quadrangle, or a heptagon.
\item Any of the operations in (1) is invertible  in the following sense:
    if a combinatorial fullerene $P$ contains a fragment on the right,
    then this fragment appears after an application of the corresponding
    operation to some fullerene $Q$, which gives the fullerene $P$
\end{enumerate}
\end{theorem}
\begin{proof} The proof consists of two parts.
\begin{enumerate}
\item First we prove that any combinatorial fullerene except for the
    dodecahedron contains one of the fragments on the right.
\item Then we prove that any fragment on the right can be moved to the
    fragment on the left by a sequence of straightenings along edges,
    which will be defined due to Proposition \ref{utv-unfold}.
\end{enumerate}
Part (1).
\begin{lemma}\label{NIPR-lemma}
 Let $P$ be a fullerene. If $P$ is not an $IPR$-fullerene, then
it contains one of the following fragments drawn on Fig.~\ref{NIPR-F} with
all the facets on the fragment being pairwise different.
\end{lemma}
\begin{figure}[h]
\begin{tabular}{cccc}
\includegraphics[height=3cm]{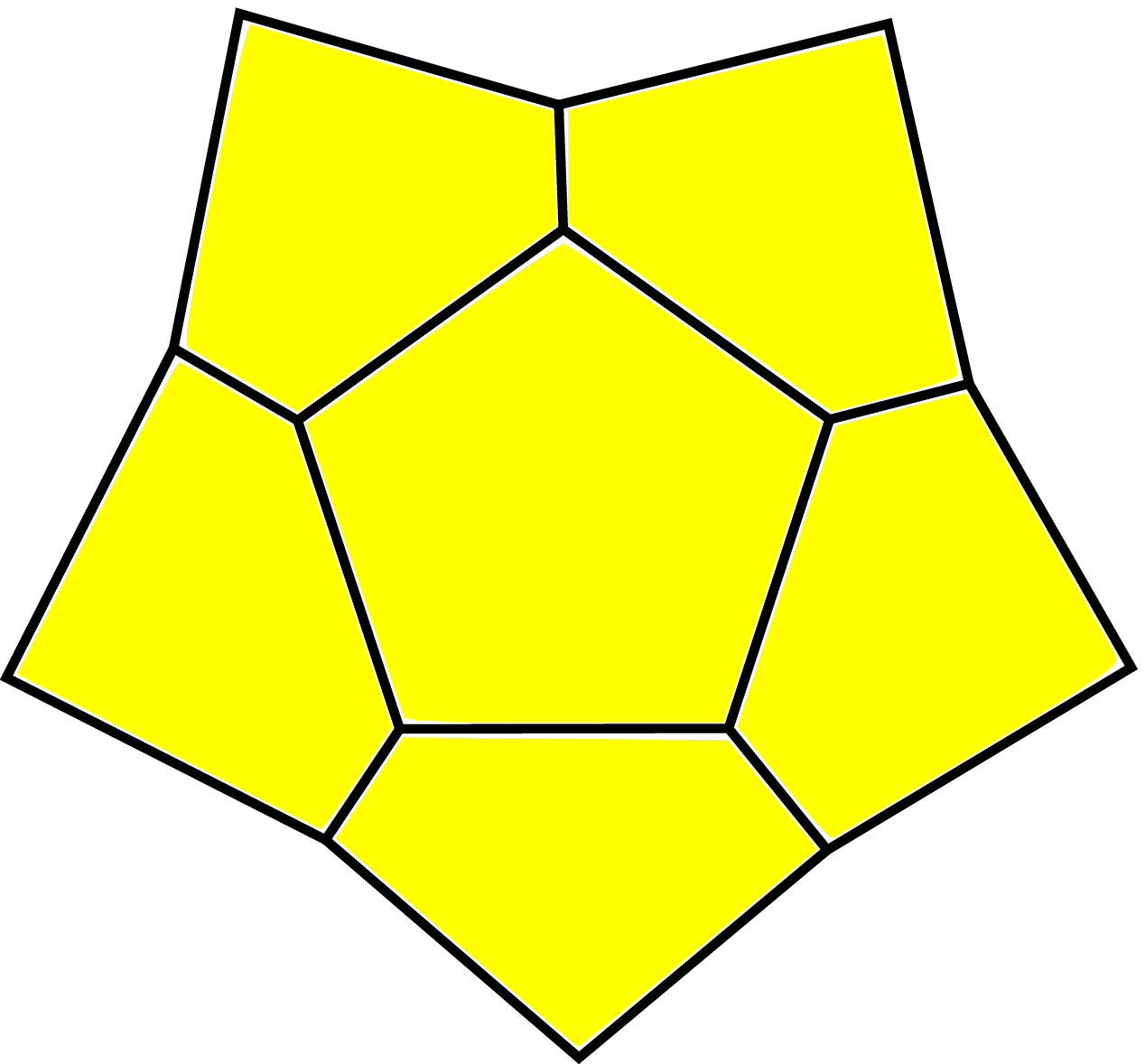}&\includegraphics[height=3cm]{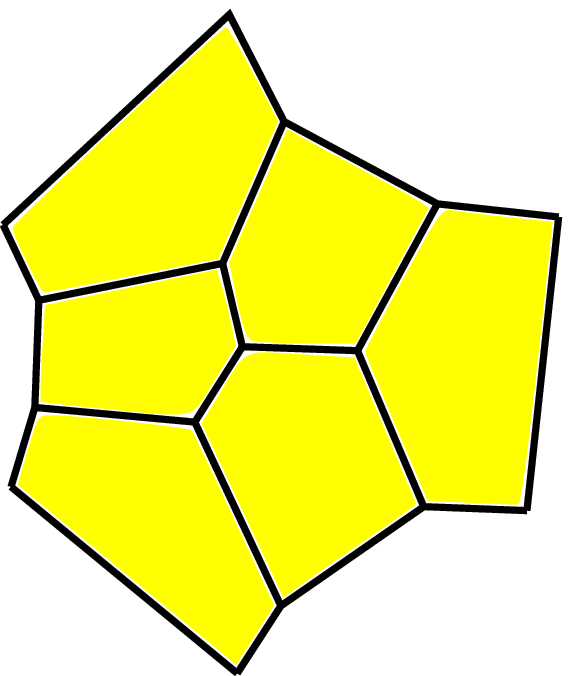}&\includegraphics[height=3cm]{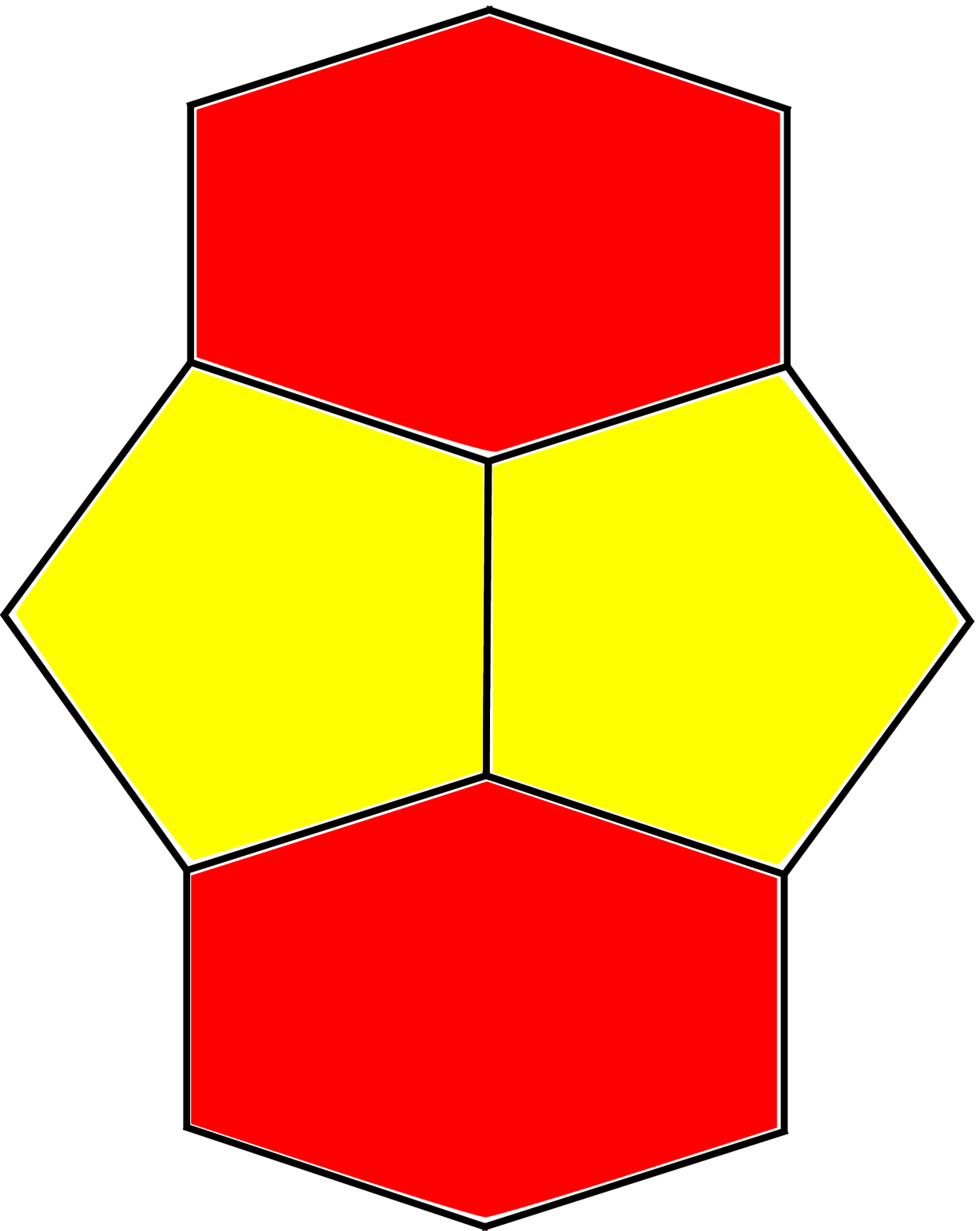}&\includegraphics[height=3cm]{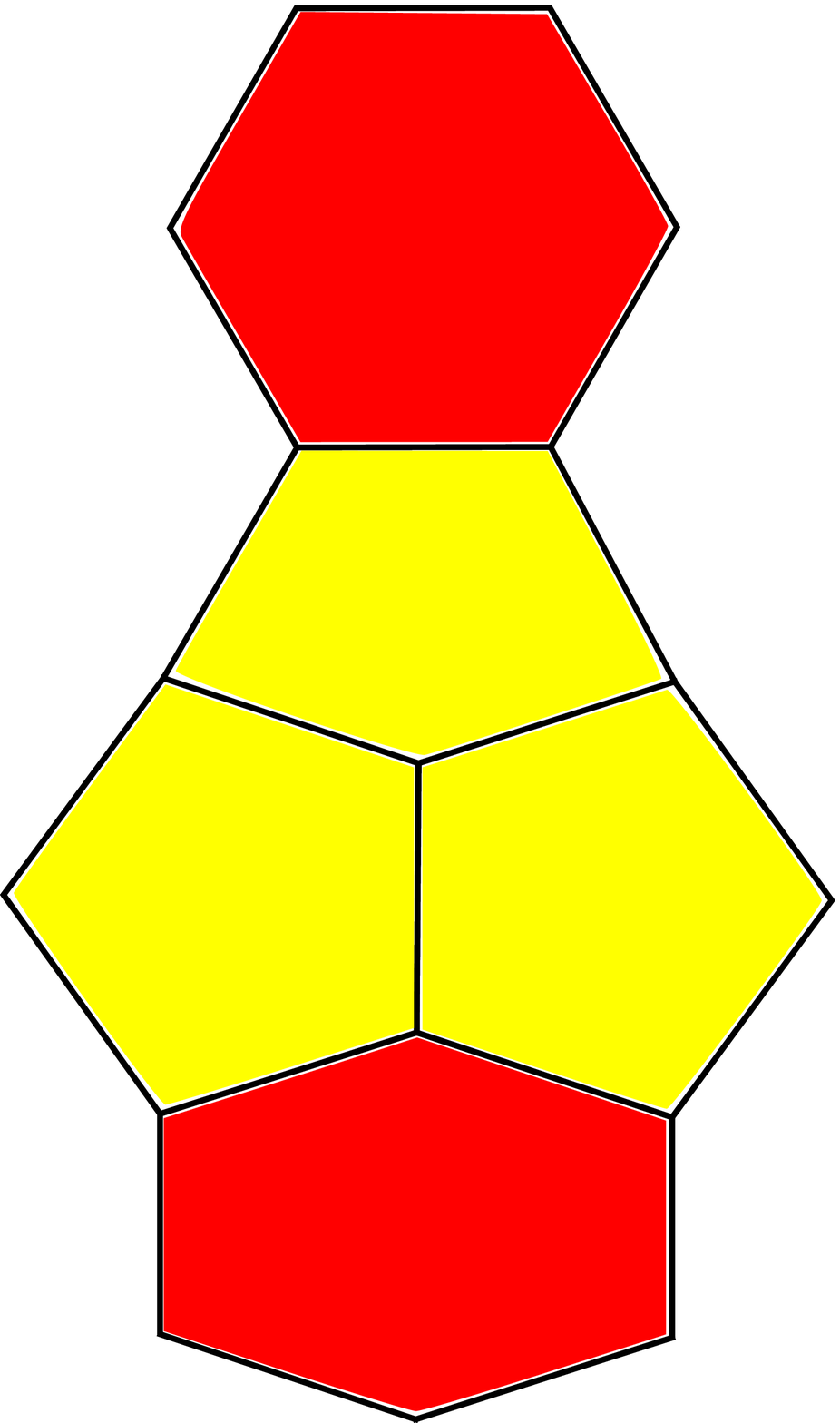}\\
a)&b)&c)&d
\end{tabular}
\caption{}\label{NIPR-F}
\end{figure}
\begin{proof}
The facets on each fragment are pairwise different by Corollary
\ref{Flag-ij}, since they belong to a set of two adjacent facets and facets
surrounding them.

By assumption $P$ contains two incident pentagons: $F_i$ and $F_j$, $F_i\cap
F_j\ne\varnothing$. Consider the facets $F_k$, $F_l$ such that $F_i\cap
F_j\cap F_k$ and $F_i\cap F_j\cap F_l$ are vertices. If both facets $F_k$,
$F_l$ are hexagons, then $P$ contains the fragment c). Else one of them is a
pentagon. Without loss of generality let it be $F_k$. We obtain three
pentagons with a common vertex (Fig.~\ref{NIPR-Pic}a).

If $F_p$, $F_q$, $F_l$ are pentagons, we obtain the fragment b).

Let two of them, say $F_p$ and $F_q$ be hexagons. Then either $F_u$ is a
pentagon and we obtain the fragment c), or $F_u$ is a hexagon
(Fig.~\ref{NIPR-Pic}b). If $F_v$, $F_l$, or $F_w$ is a hexagon, we obtain the
fragment d). If all of them are pentagons, we obtain the fragment b).

Now let one of the facets $F_p$, $F_q$, $F_l$, say $F_l$ be a hexagon, and
two others -- pentagons (Fig.~\ref{NIPR-Pic}c). Then either $F_u$ is a
hexagon and we obtain the fragment d), or it is a pentagon and we obtain the
fragment a).
\begin{figure}[h]
\begin{tabular}{ccc}
\includegraphics[scale=0.2]{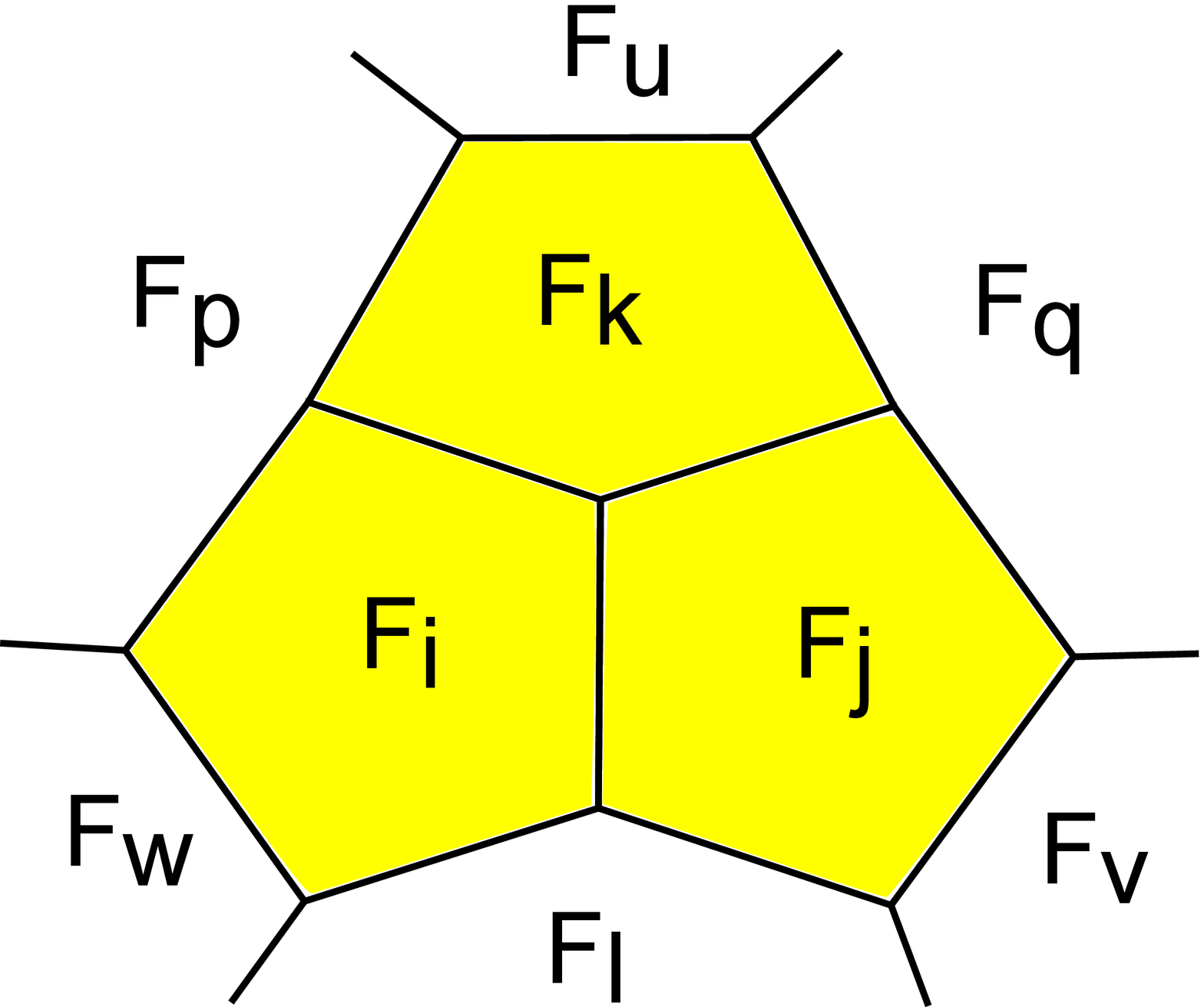}&\includegraphics[scale=0.15]{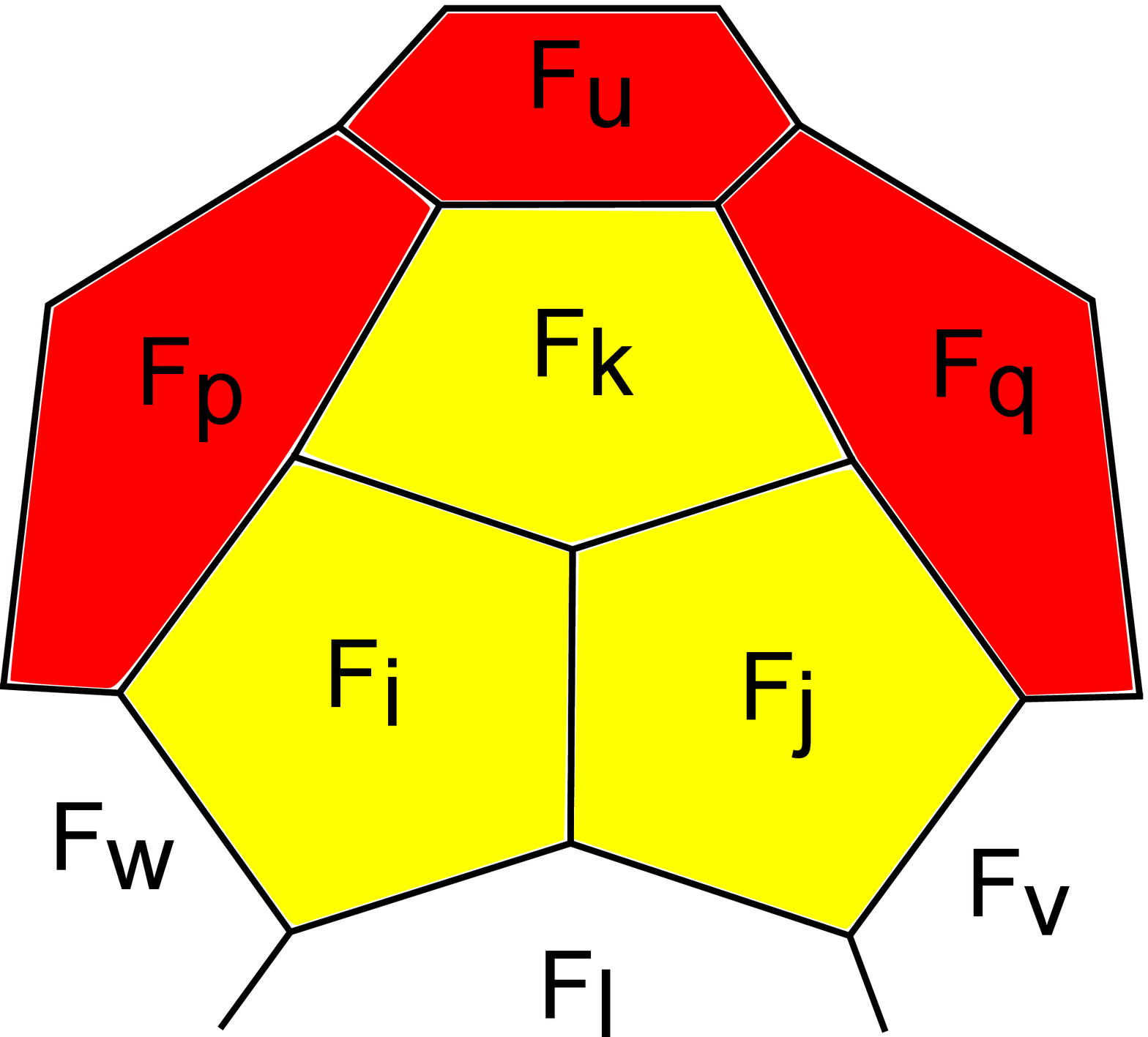}&\includegraphics[scale=0.15]{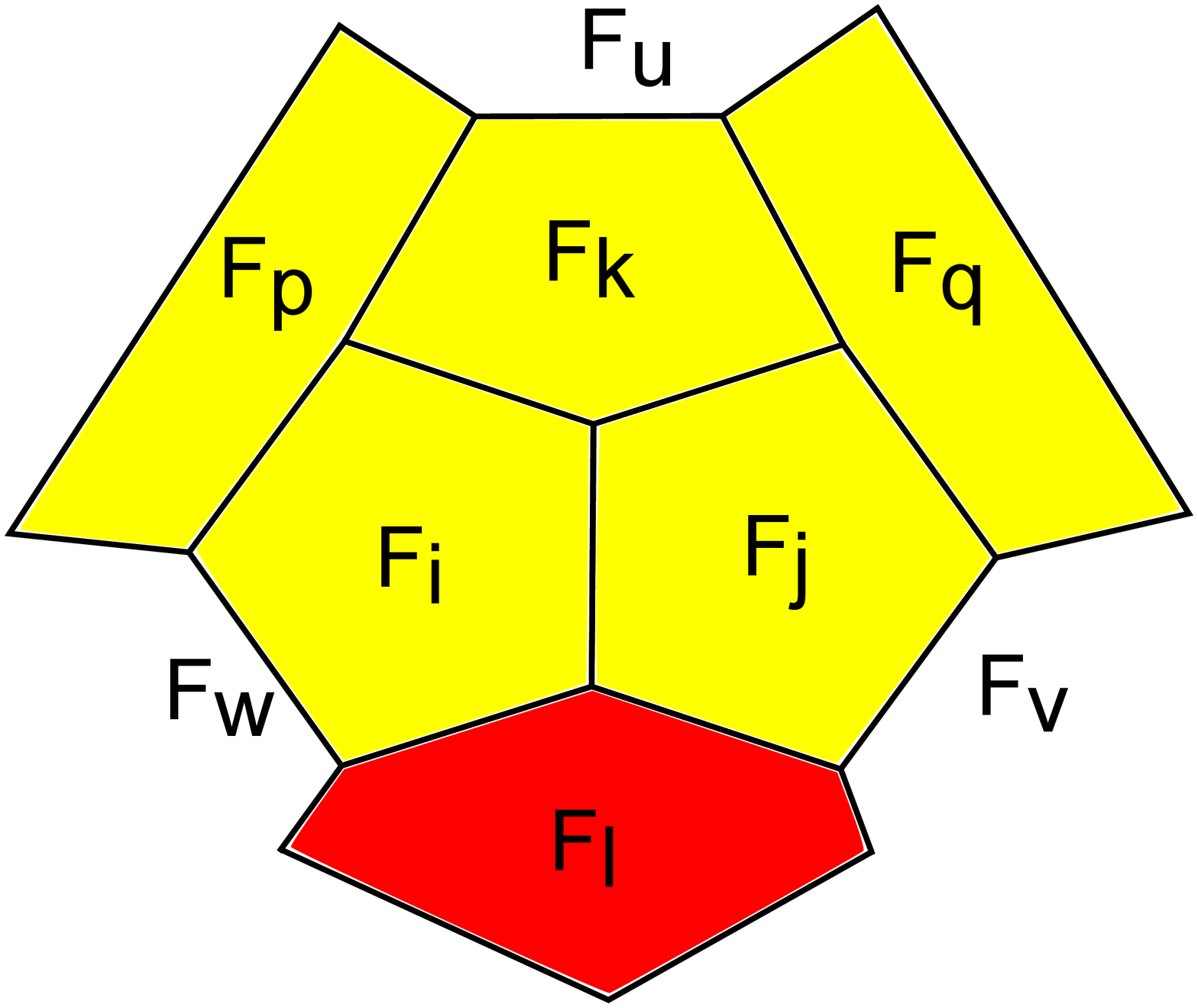}\\
a)&b)&c)
\end{tabular}
\caption{}\label{NIPR-Pic}
\end{figure}

\end{proof}

\begin{lemma}\label{IPR-lemma}
Let $P$ be an $IPR$-fullerene. Then it contains one of the fragments drawn on
Fig.~\ref{Roads-F} with all the facets on the fragment being pairwise
different. Here
\begin{enumerate}
\item on a) there can be arbitrary $k\geqslant 1$ hexagons between
    pentagons;
\item on b) there can be arbitrary $k_1,k_2\geqslant 0$,
    $\max\{k_1,k_2\}\geqslant 1$ hexagons between the left pentagon and
    the red hexagon, and between the red hexagon and the right pentagon.
\end{enumerate}
\end{lemma}
\begin{figure}[h]
\begin{tabular}{ccc}
\includegraphics[scale=0.2]{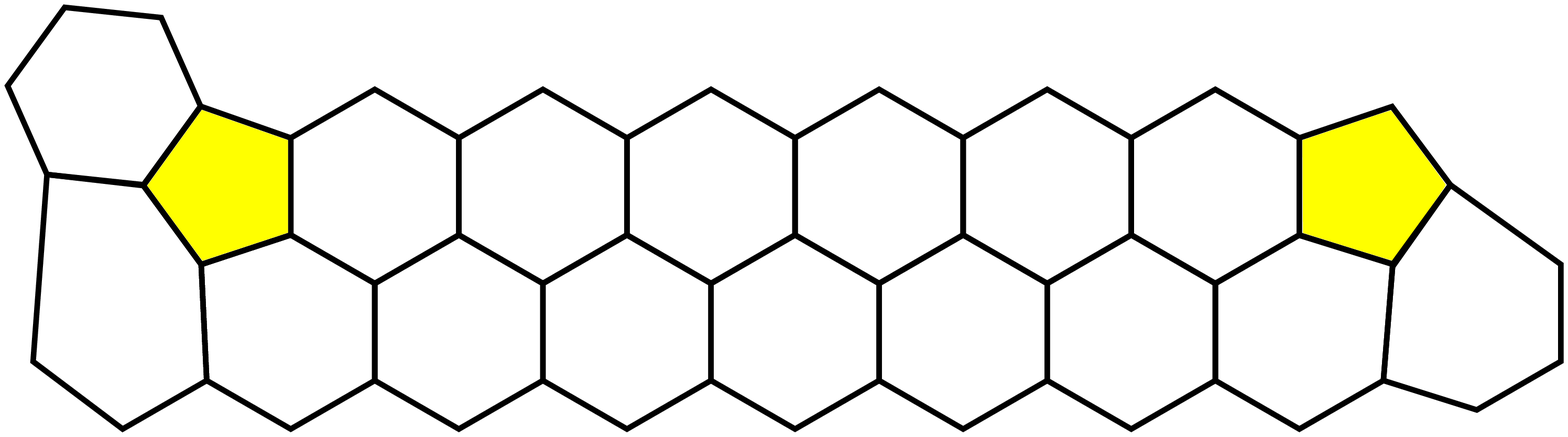}&\includegraphics[scale=0.2]{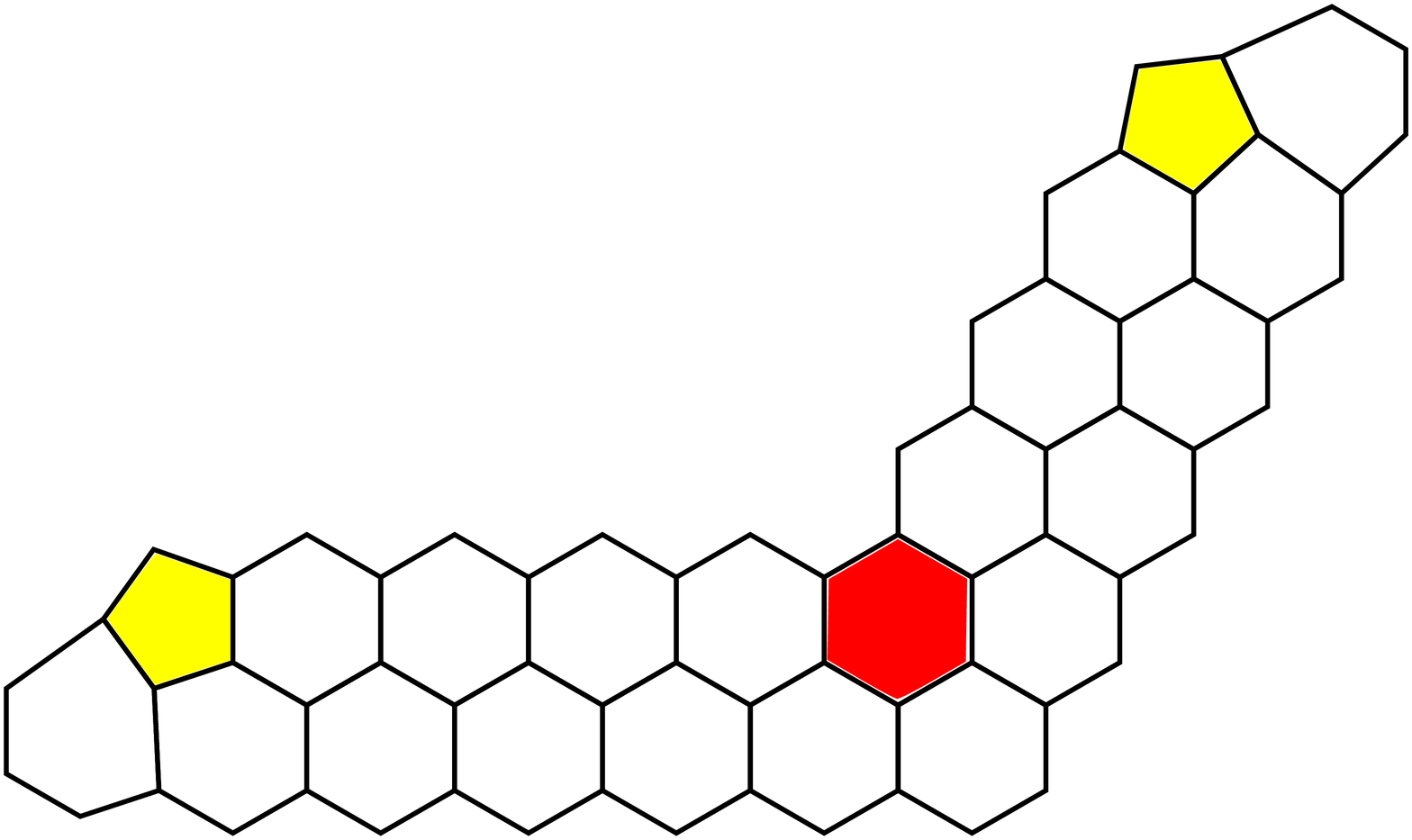}&\includegraphics[scale=0.25]{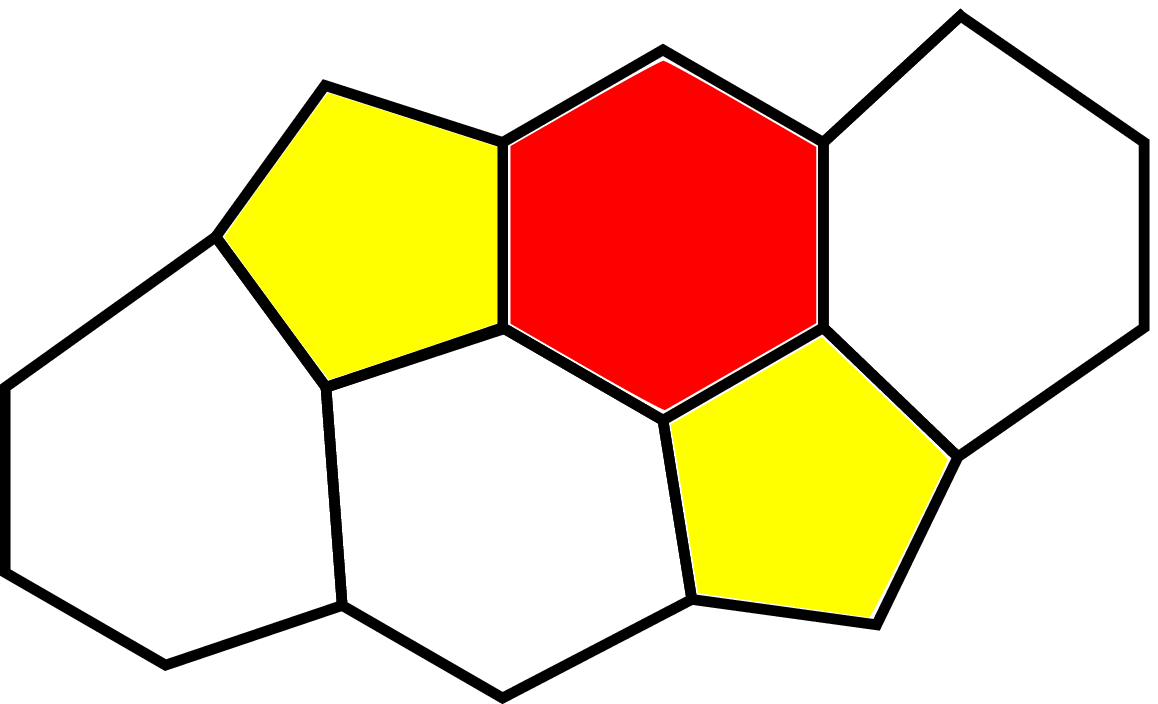}\\
a)&b)&c)
\end{tabular}
\caption{}\label{Roads-F}
\end{figure}

\begin{proof}
Consider the shortest thick path $\mathcal{P}=(F_{i_0},F_{i_1},\dots,
F_{i_{s+1}})$ connecting two different pentagons in $P$. Then $s\geqslant 1$.

\begin{figure}[h]
\begin{center}
\includegraphics[scale=0.2]{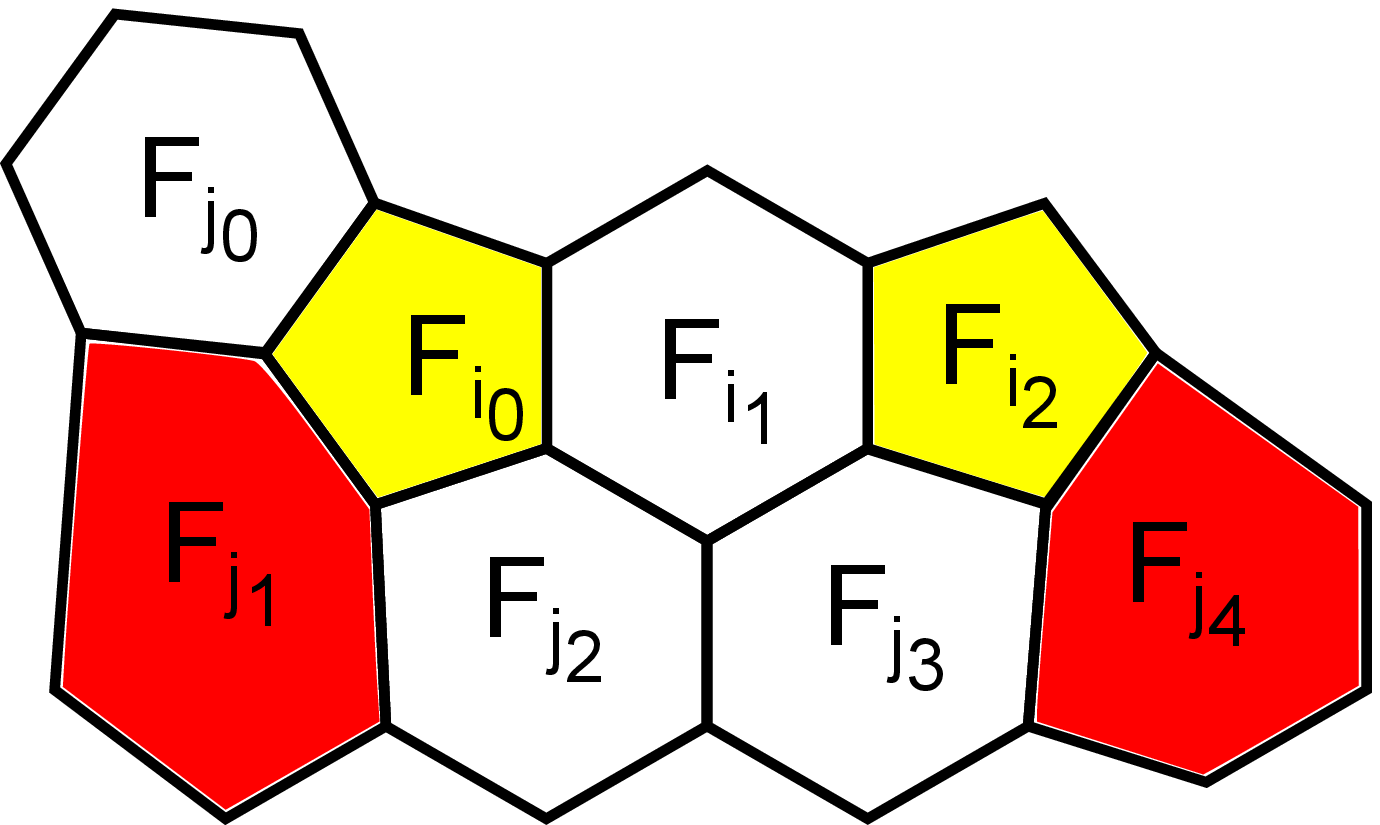}
\end{center}
\caption{}\label{s-1}
\end{figure}
Let $s=1$. Then we obtain one of the fragment a) and c). On the fragment c)
all the facets are pairwise different by Corollary \ref{Flag-ij}. The
fragments a) is drawn on Fig.~\ref{s-1}. By Corollary \ref{Flag-ij} the
facets $\{F_{i_0},F_{i_1}\}$ and the facets surrounding them are pairwise
different. The same is true for $\{F_{i_1},F_{i_2}\}$. Then if two facets on
the fragment coincide, then this is a pair of red hexagons. This is
impossible, since they belong to the facets surrounding two hexagons
$F_{j_2}$ and $F_{j_3}$.

Now let $s\geqslant 2$. Then we have.

(1) $F_{i_0}\ne F_{i_{s+1}}$ are pentagons and facets in
$\mathcal{H}_6=\{F_{i_j},1\leqslant j\leqslant s\}$ are hexagons, else there
is a shorter thick path than $\P$.

(2) $F_{i_1}$ intersects no pentagons except for $F_{i_0}$, $F_{i_k}$
intersects no pentagons except for $F_{i_{k+1}}$, and
$F_{i_2},\dots,F_{i_{k-1}}$ intersect no pentagons.  Else there is a thick
path shorter than~$\P$.

(3) All the facets in $\mathcal{H}_6$ are pairwise different, else if
$F_{i_k}=F_{i_l}$ with $k<l$, then the thick path $(F_{i_0},\dots,
F_{i_k},F_{i_{l+1}},\dots, F_{i_{s+1}})$ is shorter than $\P$.

(4) If $F_{i_p}\cap F_{i_q}\ne\varnothing$ for $p<q$, then $q=p+1$. Else the
thick path $(F_{i_1},\dots, F_{i_p},F_{i_q},F_{i_{q+1}},\dots, F_{i_s})$ is
shorter than $\P$.

(5) For any vertex of $P$ the thick path $\P$ passes no more than two facets
containing this vertex. Indeed, if $F_{i_k}\cap F_{i_l}\cap F_{i_r}$ is a
vertex for some $k<l<r$, then $F_{i_k}\cap F_{i_r}\ne\varnothing$, which
contradicts (4).

(6) Consider a hexagon $F_{i_k}$. The edges $F_{i_{k-1}}\cap F_{i_k}$ and
$F_{i_{k+1}}\cap F_{i_k}$ can not be incident, therefore we either <<go
forward>>, if they are opposite, or <<turn left>> or <<turn right>>, if there
is an edge incident to them both.

(7) Two successive turns can not be in the same direction. Without loss of
generality assume that there are two successive turns left
(Fig.~\ref{turn}a): after $F_{i_p}$ and $F_{i_q}$, $q\geqslant p+1$.
\begin{figure}[h]
\begin{tabular}{cccc}
\includegraphics[scale=0.3]{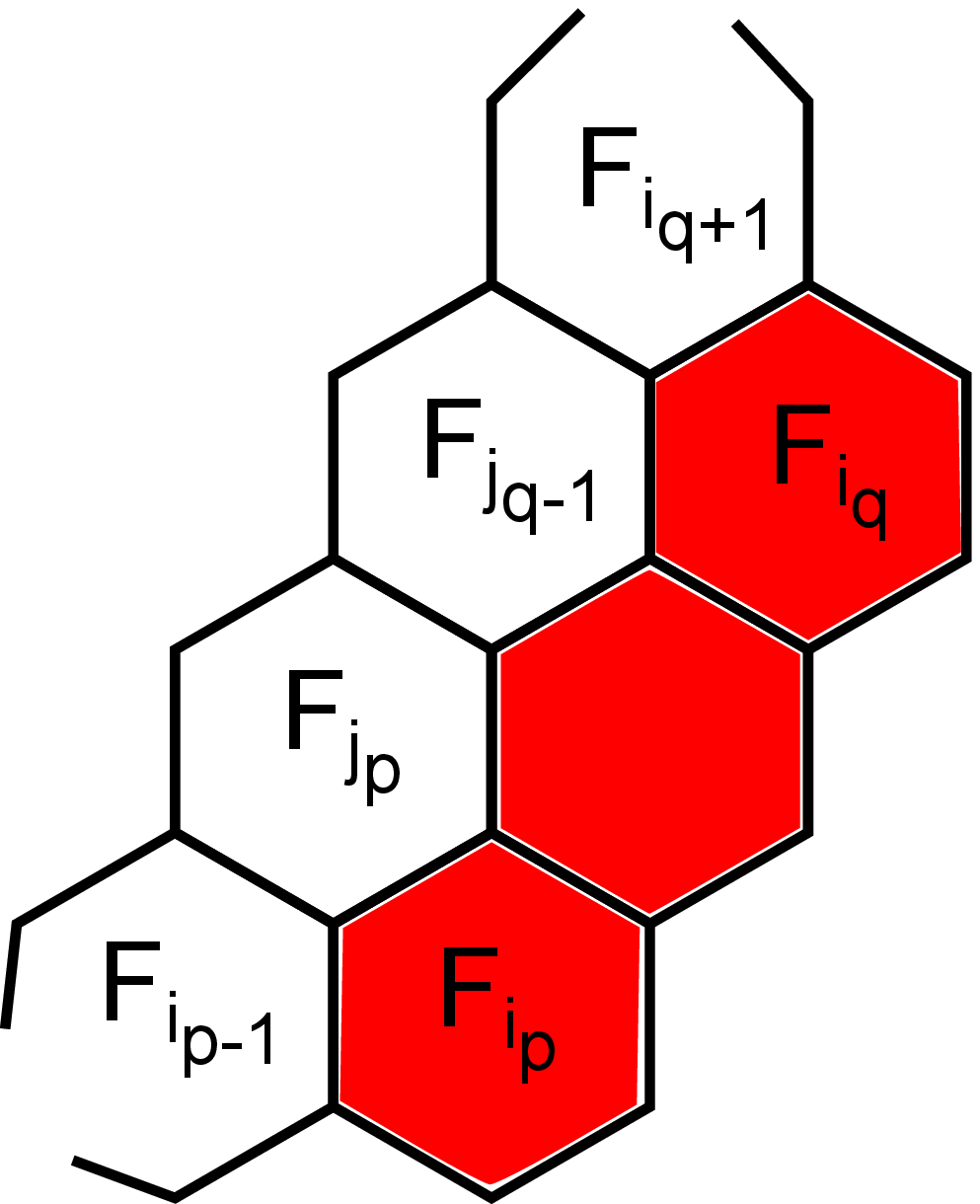}&\includegraphics[scale=0.3]{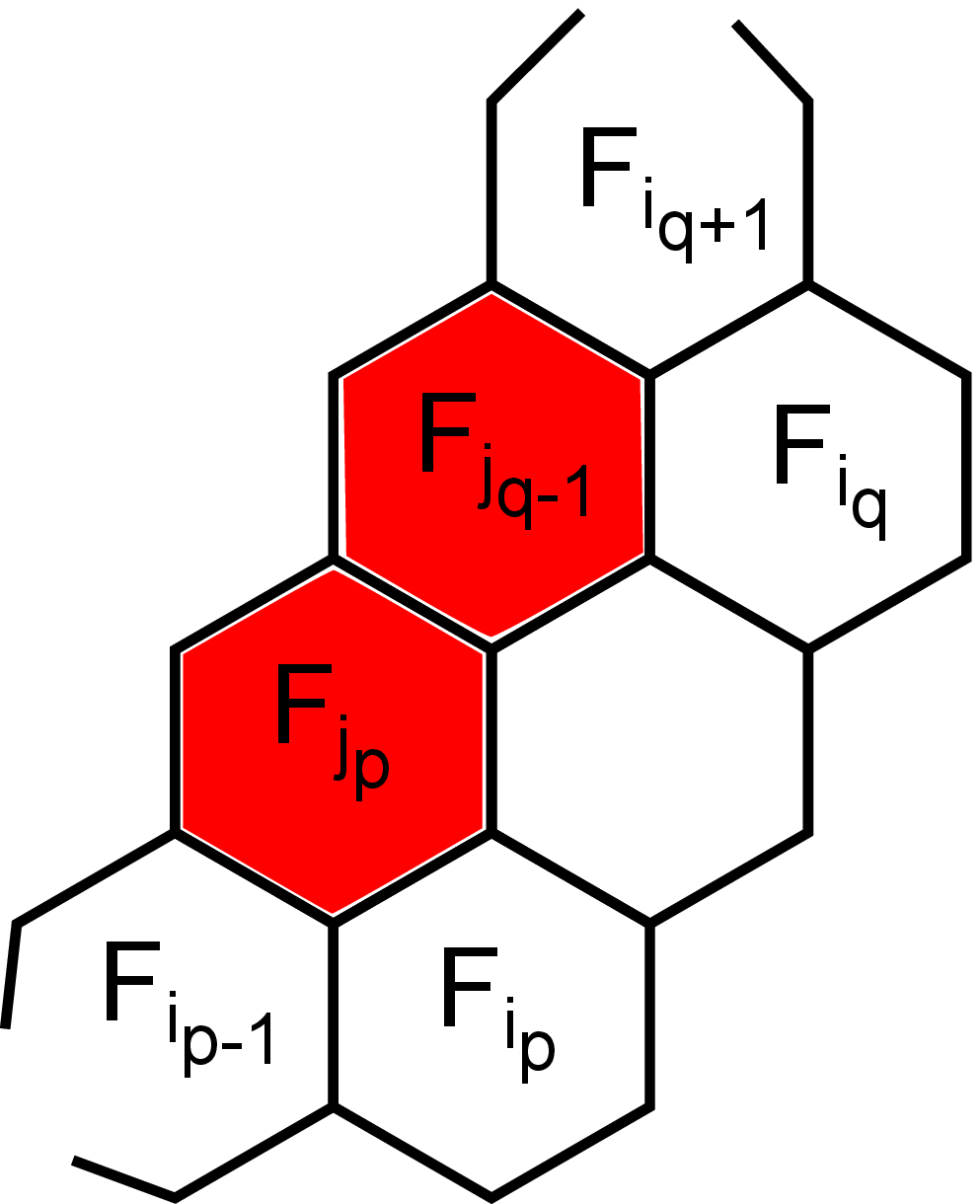}&\includegraphics[scale=0.3]{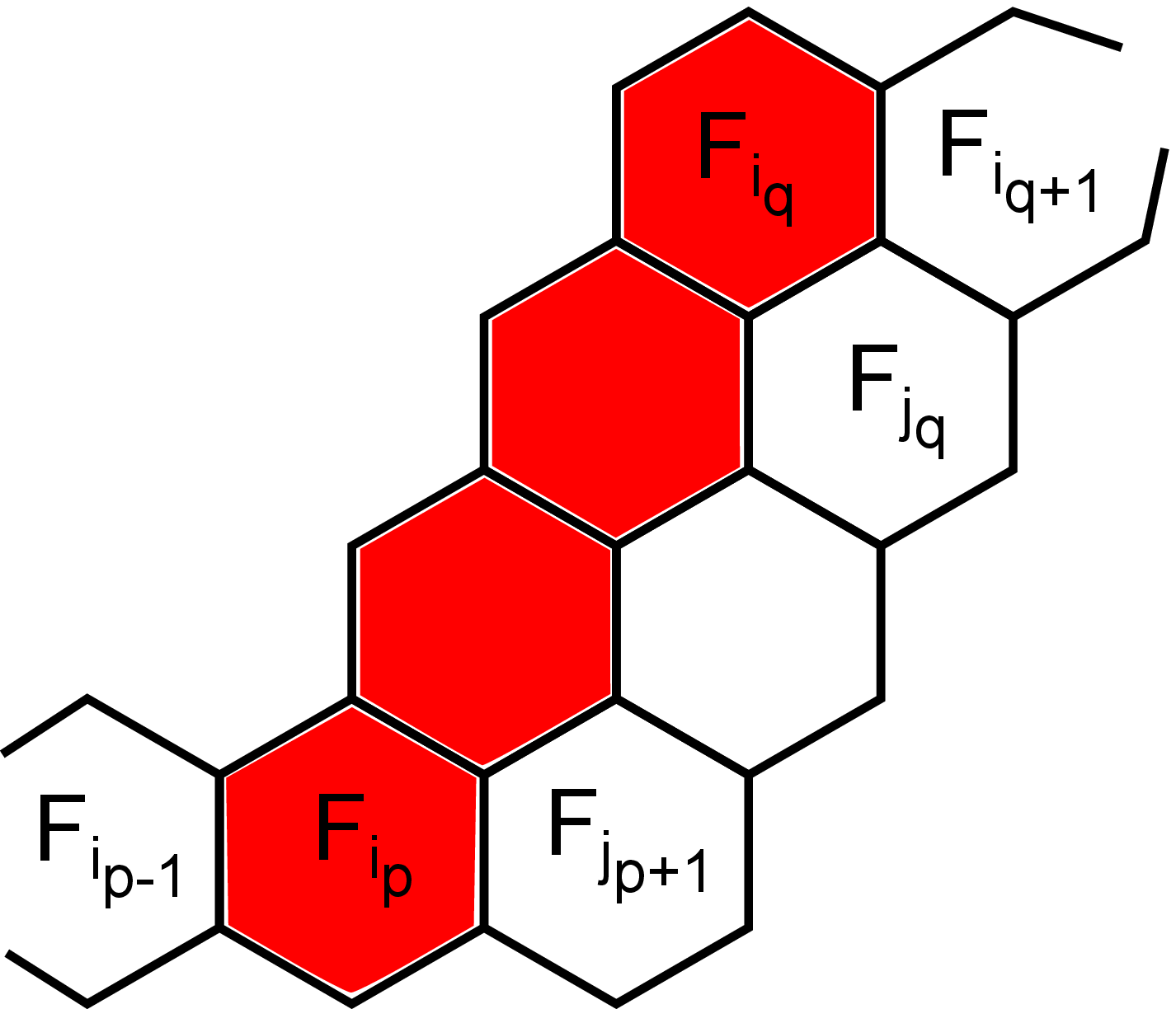}&\includegraphics[scale=0.3]{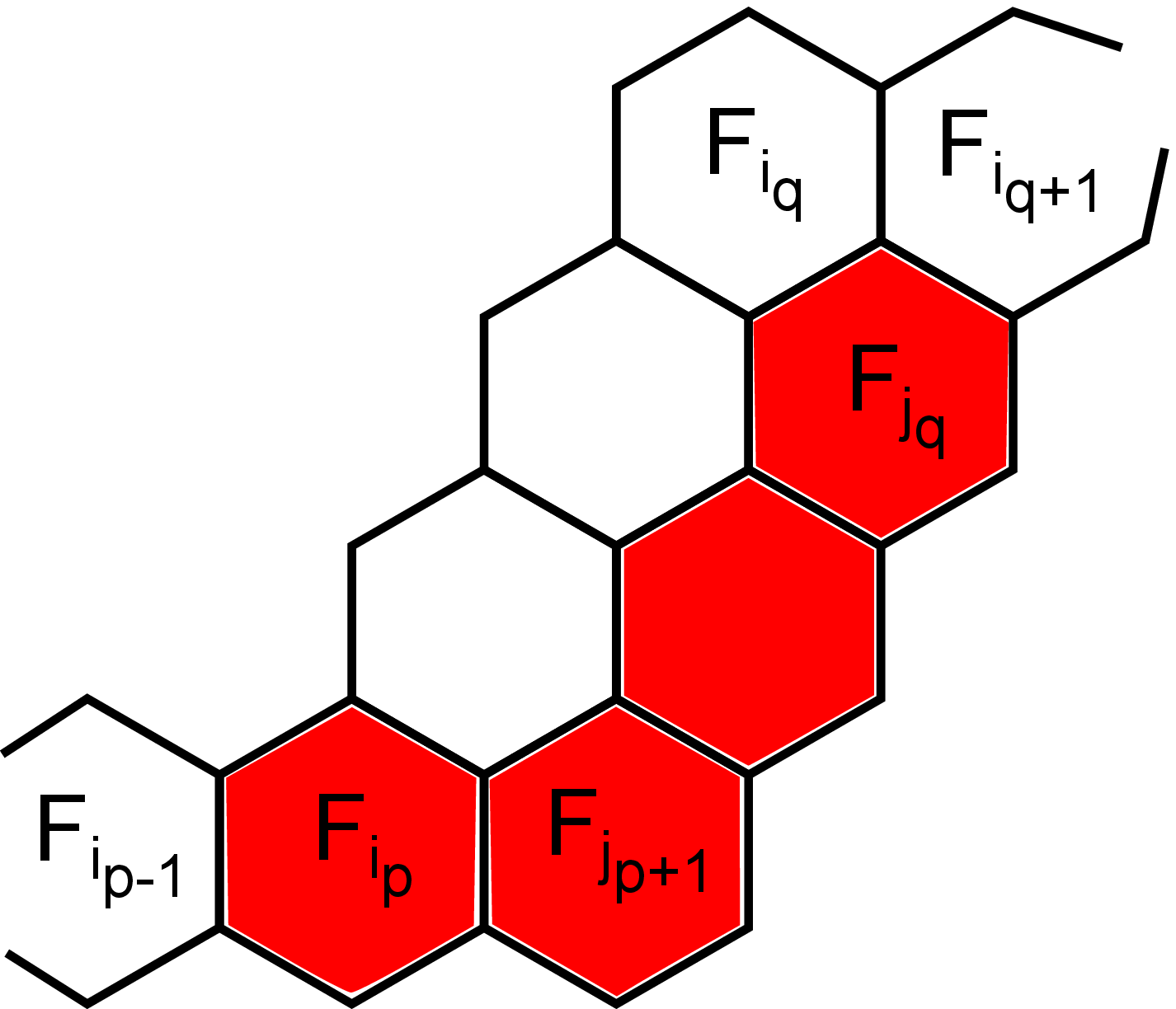}\\
a)&b)&c)&d)
\end{tabular}
\caption{}\label{turn}
\end{figure}
Then the <<direct segment>> $(F_{i_p},F_{i_{p+1}},\dots,
F_{i_{q-1}},F_{i_q})$  borders on the left the <<direct segment>>
$(F_{i_{p-1}},F_{j_p},\dots, F_{j_{q-1}}, F_{i_{q+1}})$ (Fig.~\ref{turn}b),
where $\{F_{j_p},\dots, F_{j_{q-1}}\}$ is some collection of hexagons by (2).
Then the thick path $(F_{i_0},\dots, F_{i_{p-1}},F_{j_p},\dots, F_{j_{q-1}},
F_{i_{q+1}},\dots, F_{i_{s+1}})$ is shorter than~$\P$.

(8) $\B$ can be chosen in such a way that there are no more than one
<<turn>>. Indeed, let there be two successive <<turns>>. Consider the first
two of them: after $F_p$, and after $F_q$, $q\geqslant p+1$. By (7) they
should be in opposite directions. Without loss of generality we can assume
that the first turn was left, and the second -- right (Fig.~\ref{turn}c).
Then on the right the <<direct segment>> $(F_{i_p},F_{i_{p+1}},\dots,
F_{i_{q-1}},F_{i_q})$ borders the <<direct segment>> $(F_{j_{p+1}},\dots,
F_{j_q},F_{i_{q+1}})$, where $\{F_{j_{p+1}},\dots, F_{j_q}\}$ are hexagons by
(2). Then the thick path $(F_{i_0},\dots, F_{i_p}, F_{j_{p+1}},\dots,
F_{j_q},F_{i_{q+1}},\dots, F_{i_{s+1}})$ has the same minimal length as $\B$,
hence it has all the properties (1)-(7). The new thick path has the first
turn later then $\B$. Applying this arguments to new paths we can move until
there are no more than one turn. In the end we obtain a minimal path $\B'$.
If there is no turns, then we obtain the fragment a). Else if the turn is
left, we obtain the fragment b). If the turn is right, then walking backward
we obtain the left turn.

To finish the proof we need to show that all the facets of the fragments a)
and b) we obtained are pairwise different. We do this by steps.

(a) $F_{i_0}$ and $F_{i_s}$ are the only pentagons on the fragments,
therefore they are different from other facets on the fragments.

(b) The facets of $\B'$ are pairwise different by (a) and (3).

\begin{figure}[h]
\begin{tabular}{cc}
\includegraphics[scale=0.27]{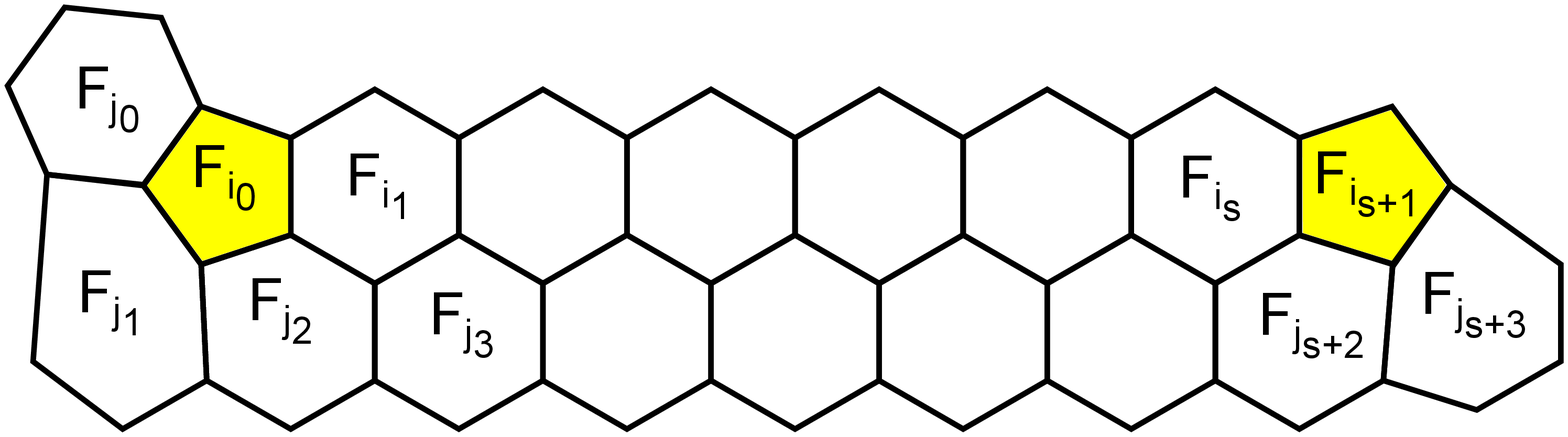}&\includegraphics[scale=0.27]{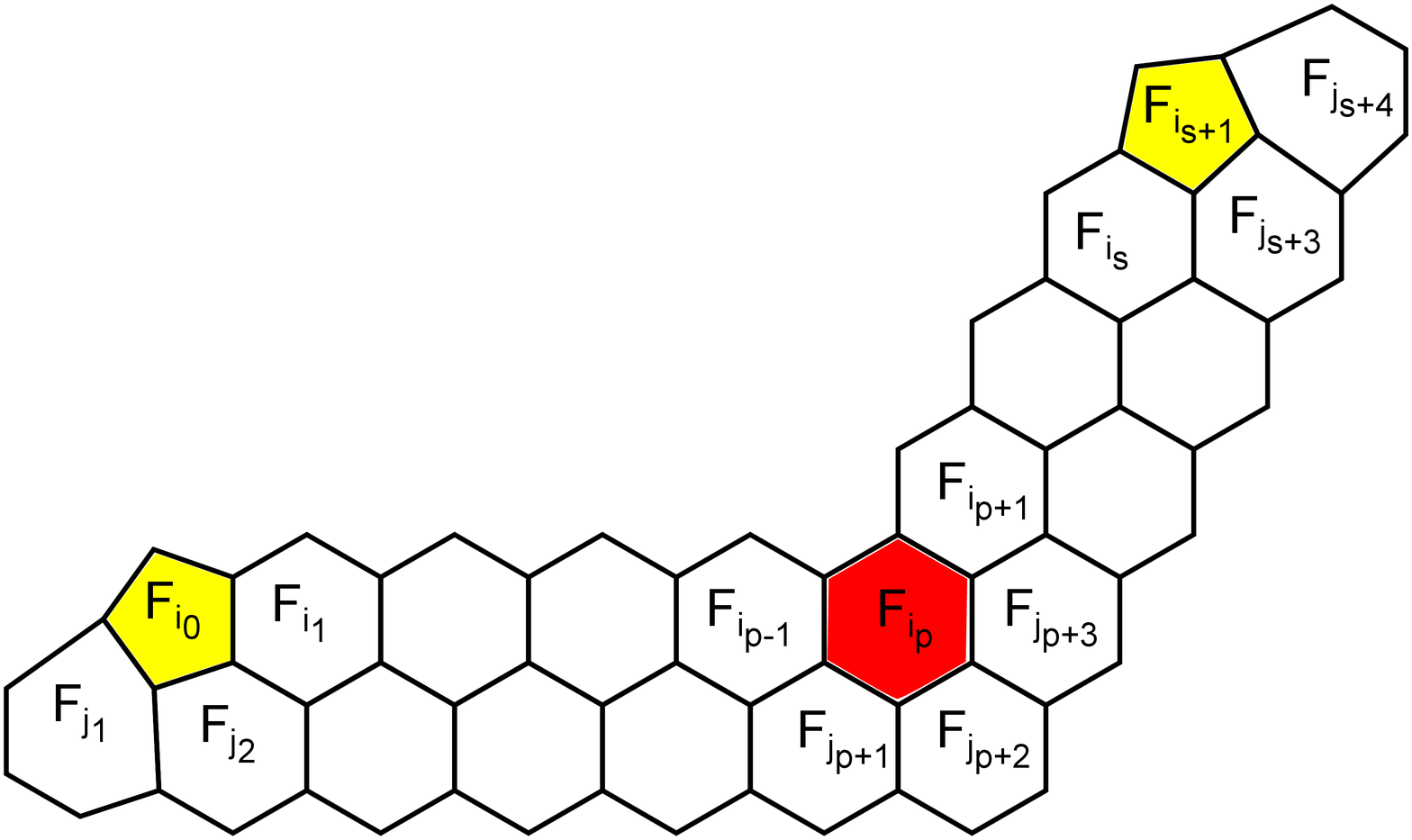}\\
a)&b)
\end{tabular}
\caption{}\label{Roads-F-L}
\end{figure}

(c) $\{F_{i_0},F_{i_1}\}$ and the facets surrounding them are pairwise
different by Lemma \ref{Flag-ij}. The same is true for
$\{F_{i_s},F_{i_{s+1}}\}$.

(d) By (c) if one of the facets $F_u$ surrounding a pentagon $F_{i_0}$ on the
picture coincides with some other facet, then this should be either
$F_{i_l}$, $3\leqslant l\leqslant s-1$, or  $F_{j_r}$, $4\leqslant r\leqslant
s+1$ on a) and $4\leqslant r\leqslant s+2$ on b) . If $F_u=F_{i_l}$, then the
thick path $(F_{i_0},F_{i_l},F_{i_{l+1}},\dots, F_{i_{s+1}})$ is shorter then
$\B'$. If $F_u=F_{j_r}$, then the thick path
$(F_{i_0},F_{j_r},F_{i_{r-1}},\dots, F_{i_{s+1}})$ is shorter than $\B'$ for
all $r$ on a) and for $r\leqslant {p+1}$ on b).

Consider the case b) when $F_u=F_{j_r}$, $p+2\leqslant r\leqslant s+2$. The
facets $F_{i_1}$, $F_{j_1}$, and $F_{j_4}$ belong to the facets surrounding
$F_{j_2}$ and $F_{j_3}$,  therefore $F_u\ne F_{j_4}$ and $r\geqslant 5$. Then
the thick path $(F_{i_0},F_{j_r},F_{i_{r-2}},\dots, F_{i_{s+1}})$ is shorter
than $\B'$.

Thus we have proved that any of the facets on the picture surrounding
$F_{i_0}$ differs from all other facets. By the symmetry the same is true for
the facets surrounding $F_{i_{s+1}}$.

(e) We have proved that if two facets coincide, then no more than one of them
has the form $F_{i_l}$,  $2\leqslant l\leqslant s-1$, and others have the
form $F_{j_r}$, where  $3\leqslant r\leqslant s+1$ for a) and  $3\leqslant
r\leqslant s+2$ for b). If $F_{i_l}=F_{j_r}$, then for all $r$ on a) and for
$r\leqslant p+1$ on b) we have $F_{i_l}\notin\{F_{i_{r-2}},F_{i_{r-1}}\}$,
and $F_{i_l}\cap F_{i_{r-2}}\cap F_{i_{r-1}}$ is a vertex, which is a
contradiction to (5). On b) by the symmetry $F_{i_l}\ne F_{j_r}$ for
$r\geqslant p+3$. On b) if $F_{i_l}=F_{j_{p+2}}$, then $F_{i_l}\cap
F_{i_p}\ne\varnothing$. Then $l=p+1$ or $p-1$ by (4). This is a
contradiction, since $F_{i_{p-1}}$, $F_{i_{p+1}}$ and $F_{j_{p+2}}$ are
pairwise different as surrounding the facet $F_{i_p}$.

(f) Now let $F_{j_r}=F_{j_t}$ for $3\leqslant r<t\leqslant s+1$ on a) or
$3\leqslant r<t\leqslant s+2$ on b). Then $t\geqslant r+3$, since the facets
$F_{j_r}$, $F_{j_{r+1}}$, and $F_{j_{r+2}}$ surround either two successive
facets $F_{i_l}$ and $F_{i_{l+1}}$, or one facet $F_{i_p}$, hence are
pairwise different.
Then for all $r,t$ on a) and for $t\leqslant p+1$ on b) the thick path
$(F_{i_0},\dots, F_{i_{r-2}},F_{j_r}=F_{j_t},F_{i_{t-1}},\dots, F_{i_s})$ is
shorter than $\B'$. The case $p+3\leqslant r$ is symmetric on b). Now the
case $r\leqslant p+2\leqslant t$ is left. If $r\leqslant p+1$, $p+2\leqslant
t$, then the thick path $(F_{i_0},\dots,
F_{i_{r-2}},F_{j_r}=F_{j_t},F_{i_{t-2}},\dots, F_{i_s})$ is shorter than
$\B'$. The case $r\leqslant p+2$, $p+3\leqslant t$ is symmetric. These two
cases cover the remained possibilities.

Thus we have proved that all the facets on fragments a) and b) are indeed
different.
\end{proof}

Now using Lemmas \ref{NIPR-lemma} and \ref{IPR-lemma} we obtain that any
fullerene contains one of the fragments that can appear after application of
one of the operation in (1).

\begin{lemma}\label{Inverse-lemma}
Let a combinatorial fullerene $P$ contain one of the fragments that appear
after application of operations in (1). Then $P$ indeed can be obtained from
some fullerene $Q$ by the corresponding operation.

Any of the operations in (1) are decomposed as a composition of $(1;4,5)$-,
$(1;5,5)$-, $(2,6;4,5)$-, $(2,6;5,5)$-, $(2,6;5,6)$-, $(2,7;5,5)$-, and
$(2,7;5,6)$-truncations.

All the intermediate polytopes have facets pentagons, hexagons, and no more
than one exceptional facet, which is either a quadrangle, or a heptagon.
\end{lemma}
\begin{proof}
On Fig.~\ref{OP-1-long}, \ref{OP-2-long}, \ref{OP-34-long}, \ref{OP-5-long},
\ref{OP-6-long}, and \ref{OP-7-long} for each fragment we find a sequence of
straightenings along thickened edges that transforms it to the corresponding
fragment before the application of the operation.  These operations are well
defined by Proposition \ref{utv-unfold}.

Yellow color depicts pentagons, red color -- hexagons, blue color --
quadrangles, and green color -- heptagons. We see that on each fragment on
each step there are at most one exceptional facet, which is either a
quadrangle or a heptagon. All other facets of the polytope are pentagons and
hexagons.

Each thickened edge that belongs to a quadrangle intersects by vertices
either a pentagon and a hexagon, or two hexagons. The inverse operation are
$(1;4,5)$- and $(1;5,5)$-truncations.

Each thickened edge that belongs to two pentagons intersects by vertices
either a pentagon and a hexagon, or two hexagons, or a hexagon and a
heptagon. The inverse operations are $(2,6;4,5)$-, $(2,6;5,5)$-, and
$(2,6;5,6)$-truncations.

Each thickened edge that belongs to a pentagon and a hexagon intersects by
vertices either two hexagons, or a hexagon and a heptagon. The inverse
operation are $(2,7;5,5)$- and $(2,7;5,6)$-truncations.
\begin{figure}[h]
\begin{center}
\includegraphics[height=2.6cm]{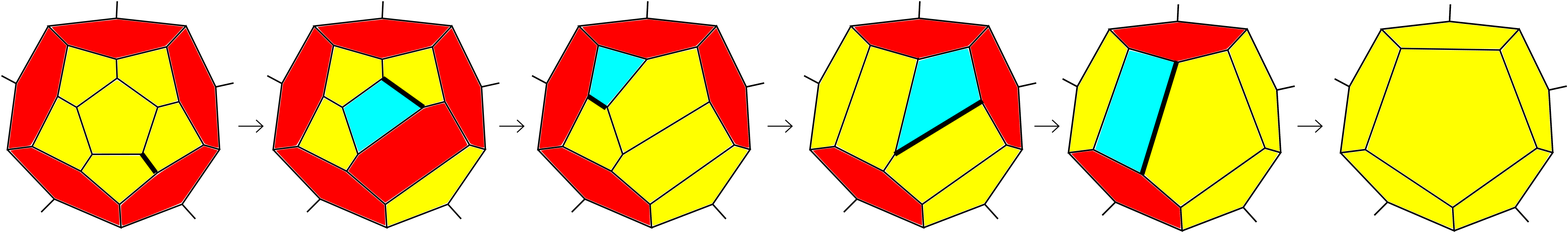}
\end{center}
\caption{}\label{OP-1-long}
\end{figure}

\begin{figure}[h]
\begin{center}
\includegraphics[height=3cm]{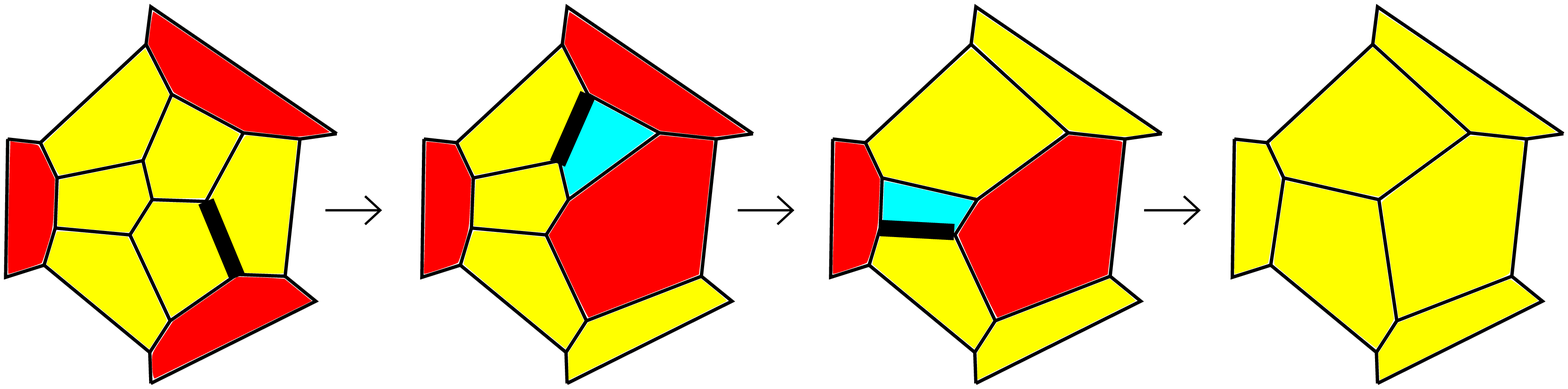}
\end{center}
\caption{}\label{OP-2-long}
\end{figure}

\begin{figure}[h]
\begin{tabular}{ccc}
\includegraphics[height=2.4cm]{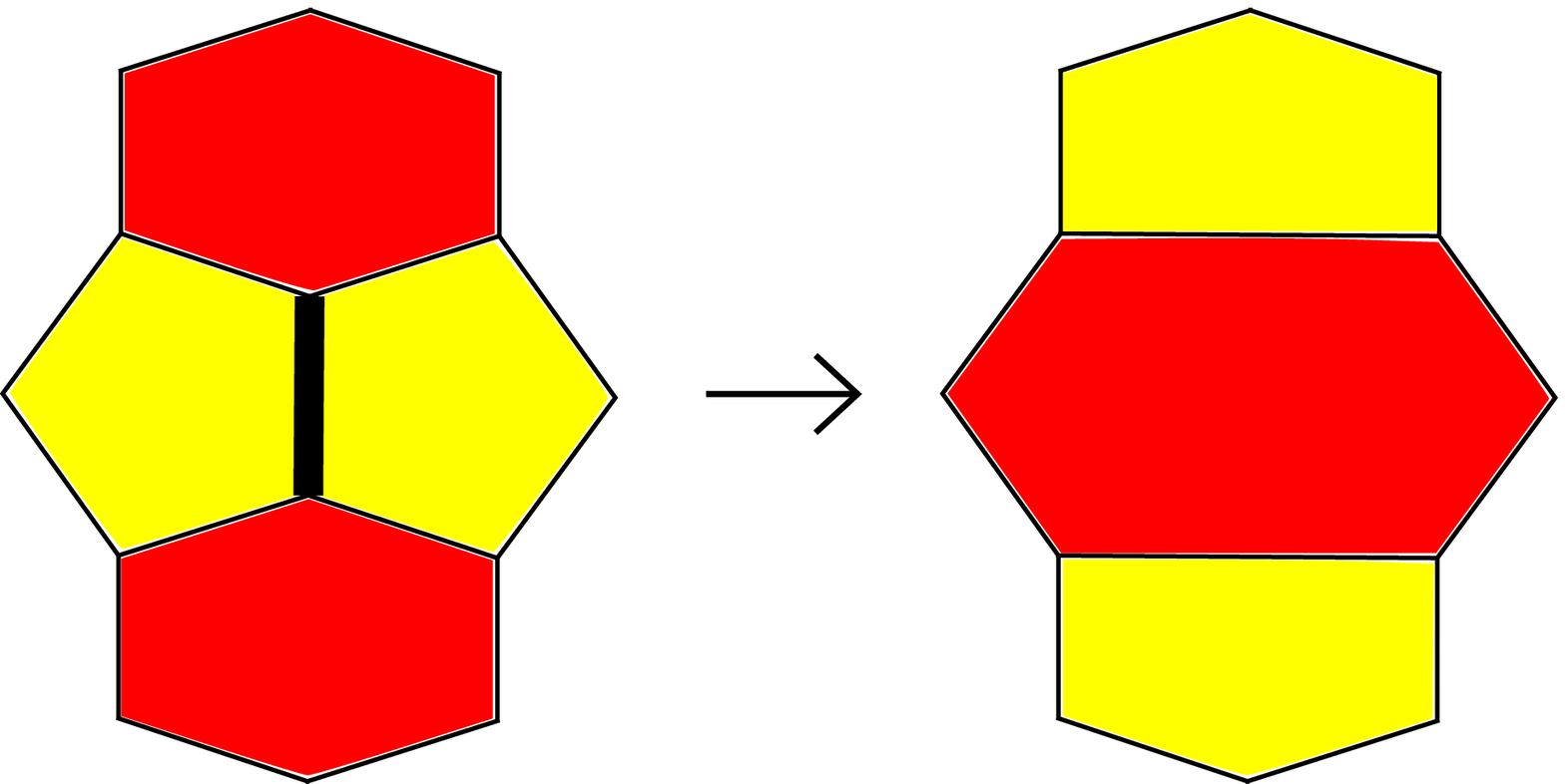}&&\includegraphics[height=2.4cm]{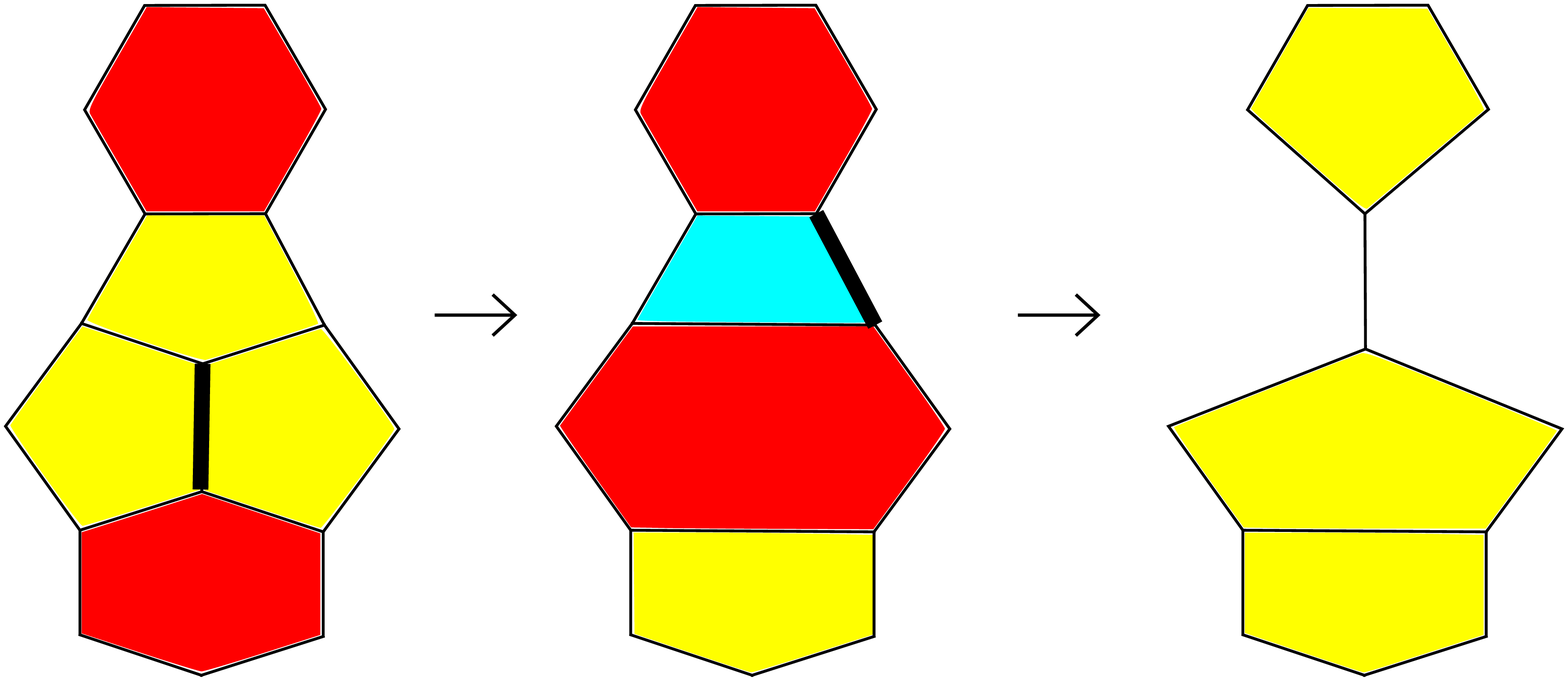}\\
a)&&b)
\end{tabular}
\caption{}\label{OP-34-long}
\end{figure}

\begin{figure}[h]
\begin{center}
\includegraphics[height=2.4cm]{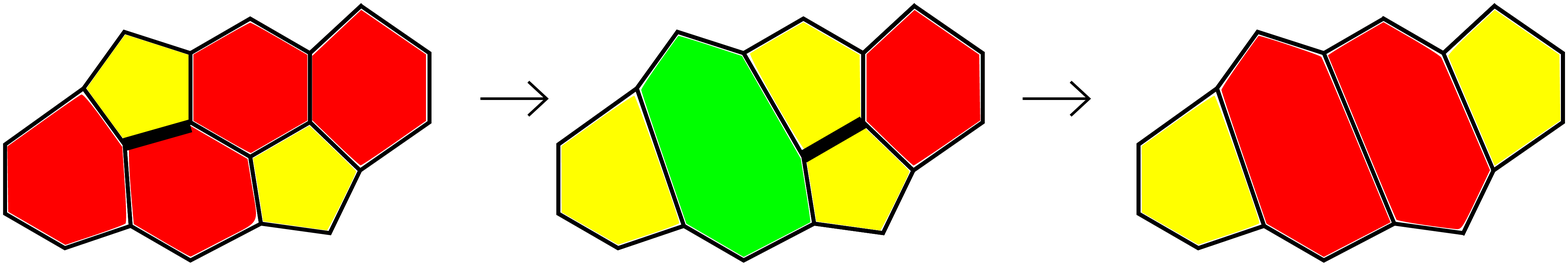}\\
\end{center}
\caption{}\label{OP-5-long}
\end{figure}

\begin{figure}[h]
\begin{center}
\includegraphics[scale=0.15]{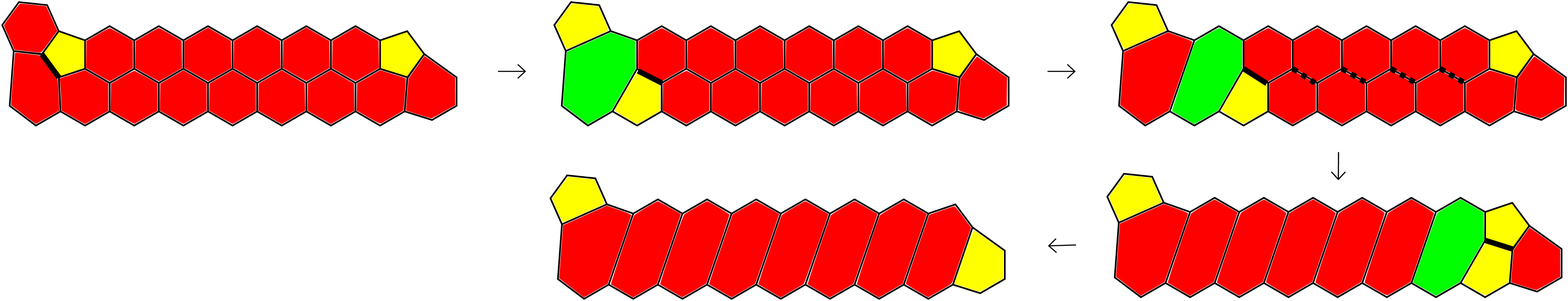}
\end{center}
\caption{}\label{OP-6-long}
\end{figure}

\begin{figure}[h]
\begin{center}
\includegraphics[scale=0.13]{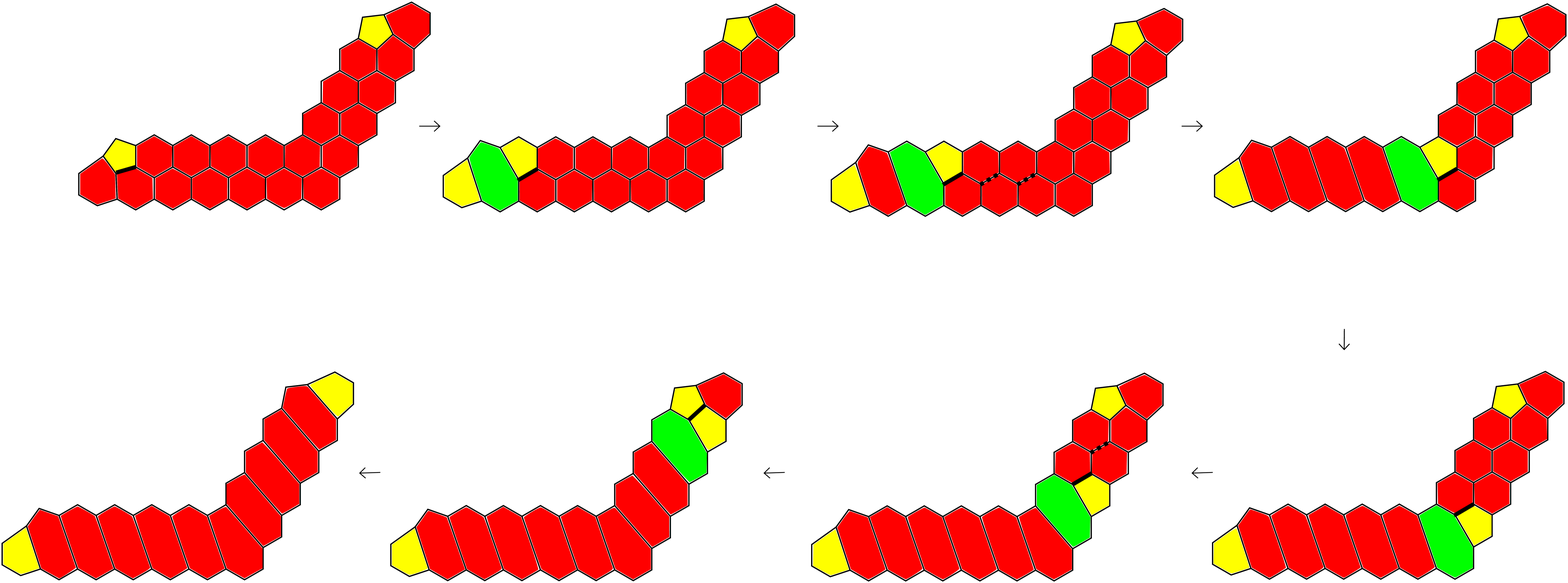}
\end{center}
\caption{}\label{OP-7-long}
\end{figure}
\end{proof}
Now the theorem follows directly from Lemma \ref{Inverse-lemma}. The first
two operations are movements from the $k$-th to the $(k+1)$-th fullerene in
the Families I and II by Theorems \ref{5-belt-theorem} and \ref{131-theorem}.
\end{proof}


\end{document}